\documentclass[10pt]{amsart}

\usepackage{amsmath,amsfonts,amsthm,graphicx}

\title[$C^1$ density of stable ergodicity]{$C^1$ density of stable ergodicity}
\author[A. Avila, S. Crovisier, A. Wilkinson]{A. Avila, S. Crovisier, and A. Wilkinson}
\date{\today}

\thanks{A.A. was partially supported by the ERC Starting Grant Quasiperiodic
and S.C. by the ERC Advanced Grant NUHGD.
A.A. and S.C. were partially supported by the Balzan Research Project of J.Palis.
A.W. was supported by NSF grant DMS-1316534.
}

\address{Artur Avila \newline
\rm CNRS, IMJ-PRG, UMR 7586, Univ Paris Diderot, Sorbonne Paris Cit\'e,
Sorbonnes Universit\'es, UPMC Univ Paris 06, F-75013, Paris, France \&
IMPA, Estrada Dona Castorina 110, Rio de Janeiro, Brasil}
\address{Sylvain Crovisier \newline
\rm CNRS - Laboratoire de Math{\'e}matiques d'Orsay, UMR 8628\newline
Universit{\'e} Paris-Sud 11, 91405 Orsay Cedex, France.}
\address{Amie Wilkinson \newline
\rm Department of Mathematics,\newline
University of Chicago, 5734 S. University Avenue Chicago, Illinois 60637, USA.}

\theoremstyle{plain}
\newtheorem{theorem}{Theorem}[section]
\newtheorem{lemma}[theorem]{Lemma}
\newtheorem{sublemma}[theorem]{Sublemma}
\newtheorem{corollary}[theorem]{Corollary}
\newtheorem{proposition}[theorem]{Proposition}

\newtheorem{conjecture}{Conjecture}

\newtheorem{quest}[conjecture]{Question}

\newtheorem*{corollary*}{Corollary}
\newtheorem*{claim}{Claim}

\newtheoremstyle{vThm*}%
{}{}%
{\itshape}%
{-3pt}{\bfseries}%
{}{ }%
{\thmnote{#3}}%
\theoremstyle{vThm*}
\newtheorem*{nThm*}{}

\theoremstyle{definition}
\newtheorem{definition}[theorem]{Definition}

\def\length{\mathrm{Length}}
\def\diam{\mathrm{diam}}

\def\Diff{\operatorname{Diff} }
\def\diff{\operatorname{Diff} }

\def\mod{\hbox{mod} }

\def\eproof{$\Box$ \medskip}

\def\title{\em}

\def\bar{\overline}

\def\id{\operatorname{Id}}
\def\phcs{{H^s_{\hbox{\tiny Pes}}}}
\def\phcu{{H^u_{\hbox{\tiny Pes}}}}
\def\phc{{H_{\hbox{\tiny Pes}}}}

\def\Jac{\operatorname{Jac}}

\def\Nuh{\operatorname{Nuh}}

\def\GL{\operatorname{GL}}
\def\DS{\mathcal{DS}}

\def\cW{\mathcal{W}}

\def\cF{\mathcal{F}}
\def\cL{\mathcal{L}}
\def\cU{\mathcal{U}}

\def\cN{\mathcal{N}}
\def\cG{\mathcal{G}}

\def\cC{\mathcal{C}}
\def\cD{\mathcal{D}}
\def\cQ{\mathcal{Q}}
\def\cO{\mathcal{O}}

\def\cV{\mathcal{V}}
\def\cW{\mathcal{W}}

\def\vol{m}
\def\cH{\mathcal{H}}

\def\cP{\mathcal{P}}

\def\diag{\operatorname{diag}}

\def\d{{\underline {j}}}

\def\transverse{\,\raise2pt\hbox to1em{\hfil$\top$\hfil}\hskip -1em \hbox
to1em{\hfil$\cap$\hfil}\,} 
\def\be{\begin{equation}}
\def\ee{\end{equation}}

\def\cA{\mathcal{A}}
\def\cG{\mathcal{G}}
\newcommand\R{\mathbb R}
\newcommand\RR{{\mathbb R}}

\newcommand\ZZ{{\mathbb Z}}
\newcommand\Z{{\mathbb Z}}
\newcommand\NN{{\mathbb N}}

\newlength{\figboxwidth} 
\def\eproof{\hfill$\Box$ \bigskip}

\makeatletter
\let\original@addcontentsline\addcontentsline
\newcommand{\dummy@addcontentsline}[3]{}
\newcommand{\DeactivateToc}{\let\addcontentsline\dummy@addcontentsline}
\newcommand{\ActivateToc}{\let\addcontentsline\original@addcontentsline}
\makeatother

\begin{document}

\begin{abstract}

We prove a $C^1$ version of a conjecture by Pugh and Shub:
among partially hyperbolic volume-preserving $C^r$ diffeomorphisms, $r>1$,
the stably ergodic ones are $C^1$-dense.

To establish these results, we develop new perturbation tools for the $C^1$ topology:
linearization of horseshoes while preserving entropy, and creation of ``superblenders" from hyperbolic sets
with large entropy.

\end{abstract}

\maketitle



\addcontentsline{toc}{section}{\mbox{}\quad\quad Introduction}
\DeactivateToc
\section{Introduction}

A volume-preserving, $C^2$ diffeomorphism $f\colon M\to M$ of a compact Riemannian manifold $M$ is {\em stably
ergodic} if any volume preserving, $C^2$  diffeomorphism $g\colon M\to M$ that is sufficiently close to $f$ in the $C^1$ topology is ergodic with respect to volume.

Interest in stably ergodic dynamical systems dates back at least to 1954 when Kolmogorov in his ICM address posited the existence of stably ergodic flows \cite{Ko}.    The first natural class of stably ergodic diffeomorphisms was established by Anosov in the 1960's.  These so-called Anosov diffeomorphisms remained the only known examples of stably ergodic diffeomorphisms for nearly 30 more years.

Grayson, Pugh and Shub \cite {GPS} established the existence of
non-Anosov stably ergodic diffeomorphisms in 1995.  
They considered the time-one map of the geodesic flow for a surface
of constant negative curvature.  These examples belong to a class of 
dynamical systems known as the partially hyperbolic diffeomorphisms.  

A diffeomorphism $f\colon M\to M$ is {\em partially hyperbolic} if there is a continuous splitting of the tangent bundle $TM = E^u\oplus E^c\oplus E^s$, invariant under the derivative $Df$, such that vectors in $E^u$ are uniformly expanded in length by $Df$, vectors in $E^s$ are uniformly contracted, and the expansion and contraction of the length of vectors in $E^c$ is dominated by the expansion and contraction in $E^u$ and $E^s$, respectively.  See Section~\ref{ss=domsplit} below for a precise definition.

Based on \cite{GPS} and related work, Pugh and Shub
formulated in 1996 a bold conjecture about partial hyperbolicity and stable ergodicity \cite {PS}:

\begin{nThm*}{\bf{Stable Ergodicity Conjecture.}} Stable 
ergodicity is $C^r$-dense among the $C^r$ partially hyperbolic,
volume-preserving  diffeomorphisms  on a compact connected manifold, for any $r>1$.
\end{nThm*}

Restricting to the class of partially hyperbolic diffeomorphisms for which the center bundle $E^c$ is one-dimensional, this conjecture
was proved by F. Rodriguez-Hertz,
M.A. Rodriguez-Hertz and Ures \cite {RRU}.
Here we will establish, in general, a
$C^1$ version of the Stable Ergodicity Conjecture:

\begin{nThm*}{\bf{Theorem A.}}
Stable ergodicity is $C^1$-dense
in the space of $C^r$ partially
hyperbolic volume-preserving diffeomorphisms on a compact connected manifold,
for any $r > 1$.
\end{nThm*}
Theorem A has a more precise formulation, Theorem A', stated in Section~\ref{ss=blenders}.
Among the class of partially hyperbolic diffeomorphisms where $E^c$ has dimension one or two,
Theorem A has been established earlier in \cite {BMVW} and \cite {RRTU}.

\section{Some background and structure of the proof}

We fix notation to be used throughout the paper.  Let  $M$ be a compact, connected boundaryless
manifold with a fixed Riemannian metric, and denote by $m$ the volume in this metric, normalized so that $m(M)=1$.
Sometimes we will also consider a symplectic structure $\omega$
and its associated normalized volume $m$.
The spaces of $C^r$, volume preserving and symplectic diffeomorphisms
are denoted by $\Diff^r_m(M)$ and $\Diff^r_\omega(M)$, respectively.

\subsection{The Hopf argument and stable ergodicity}

One of the earliest arguments for proving ergodicity, still in use today, was originally employed by Eberhard Hopf
in the 1930's  to prove ergodicity of the geodesic flow for closed, negatively curved surfaces.  These flows (``Anosov flows" in current terminogy) have one-dimensional invariant expanded and contracted distributions $E^u$ and $E^s$, tangent to invariant foliations $\cW^u$ and $\cW^s$, respectively. Hopf's ergodicity argument uses the Ergodic Theorem to show that any invariant set for the flow must consist of essentially whole leaves of both $\cW^u$ and $\cW^s$.  Invariance under the flow implies that the same is true for the foliations $\cW^{cu}$ and $\cW^{cs}$ formed by flowing the leaves of  $\cW^u$ and $\cW^s$, respectively.  Since the leaves of these foliations are transverse, a version of Fubini's theorem implies that every invariant set for the flow must have full measure in neighborhoods of fixed, uniform size in the manifold.  Ergodicity follows from connectedness and this {\em local ergodicity}.  The version of Fubini's theorem employed by Hopf is fairly straightforward, since in his setting, the foliations $\cW^{cu}$ and $\cW^{cs}$ (and indeed,   $\cW^{u}$ and $\cW^{s}$) are $C^1$.

Hopf himself foresaw the general usefulness of his methods beyond the geometric context of geodesic flows.  In 1940 he wrote:  ``The range of applicability of the method of asymptotic geodesics extends far beyond surfaces of negative curvature. In~\cite{H1},  $n$-dimensional manifolds of negative curvature were already investigated. But the method allows itself to be applied to much more general variational problems with an independent variable, aided by the Finsler geometry of the problems. This points to a wide field of problems in differential equations in which it will now be possible to determine the complete course of the solutions in the sense of the inexactly measuring observer."\cite{H2}

Kolmogorov, too, saw the potential of the Hopf argument as a general method to prove ergodicity.
  In his 1954 ICM address \cite {Ko},  he  wrote:
 ``it is extremely likely that, for arbitrary $k$, there are examples of
canonical systems with $k$  degrees of freedom and with stable
transitiveness [i.e. ergodicity] and mixing... I have in mind motion along
geodesics on compact manifolds of constant negative curvature... "
  Kolmogorov's intuition was clearly guided by the 
robust nature of Hopf's ergodicity proof; indeed, inspected carefully with a modern eye, Hopf's original approach gives a
complete proof of Kolmogorov's assertion that the geodesic flow for a hyperbolic manifold remains ergodic when perturbed within the class of Hamiltonian (or even volume-preserving) flows.\footnote{For flows that are perturbations of constant negative curvature geodesic flows, the foliations $\cW^{cu}$ and $\cW^{cs}$ are $C^1$, which is enough to carry out the Fubini step. The same regularity of the foliations does not for geodesic flows on arbitrary negatively curved manifolds.}

Around ten years later, 
Anosov \cite {Anosov}  generalized Hopf's theorem to closed manifolds 
of strictly negative (but far from constant) sectional curvatures, in any dimension. 
The key advance was to
extend the Fubini part of Hopf's argument when the foliations $\cW^u$ and $\cW^s$ are not $C^1$ but still satisfy an absolute continuity property. This also gave the ergodicity of $C^2$, volume-preserving uniformly hyperbolic diffeomorphisms,  now known as the {\em Anosov diffeomorphisms}.  These are the diffeomorphisms  $f\colon M\to M$ for which there exists a $Df$-invariant splitting $TM = E^u\oplus E^s$ and an integer $N\geq 1$ such that for every unit vector $v\in M$:
\begin{equation}
\label{uniform-bundle}
v\in E^u \implies \|Df^N v\| > 2, \quad\hbox{and } v\in E^s \implies \|Df^N v\| < 1/2.
\end{equation}

Since Anosov diffeomorphisms form a $C^1$-open class, and $C^2$, volume-preserving Anosov diffeomorphisms are ergodic, it follows that Anosov diffeomorphisms are stably ergodic.  The general expectations of Hopf and Kolmogorov were thus met in Anosov's work.  But there is more to the story: while Anosov diffeomorphisms gave the first examples of stably ergodic systems, their existence raised the question of what other examples there might be.

\subsection{Dominated splittings and partial hyperbolicity: the Pugh-Shub Conjectures}\label{ss=domsplit}

Since stable ergodicity is by definition a robust property, it is natural
to search for an alternate robust, geometric/topological dynamical
characterization of the stably ergodic diffeomorphisms.   Bonatti-Diaz-Pujals \cite{BDP} proved a key result which
can be slightly modified to show that stable ergodicity of $f\colon M\to M$ implies the existence of a {\em dominated splitting}, that is, a $Df$-invariant splitting $TM = E_1 \oplus E_2 \oplus \ldots \oplus E_k$, $k\geq 2$  and an integer $N\geq 1$ such that for every unit vectors $v,w\in M$:
\[v\in E_i, \,w\in E_j \text{ with } i<j \implies \|Df^N v\| > 2 \|Df^N w\| .
\]

The stably ergodic diffeomorphisms built by Grayson, Pugh and Shub have a dominated splitting of a special type,
incorporating features of Anosov diffeomorphisms: they are \emph{partially hyperbolic}, meaning there exists
a dominated splitting $TM=E^u\oplus E^c\oplus E^s$, with both $E^u$ and $E^s$ nontrivial, and an integer $N\geq 1$ such that \eqref{uniform-bundle}
holds for any unit vector $v$. The bundles $E^u$, $E^c$ and $E^s$ are called the \emph{unstable, center} and \emph{stable bundles}, respectively.  Anosov diffeomorphisms are the partially hyperbolic diffeomorphisms for which $E^c$ is trivial.

 As with Anosov diffeomorphisms and flows, the unstable bundle $E^u$ and the
stable bundle $E^s$ of a partially hyperbolic diffeomorphism are uniquely
integrable, tangent to the unstable and
stable foliations, $\cW^u$ and $\cW^s$, respectively.  However, they do not span $TM$, and thus are not transverse (except in the Anosov case).  The Hopf argument for local ergodicity may fail without further assumptions: there do exist non-ergodic partially hyperbolic diffeomorphisms.

As a substitute for transversality \cite{GPS} introduced an additional assumption, a property called ``accessibility,''
which is a term borrowed from the control theory literature.\footnote{Brin and Pesin
had earlier studied the accessibility property and its stability
properties in the 1970's and used it to prove topological properties
of conservative partially hyperbolic systems.}
A partially
hyperbolic diffeomorphism is {\em accessible} if any two points
in the manifold can be connected by a continuous path that is
piecewise contained in leaves of $\cW^u$ and $\cW^s$.
Then we say that $f$ is \emph{stably accessible} if any volume preserving $C^1$ perturbation of $f$ is 
accessible.

The route suggested by Pugh-Shub to prove the Stable Ergodicity Conjecture
was to establish two other conjectural statements about the space of volume-preserving partially hyperbolic diffeomorphisms: 1)
stable accessibility is dense, and 2)
accessibility implies ergodicity.

The $C^r$-density of stable accessibility remains open, except in the case of
one-dimensional center bundle \cite {RRU}.
However, in the $C^1$ case we are concerned with, it was established by
Dolgopyat and Wilkinson \cite {DW} in 2003.

The connection between accessibility and ergodicity is certainly reasonable at the heuristic level:
The Hopf argument can be carried out in the partially hyperbolic setting,
and the ergodic (and even mixing) properties of the system
can be reduced to the ergodic properties of the measurable
equivalence relation generated by the pair of foliations
$(\cW^u, \cW^s)$.  However, in trying to show that these ergodic properties
follow from the assumption of accessibility, one quickly encounters substantial
issues involving sets of measure zero and delicate measure-theoretic
and geometric properties of the foliations. 

Currently these issues can be
resolved  under an additional assumption called \emph{center bunching}
(which constrains the nonconformality of the dynamics along the
center bundle $E^c$), a result due to Burns and Wilkinson \cite {BW}.
In all likelihood, a significant new idea is needed
to make further progress. While including a broad range of examples, the center bunched diffeomorphisms are  not dense
 among partially hyperbolic diffeomorphisms, except among those with
one-dimensional center bundle, where center bunching always holds.

\subsection{Local ergodicity for partially hyperbolic systems: the RRTU argument}

While we don't know how to derive ergodicity from accessibility alone, a
relatively simple argument (due to Brin)
allows one to conclude that accessibility
implies {\em metric transitivity:}  the property that almost every orbit is
dense.  This motivates the idea of combining accessibility with some
``local'' ergodicity mechanism in order to achieve full ergodicity.  The
approach we will follow, due to \cite {RRTU}, is in a certain sense closer
to the original applications of the Hopf argument
than the Pugh-Shub approach, but with several crucial
new ingredients added.

This modified Hopf argument uses a measurable version of stable and unstable foliations  for a nonuniformly hyperbolic set,  called the Pesin unstable and stable disk families.  Nonuniform hyperbolicity is defined using quantities called Lyapunov exponents.
A real number $\chi$ is a \emph{Lyapunov exponent} of a diffeomorphism $f: M\to M$ at $x\in M$ if there exists a nonzero vector
$v\in T_xM$ such that
\begin{equation}\label{e=lyaplim}
\lim_{n\to\infty} \frac{1}{n} \log \|Df^n(v)\| = \chi.
\end{equation}
If $f$ preserves the volume $m$, then Kingman's ergodic theorem implies that the limit 
$\chi=\chi(x,v)$ in \eqref{e=lyaplim}
exists for  for $m$-almost every $x\in M$ and every nonzero $v\in T_xM$. 
Furthermore, Oseledets's theorem gives that $\chi(x,\cdot)$ can assume at most $\dim(M)$ distinct values $\chi_1(x), \ldots, \chi_{\ell(x)} (x)$, and that there exists a {\em measurable}, $Df$-invariant splitting of the tangent bundle (defined over a full $m$-measure subset of $M$)
\begin{equation}\label{e=oseledec}
T M = E_1\oplus E_2\oplus \cdots \oplus E_\ell, 
\end{equation}
such that $\chi(x,v) = \chi_i(x)$ for $v\in E_i(x)\setminus \{0\}$.  The set of $x\in M$ where these conditions are satisfied is called the set of {\em Oseledec regular points (for $m$ and $f$)}.  
Volume preservation implies that the sum of the Lyapunov exponents is zero .

By separately summing the Lyapunov subspaces $E_i$ corresponding to $\chi_i > 0$ and $\chi_i < 0$ we obtain a measurable splitting  $TM = E^+\oplus E^0\oplus E^-$, where $E^0$ is the Lyapunov subspace
(possibly trivial) for the exponent $0$.
We say that a positive $m$-measure invariant set $X\subset M$ (not necessarily compact) is {\em nonuniformly hyperbolic} if for almost every $x\in X$, the exponents $\chi_i(x)$ exist and are nonzero; that is, $E^0$ is trivial over $X$.

Nonuniform hyperbolicity is a natural property for studying  ergodicity:
Bochi-Fayad-Pujals~\cite{BFP} have shown that among stably ergodic diffeomorphisms there exists a $C^1$-dense and open subset of systems
that  are \emph{nonuniformly Anosov}: they have a dominated splitting $TM=E\oplus F$ and at $m$-almost every point $x$
we have $E^+=E$ and $E^-=F$.

Tangent to $E^+$ and $E^-$ for a nonuniformly hyperbolic set are two measurable families of invariant, smooth disks called the Pesin unstable and stable disk families, respectively.
The dimension and diameter of these Pesin stable disks vary measurably over the manifold.
If we want to
implement the Hopf argument as Hopf and Anosov did, the main  issue is to
ensure the transverse intersection of stable and unstable disks. 
Several difficulties arise:
\begin{enumerate}
\item\label{difficulty1} The Pesin disks should have ``enough dimension'' so that transversality is
even possible.\footnote {Note that in the Pugh-Shub approach,
the Hopf argument is applied
in a context where the dimensions of the strong stable and unstable
manifolds are
not enough to allow for transversality, but the analysis is much more
involved than the simple implementation we are discussing, and it
ends up depending on the center bunching condition.}
\item\label{difficulty2} The spaces $E^+$ and $E^-$ along which they align should display some
definite transversality (i.e. uniform boundedness of angle between them).
\item\label{difficulty3} The stable and unstable disks should be ``long enough'' so that
they have the opportunity to intersect.
\end{enumerate}

It turns out that for an accessible, partially hyperbolic diffeomorphism, difficulties (\ref{difficulty1}) and (\ref{difficulty2}) can be addressed globally if there exists a single nonuniformly hyperbolic set where $E^+$, $E^-$ have constant dimensions and
$E^+\oplus E^-$ is dominated over this set.
Metric transitivity, implied by accessibility,  gives that the nonuniformly hyperbolic set must be dense, so that this dominated splitting extends to the whole manifold.

Such a set is constructed  in \cite{RRTU} for an open set of partially hyperbolic diffeomorphims, dense among those with $\dim(E^c) =2$.  This follows almost directly from results in \cite{BV} and \cite{BB}.  For the partially hyperbolic diffeomorphisms with  $\dim(E^c) > 2$, more delicate arguments are required.  These are contained in our previous work \cite{ACW1}, where we prove:

\begin{theorem}\cite{ACW1}\label{t=ACW1}
For a  generic map $f \in \Diff^1_m(M)$, either 
\begin{enumerate}
\item the Lyapunov exponents of $f$ vanish almost everywhere, or
\item $f$ is {\em non-uniformly Anosov}, meaning there exists a dominated splitting $TM = E^+\oplus E^-$ and $\chi_0>0$ such that for $m$-a.e. $x\in M$, and for every unit vector $v\in T_xM$
\begin{equation*}
v\in E^+  \implies \chi(x,v)> \chi_0  , \quad\hbox{and } v\in E^- \implies \chi(x,v) < -\chi_0.
\end{equation*}
Moreover, $f$ is ergodic.
\end{enumerate}
\end{theorem}

Partial hyperbolicity clearly forbids the first case, and thus for a $C^1$ residual set of  $C^1$ partially hyperbolic diffeomorphisms there is a globally dominated, nonuniformly hyperbolic splitting $E^+ \oplus E^-$.  Using the density of $C^1$ volume preserving diffeomorphisms among the $C^2$ \cite{A} and a semicontinuity argument, one obtains a $C^1$ open and dense set of $C^2$, colume-preserving  partially hyperbolic diffeomorphisms with a  non-uniformly hyperbolic set whose splitting $TM=E^+\oplus E^-$ is dominated.

Among these diffeomorphisms, partial hyperbolicity gives  the strong unstable and stable foliations $\cW^u$ and $\cW^s$
whose leaves  subfoliate the unstable and stable Pesin disks, respectively.
 Thus the Pesin stable and unstable disks are {\em definitely large} in the strong directions $E^u$ and $E^s$.
To address  difficulty (3), we need a way to
increase the size in the
non-strong directions in $E^+$ and $E^-$ to a definite scale as well.
This can be achieved using  a technique introduced in \cite {RRTU}.

In \cite {RRTU}, it is shown how so-called stable and unstable
{\em blenders} can be used to
resolve this third difficulty in a partially hyperbolic context.
An unstable blender is a ``robustly thick'' part of a
hyperbolic set, in the sense
that its stable manifold meets every strong unstable manifold $\cW^u(x)$ that comes
near it, and moreover this property is still satisfied (by its hyperbolic
continuation) after any $C^1$
perturbation of the dynamics.  The key point is that this
property may be satisfied even if the
dimensions of the strong stable manifolds and strong
unstable manifolds are not large: the (thick)
fractal geometry of the blender will be responsible for yielding
the ``missing dimensions" and fix the lack of transversality.

Since the Pesin unstable disks contain the
strong unstable manifolds, we conclude that any Pesin unstable manifold
near the blender has a part that is trapped by the blender dynamics.
Under iteration, this trapped part
evolves according to the hyperbolic dynamics of the blender, which
enlarges even the non-strong directions to a definite size.  Analogously
defined stable blenders
play a similar role of enlarging the Pesin stable manifolds to a definite
size.  If the unstable and stable blenders are contained in a larger
transitive hyperbolic set, then those long pieces of unstable and stable
manifolds do get close to one another and will thus intersect as desired.

\subsection{Blenders}\label{ss=blenders}

How often do blenders arise in the context of partially hyperbolic dynamics?  Originally, 
blenders were constructed using a very concrete geometric model, which
was then seen to arise in the unfolding of heterodimensional cycles between
periodic orbits whose stable dimension differ by one \cite {BD}.  This construction is used in \cite{RRTU} and accounts in part for the low dimensionality assumption on $E^c$ in their result.

The fractal geometry of
such a blender effectively yields one additional dimension in the above
argument, so in order to obtain multiple additional directions, one would
need to use several such blenders.  Unfortunately, there are robust
obstructions to the construction of some of the heterodimensional
cycles needed to produce such blenders.

A rather different approach to the construction of blenders was introduced
by Moreira and Silva \cite {MS}.  The basic idea
is that,
starting from a hyperbolic set whose fractal dimension is large enough to
provide the desired additional dimensions, a blender will arise after a
generic perturbation (``fractal transversality'' argument).  In their
work, they succeeded in implementing this idea  to obtain a blender
yielding a single additional dimension.

Here we will show that if the dimension of the hyperbolic set
is ``very large'', close to the dimension of the entire ambient manifold,
then a superblender (a blender capable of yielding all desired additional
dimensions), can be produced by a suitable $C^1$ small perturbation.
As it turns out, any regular perturbation of an ergodic
nonuniformly Anosov map admits such very large
hyperbolic sets.  Using Theorem~\ref{t=ACW1},
we can then conclude that superblenders
appear $C^1$ densely among partially hyperbolic dynamical systems, and
Theorem A follows.  In fact we have:

\begin{nThm*}{\bf{Theorem A'.}}
For any $r>1$, the space of $C^r$ partially
hyperbolic volume-preserving diffeomorphisms on a compact connected manifold
 contains a  $C^1$ open and dense subset of diffeomorphisms that are non-uniformly Anosov, ergodic and in fact Ber\-noulli.

\end{nThm*}

\ActivateToc
\addcontentsline{toc}{subsection}{\mbox{}\quad\quad\quad Discussion and questions}
\DeactivateToc
\subsection{Further discussion and questions}

To obtain absolute continuity of invariant foliations, a $C^{1+\alpha}$-regularity hypothesis is needed for the Hopf argument.  It still unknown if stable ergodicity can happen in the $C^1$-topology. In particular, the following well-known question is open.

\begin{quest}
Does there exist a non-ergodic volume preserving Anosov $C^1$-diffeo\-morphism on a connected manifold?
\end{quest}

For smoother systems, Tahzibi has shown~\cite{tahzibi2} that stable ergodicity
can hold for diffeomorphisms with a dominated splitting that are not partially hyperbolic.
One can thus hope to characterize stable ergodicity
by the existence of a dominated splitting.
The following conjecture could be compared to~\cite[Conjecture 0.3]{DW} about robust transitivity.

\begin{conjecture}
The sets of stably ergodic diffeomorphisms and of those
having a non-trivial dominated splitting have the same $C^1$ closure in $\diff^r_m(M)$, $r>1$.
\end{conjecture}

If one considers the space $\diff^r_\omega(M)$ of $C^r$-diffeomorphisms preserving
a symplectic structure $\omega$, a partial version of Theorem A has been established
with different techniques by Avila, Bochi and Wilkinson~\cite{ABW}:
\emph{The generic diffeomorphism in $\diff^1_\omega(M)$  is ergodic once it is partially hyperbolic.}

Note that in the second case the diffeomorphism is not always nonuniformly Anosov.
For that reason we cannot build blenders and obtain a symplectic version of Theorem A.

\begin{quest}
Is stable ergodicity $C^1$-dense in the space $\diff^r_\omega(M)$, $r>1$?
\end{quest}

\ActivateToc


\section{Further results and techniques in the proof}

In addition to the use of Theorem~\ref{t=ACW1}, there are two substantial new ideas used in the proof of Theorems A and A':
\begin{itemize}
\item {\bf  A ``Franks' Lemma" for horseshoes.}  (Theorem B)
\item {\bf Superblenders.}  (Theorem C)
 \end{itemize}


Before describing these ideas in further detail, we fix some notations.
For $f \in \Diff^1_m(M)$, $x \in M$, and a subspace $F \subset T_x M$,
we denote by $\Jac_F(f,x)>0$ the Jacobian of $Df$ restricted to $F$, i.e., the product of the singular values of $Df(x)|F$.

If $\Lambda$ is an invariant compact set, we denote by $h_{top}(\Lambda,f)$ the topological entropy of the restriction
of $f$ to $\Lambda$. If $\mu$ is an invariant probability measure, its entropy is denoted by $h_{top}(\mu,f)$.

As before, if $x$ is an Oseledets regular point, we denote by $\chi_1(x)>\dots>\chi_{\ell(x)}(x)$ its Lyapunov exponents and
by $T_xM=E_1(x)\oplus\dots\oplus E_{\ell(x)}(x)$ the Oseledets splitting. One sometimes prefers to list the Lyapunov exponents,
counted with multiplicity: in this case they are denoted by $\lambda_1(x)\geq \lambda_2(x)\geq \dots\geq \lambda_d(x)$, where $d=\dim(M)$.

\subsection{Linearization of horseshoes}
An essential tool for approximation in $C^1$ dynamics is a simple technique known as  ``the Franks lemma".  It asserts that for any periodic orbit $\cO$ of a diffeomorphism $f$, there is is a $C^1$ small perturbation $g$ of $f$, supported in a neighborhood of $\cO$, such that $g$ is affine near $\cO$.  Further perturbation is then vastly simplified starting from this affine setting.  We introduce and prove here an analogue of the Franks lemma for horseshoes.

\begin{figure}
\includegraphics[scale=0.32]{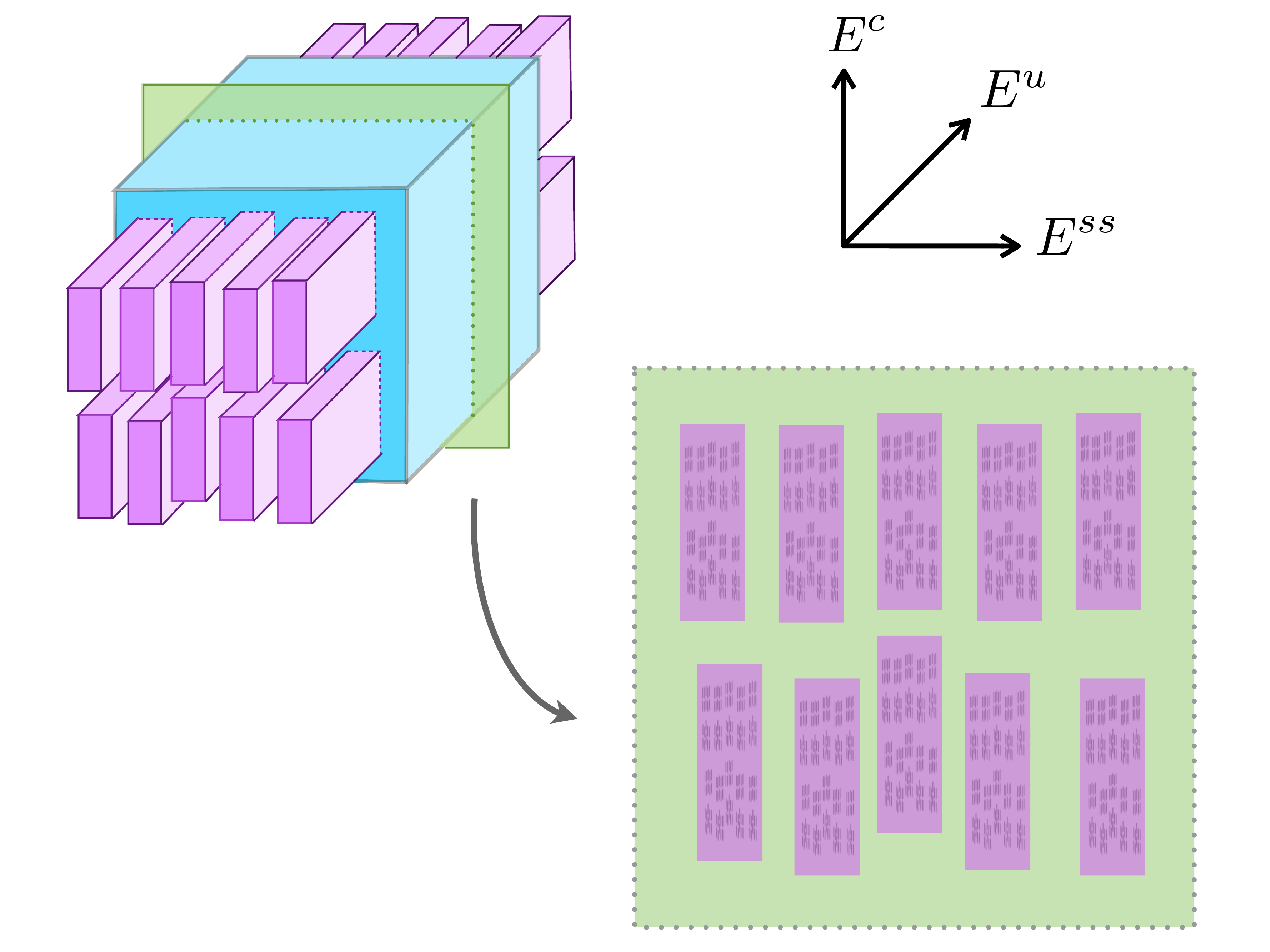}
\caption{An affine horseshoe with constant linear part and its slice inside a stable manifold.}
\end{figure}

Recall that a \emph{horseshoe} for a diffeomorphism $f$
is a transitive, locally maximal hyperbolic set $\Lambda$ that is totally disconnected and not  finite (such a set must be perfect, hence a Cantor set).
\begin{definition}
A horseshoe $\Lambda$ is \emph{affine}
if there exists a neighborhood $U$ of $\Lambda$
and a chart $\varphi\colon U\to \RR^d$ such that
$\varphi\circ f\circ \varphi^{-1}$ is locally affine near each point of $\Lambda$.

If one can choose $\varphi$ 
such that the linear part of $D(\varphi f\varphi^{-1})(x)$ coincides with some $A\in \GL(d,\RR)$ independent of $x$,
we say that $\Lambda$
has \emph{constant linear part} $A$.
\end{definition}

The next result is proved in Section~\ref{s.linearize}.    It is a key component in the proof of Theorem A'.

\begin{nThm*}{{\bf Theorem B.}} {\rm(Linearization)}
Consider a $C^r$ diffeomorphism $f$ with $r\geq 1$,
a neighborhood $\cU$ of $f$ in $\Diff^1(M)$ (in $\Diff^1_m(M)$ or in $\diff^1_\omega(M)$
if $f$ preserves the volume $m$ or the symplectic form $\omega$),
a horseshoe $\Lambda$ and $\varepsilon>0$.
Then there exist a $C^r$ diffeomorphism $g\in \cU$ with an affine horseshoe $\widetilde \Lambda$ such that:
\begin{itemize}
\item[--] $g=f$ outside the $\varepsilon$-neighborhood of $\Lambda$.
\item[--] $\widetilde \Lambda$ is $\varepsilon$-close to $\Lambda$ in the Hausdorff distance.
\item[--] $h_{top}(\widetilde \Lambda,g)  >  h_{top}(\Lambda,f) - \varepsilon$.
\end{itemize}
Moreover there exists a linearizing chart $\varphi\colon U\subset M\to \RR^d$
with $\widetilde \Lambda\subset U$ and a diagonal matrix $A$
whose diagonal entries are all distinct such that
$f$ coincides with the affine map
$z\mapsto A(z-x)+f(x)$ in a neighborhood of each point $x\in \widetilde \Lambda$.
\end{nThm*}

Such a result is false for  topologies stronger than $C^1$.
However, even in higher differentiability, one can always ``diagonalize" a sub horseshoe, as follows.
This result is proved in Section~\ref{s.diagonalize}, and is used in the proof of Theorem~B.

\begin{theorem}[Diagonalization]\label{t.diagonalize}
Consider a $C^k$ diffeomorphism $f$ with $k\geq 1$,
a neighborhood $\cU$ of $f$ in $\Diff^k(M)$ (in $\Diff^k_m(M)$ or in $\diff^k_\omega(M)$
if $f$ preserves the volume $m$ or the symplectic form $\omega$),
a horseshoe $\Lambda$ and $\varepsilon>0$.
Then there exist a $C^k$ diffeomorphism $g\in \cU$ with a horseshoe $\widetilde \Lambda$ such that:
\begin{itemize}
\item[--] $g=f$ outside the $\varepsilon$-neighborhood of $\Lambda$.
\item[--] $\widetilde \Lambda$ is $\varepsilon$-close to $\Lambda$ in the Hausdorff distance.
\item[--] $h_{top}(\widetilde \Lambda,g)  >  h_{top}(\Lambda,f) - \varepsilon$.
\item[--] $\Lambda_g$ admits a dominated splitting into one-dimensional sub bundles
$$T_\Lambda M=E_1\oplus \dots\oplus E_d.$$
\end{itemize}
\end{theorem}

\subsection{Approximation of hyperbolic measures by affine horseshoes}

For $C^r$ diffeomorphisms,
a theorem by A. Katok~\cite{katok} asserts that any  ergodic {\em hyperbolic  measure} (that is, an invariant probability measure for which $E^c$ is trivial)
can be approximated by a horseshoe.
It is possible to do this so that the horseshoe has a dominated splitting, with approximately the same  Lyapunov exponents on the horseshoe.
Strictly speaking, the original result of Katok does not explicitly mention such a
control of the Oseledets splitting, but no further work is really needed to
obtain it.  Since we have not been able to find out this precise
statement in the litterature, we include a proof of the following version
of Katok's theorem in Section~\ref{s.katok}.

\begin{theorem}[Katok's approximation]
\label{t.katok}
Consider  $r>1$, a  $C^{r}$-diffeomorphism $f$,
an ergodic, $f$-invariant, hyperbolic probability measure $\mu$, a constant $\delta>0$,
and a weak-* neighborhood $\cV$ of $\mu$ in the space of $f$-invariant probability measures on $M$. 
Then there exists a horseshoe $\Lambda \subset M$ such that:
\begin{enumerate}
\item $\Lambda$ is $\delta$-close to the support of $\mu$ in the Hausdorff distance;
\item $h_{top}(\Lambda, f)  >  h(\mu,f) - \delta$;
\item all the invariant probability measures supported on $\Lambda$ lie in $\cV$;
\item if $\chi_1>\dots>\chi_\ell$ are the distinct Lyapunov exponents of $\mu$,
with multiplicities $n_1,\dots,n_\ell\geq 1$, then there exists a dominated splitting on $\Lambda$:
$$T_\Lambda M=E_1\oplus\dots\oplus E_\ell,\quad \text{ with } \dim(E_i)=n_i;$$
\item there exists $n\geq 1$ such that for each $i=1,\dots,\ell$,
each $x\in \Lambda$ and each unit vector $v\in E_i(x)$,
$$\exp((\chi_i-\delta)n)\leq\|Df_0^n(v)\|\leq \exp((\chi_i+\delta)n).$$
\end{enumerate}
\end{theorem}

This can be combined with Theorem B in order to obtain (after a $C^1$-perturbation)
an affine horseshoe that approximates the measure, to give a measure theoretic version of Theorem B:

\begin{nThm*}{{\bf Theorem B'.}}
Consider $r>1$, a $C^r$ diffeomorphism $f$, a $C^1$-neighborhood $\cU\subset \Diff^r(M)$ of $f$,
an  ergodic hyperbolic measure $\mu$, a constant $\delta>0$
and a weak-* neighborhood $\cV$ of $\mu$ in the space of $f$-invariant probability measures on $M$. 
There exists $g\in \cU$ and an affine horseshoe $\Lambda$ with constant linear part $A$ such that:
\begin{itemize}
\item[--] $\Lambda$ is $\delta$-close to the support of $\mu$ in the Hausdorff distance;
\item[--] $h_{top}(\Lambda, g)  >  h(\mu,f) - \delta$;
\item[--] all the $g$-invariant probability measures supported on $\Lambda$ lie in $\cV$;
\item[--] $A$ is diagonal, with distinct real positive eigenvalues whose logarithms
$\lambda_1>\dots>\lambda_d$ are $\delta$-close to the Lyapunov exponents of $\mu$ (with multiplicity).
\end{itemize}
If $f$ preserves the volume $m$ or a symplectic form $\omega$, then  $g$ can be chosen to preserve it as well.
\end{nThm*}

\subsection{Blenders}\label{ss.blender}
The definition of a blender is not fixed in the literature (see, e.g. \cite{BCDW} for an informal discussion), so we will choose a rather general definition that is suited to our purposes.  We first define  stable (and analogously, unstable) blenders.

The data for a stable blender are: a horseshoe $\Lambda$ with a partially hyperbolic subsplitting $T_\Lambda M = E^{uu} \oplus E^{c} \oplus E^{s}$,  the blender itself, which is a local chart (``box") centered at a point of the horseshoe, and finally a conefield in the box that  contains $E^{uu}$ at points of $\Lambda$. The blender property requires that any disk tangent to the $E^{uu}$ cone and crossing the box meets the stable manifold of $\Lambda$ (see below for a formal definition).

The dimension of the center bundle of the splitting in some sense describes the strength of the blender.  In the classical blender construction, the dimension of $E^{c}$ is low -- either 1 or 2 \cite{BD,
RRTU}.  The reason for the low-dimensionality of $E^c$ in these constructions is the challenge of controlling the dynamics of $f$ in the central direction.  Roughly, the smaller the dimension of $E^c$, the less wiggle room
for a $uu$-disk to avoid the stable manifold of $\Lambda$.
The other extreme, where $\dim(E^{uu}) =1$ and $\dim(E^{c})$ is arbitrary,  is a ``superblender," which is what we construct here.

\begin{figure}[h]\label{f=superblender}
\includegraphics[scale=.35]{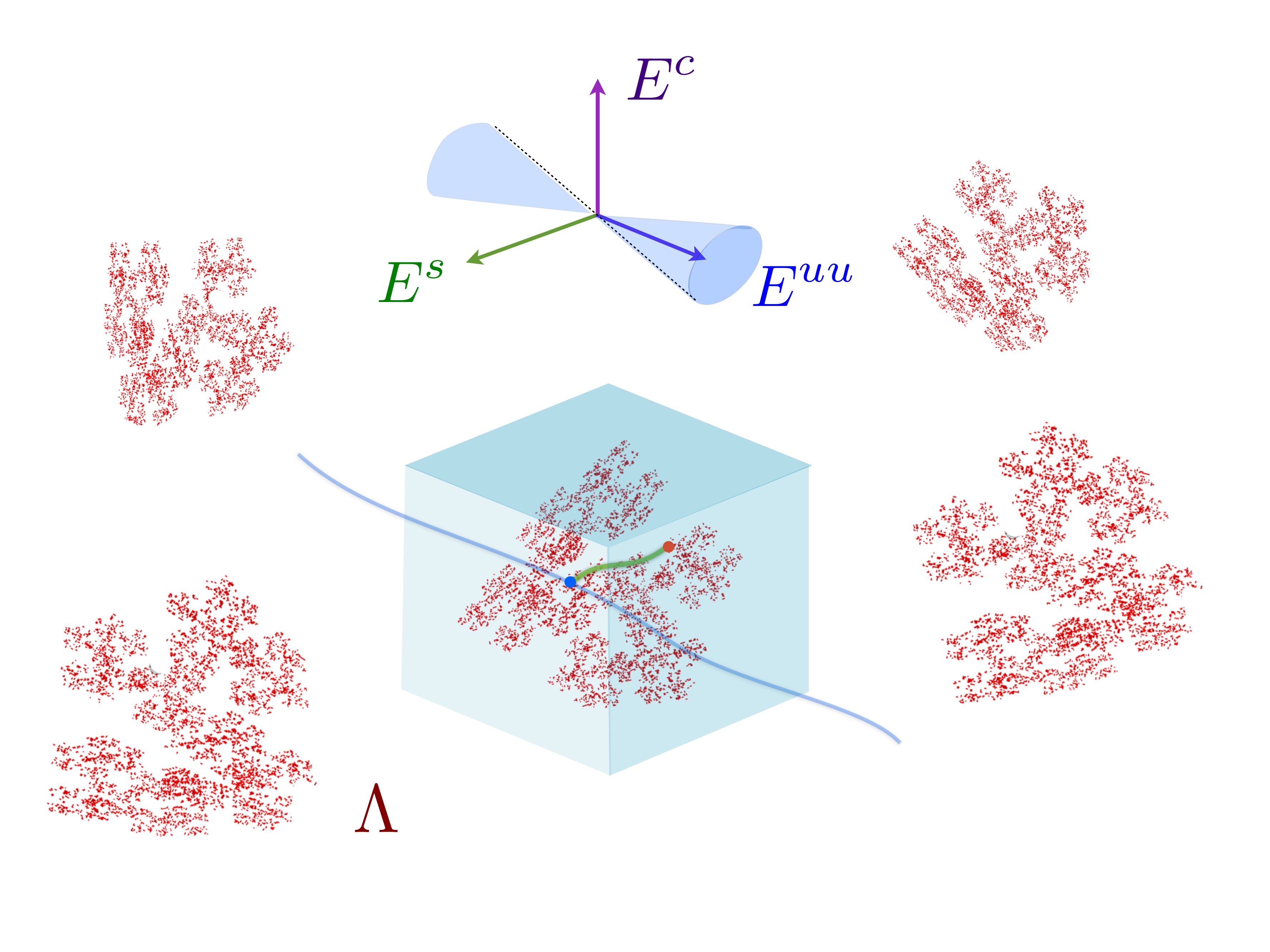}
\caption{A stable blender}
\end{figure}

In order to give a precise definition, we
fix an integer $1\leq d_{cs} \leq d:= \dim(M)$.
\begin{definition}
A horseshoe $\Lambda$ with a dominated splitting
$$T_\Lambda M=E^{u}\oplus E^s=(E^{uu}\oplus E^c) \oplus E^{s}, \quad
\dim(E^c\oplus E^s)=d_{cs}$$ 
is a \emph{$d_{cs}$-stable blender} if
there is a chart $\varphi\colon U\to (-1,1)^d$ of $M$ such that
\begin{itemize}
\item[--] $x=\varphi^{-1}(0)$ belongs to $\Lambda$ and
$\varphi^{-1}((-1,1)^{d-d_{cs}}\times \{0\}^{d_{cs}})$
is contained in the strong unstable manifold of $x$ tangent to $E^{uu}(x)$,
\item[--] the graph $\{(x,\theta(x))\}$ of any $1$-Lipschitz map
$\theta\colon (-1,1)^{d-d_{cs}}\to (-1,1)^{d_{cs}}$ meets the
local stable set of the hyperbolic continuation
$\Lambda_g$ of $\Lambda$ for each diffeomorphism $g$ that is $C^1$-close to $f$.
\end{itemize}
\end{definition}
This definition asserts that the stable set of $\Lambda$ behaves as a $d_{cs}$-dimensional
space transverse to the strong unstable direction.
We define similarly the notion of \emph{$d_{cu}$-unstable blender}.
We say that $\Lambda$ is a {\em stable superblender} if it is a $k$-stable blender for all
$k\in [\dim(E^s), d-1]$.    Equivalently, $\Lambda$ is a stable superblender if it is a $(d-1)$-stable blender and moreover its unstable bundle splits as dominated sum of one-dimensional subbundles:
$$E^u=E^u_1\oplus \cdots \oplus E^u_{d-d_{s}}.$$
Analogously, $\Lambda$ is an  {\em unstable superblender} if it is a $k$-unstable blender for all 
$k\in [d_u, d-1]$. Finally, $\Lambda$ is a {\em superblender} if it is both a stable and unstable superblender.

Blenders are obtained with the following theorem, proved in Section~\ref{s.birth}.
\begin{nThm*}{{\bf Theorem C.}}
Consider an integer $k\geq 1$, a $C^k$-diffeomorphism $f$  and an affine horseshoe
$\Lambda$ of $f$ with constant linear part $A\in \GL(d,\RR)$ such that:
\begin{itemize}
\item[--] $A$ preserves a dominated decomposition $\RR^d=E^u\oplus E^s=(E^{uu}\oplus E^c)\oplus E^s$.
\item[--] $A^{-1}$ is a contraction on $E^{uu}\oplus E^c$
and $A$ is a contraction on $E^s$.
\item[--] The measure of maximal entropy on $\Lambda$ ``almost satisfies" the Pesin formula:
\begin{equation}\label{e.almost-pesin}
h_{top}(\Lambda,f) > \log\Jac_{E^u}(A) -  \frac{1}{2k} \chi^u_{inf}(A),
\end{equation}
where $\chi^u_{inf}(A)$ is the smallest positive Lyapunov exponent of $A$.
\end{itemize}
Then there exists a $C^k$-perturbation $g$ of $f$ supported in a small neighborhood of $\Lambda$
such that the hyperbolic continuation $\Lambda_g$ is a $d_{cs}$-stable blender.
If $f$ preserves the volume $m$, then one can choose $g$ to preserve it also.
\end{nThm*}

We elaborate on the final hypothesis of Theorem C. In \cite{Pe}  Pesin proved that the Ruelle's inequality becomes equality in the case where $f$ is $C^2$ and the invariant measure $\mu$ is the volume $m$:
\begin{equation}\label{e=Pesineq}
h_m(f) =  \int_M  \sum_{\chi_i(x) \geq 0} \dim(E_i(x)) \chi_i(x) \, dm(x).
\end{equation}
More generally, equality (\ref{e=Pesineq}) holds precisely when the invariant measure $m$ has absolutely continuous disintegration along Pesin unstable manifolds \cite{LedYou}.   In particular, if $m$ is supported on a (proper) $C^2$ horseshoe, (\ref{e=Pesineq}) will never hold.  We can nonetheless quantify how close $m$ comes to satisfying
(\ref{e=Pesineq}); the final hypothesis of Theorem C requires that the measure of maximal entropy for the horseshoe $\Lambda$ (whose entropy is equal to $h_{top}(\Lambda,f)$)  be ``fat" along unstable manifolds.
Since for $m-a.e.$ point $x$
the sum of the positive Lyapunov exponents $\chi_i(x)$ counted with their multiplicity $\dim(E_i(x))$
coincides with $\log\Jac_{E^u}(A)$, the equality (\ref{e=Pesineq}) {\em almost} holds.

The topological entropy and the positive Lyapunov exponents are related
to unstable dimensions $d_i$ through the Ledrappier-Young formula~\cite{LedYou2}:
$$h_{top}(f)= \sum_{\chi_i(x) \geq 0} d_i \chi_i(x).$$
Condition~\eqref{e.almost-pesin} implies that the sum of unstable dimensions $d_i$ is larger than $d^u-\frac 1 {2k}$.
In the case $d^c=1$, Moreira and Silva have obtained~\cite{MS} a much stronger result, valid in the $C^\infty$ topology and
holding even for non-affine horseshoes, with the slightly different assumption that the ``upper-unstable dimension" of $\Lambda$ is larger than $1$.
Perturbations tend to increase the dimensions associated to the lower Lyapunov exponents
and to decrease the others. Consequently we expect that an optimal hypothesis in Theorem C should be:
$$h_{top}(\Lambda,f) > \log\Jac_{E^c}(A).$$

\medskip

Denote by $\DS$ the set of diffeomorphisms in $\Diff^1_m(M)$ with a nontrivial dominated splitting.
From Theorems 1,  B' and C, we obtain (see Section~\ref{s.stable}):
\begin{nThm*}{{\bf Corollary D.}}
Any diffeomorphism $f$ in a dense set of $\DS$ has a  superblender $\Lambda$.

Moreover, there exists a dominated splitting $TM=E\oplus F$
such that $\dim(E)$ coincides with the unstable dimension of $\Lambda$ and,
for any diffeomorphism $C^1$-close to $f$, the set of points having $\dim(E)$
positive Lyapunov exponents and $\dim(F)$ negative Lyapunov exponents has positive volume.
\end{nThm*}

Theorem A' is proved in Section~\ref{s.stable} by combining Corollary D and the criterion for ergodicity
obtained in~\cite{RRTU}.




\section{Stable Ergodicity}\label{s.stable}
We now build blenders (proving Corollary D) assuming Theorems~\ref{t=ACW1}, B', and C.
We then obtain Theorem A' using the following criterion:
(similar to~\cite{RRTU}):
\begin{equation*}
\begin{split}&\text{\it Partial hyperbolicity + accessibility + stable/unstable blenders}\\
&\text{\it + positive measure sets of points with large stable (resp. unstable) dimension}\\
&\quad\quad\quad\quad\quad\quad \Longrightarrow \quad \text{\it Ergodicity.}
\end{split}
\end{equation*}
We detail this argument.

\subsection{Regularization of $C^1$-diffeomorphisms}
The proof of Theorem A' uses Theorem 1, and hence forces us to work with diffeomorphisms that are only $C^1$. To recover results for $C^r$-diffeomorphisms, $r>1$, we will use:
\begin{theorem}[Avila \cite{A}]\label{t.smoothing}
$\Diff^\infty_m(M)$ is dense in $\Diff^1_m(M)$.
\end{theorem}

\subsection{Non-uniform hyperbolicity}
Recall that there exists a measurable invariant splitting $T_xM=E^+(x)\oplus E^0(x)\oplus E^-(x)$
defined over points $x$ in a set (called \emph{set of Oseledets regular points})
with full $m$-measure, obtained by summing spaces having positive, zero and negative Lyapunov exponents.
The Pesin stable manifold theorem asserts that for $\varepsilon>0$ small,
$$W^{-}(x):=\{z: \; \limsup_{n\to +\infty} \frac 1 n\log d(f^n(x),f^n(z))\leq -\varepsilon\}$$
is an injectively immersed submanifold tangent to $E^-(x)$.
Symmetrically, one obtains an injectively immersed submanifold $W^+(x)$ tangent to $E^+(x)$.
The dimensions $\dim(E^+(x))$, $\dim(E^-(x))$ are called \emph{unstable} and \emph{stable dimensions}
of $x$.

Let us denote by $\Nuh_f$ the set of Oseledets regular points of $f$ such that $E^0(x)=\{0\}$.
As a consequence of Theorem D in~\cite{AB}, we have:
\begin{theorem}\label{t.continuity}
For any diffeomorphism $f$ in a dense $G_\delta$ set
of $\diff^1_m(M)$ and for any $\varepsilon>0$, there exists a neighborhood $\cU$
of $f$ in $\Diff^1_m(M)$ such that each $g\in \cU$ satisfies:
$$m(\Nuh_f\setminus \Nuh_g)\leq \varepsilon.$$
\end{theorem}

\subsection{Criterion for ergodicity}
For $O$ a hyperbolic periodic orbit, we define the following sets:
$$\phcu(O) = \{x\hbox{ Oseledets regular }: W^+(x)\transverse W^s(O) \neq \emptyset\},$$
$$\phcs(O) = \{x\hbox{ Oseledets regular }: W^-(x)\transverse W^u(O) \neq \emptyset\},$$
where $W_1\transverse W_2$ denotes the set of transverse intersection between manifolds $W_1,W_2$,
i.e. the set of points $x$ such that $T_xW_1+T_xW_2=T_xM$.
The \emph{Pesin homoclinic class} is $\phc(O) := \phcu(O)\cap \phcs(O)$.
We stress the fact that $\phcs(O)$ can contain points $x$ whose stable dimension $\dim(E^-(x))$
is strictly larger than the stable dimension of $O$. However the set $\phc(O)$
only contains non-uniformly hyperbolic points whose stable/unstable dimensions are the same as $O$.

As a consequence of~\cite{katok} (see also~\cite[section 20]{katok-hasselblatt}), we have:
\begin{theorem}[Katok]\label{t.spectral}
Let $r>1$ and $f\in \Diff^{r}(M)$.  Let $\mu$ be a hyperbolic invariant probability
($\mu$-almost every point has no zero Lyapunov exponent).
Then there exist (at most) countably many Pesin homoclinic classes $\phc(O_n)$ whose
union has full $\mu$-measure.
\end{theorem}
In the previous statement the restriction $\mu|\phc(O_n)$ is not ergodic in general.
In the case $\mu$ is smooth this is however always the case.
\begin{theorem}[Rodriguez-Hertz\;-\;Rodriguez-Hertz\;-\;Tahzibi\;-\;Ures~\cite{RRTU}]\label{t.criterion}
Let $f\in \Diff^{r}_m(M)$ with $r>1$
and let $O$ be a hyperbolic periodic point such that $m(\phcu(O))$ and $m(\phcs(O))$
are positive.
Then $\phcs(O),\phcu(O),\phc(O)$ coincide $m$-almost everywhere
and $m|\phc(O)$ is  ergodic.  
\end{theorem}


From Theorem~\ref{t.criterion}  we obtain a criterion for the global ergodicity of the volume, which we will use to prove Theorem A'.

\begin{corollary}\label{c.criterion}
Let $f\in \Diff^{r}_m(M)$ with $r>1$ such that:
\begin{itemize}
\item $f$ preserves a partially hyperbolic splitting $TM=E^{uu}\oplus E^c\oplus E^{ss}$
and a dominated splitting $TM=E_1\oplus E_2$ such that $E^{uu}\subset E_1 \subset (E^{uu}\oplus E^c)$.

\item There exists a horseshoe $\Lambda$ with unstable bundle $E_1|\Lambda$ and
which is both a $(d_{uu}+d_c)$-unstable and $(d_c+d_{ss})$-stable blender, where $d_*=\dim E^*$,
\item The orbit of $m$-almost every point is dense in $M$.
\item There exist a positive $m$-measure set of regular points $x$
having unstable dimension $\dim(E^+(x))\geq \dim(E_1)$ and a positive $m$-measure set of regular points
having stable dimension $\dim(E^-(x))\geq \dim(E_2)$.
\end{itemize}
Then $f$ is ergodic.
\end{corollary}
\begin{proof}
Let us consider two charts $\varphi^{u},\varphi^{s}$
centered at two points $x^{u},x^{s}\in \Lambda$ as in the definition
of unstable and stable blenders given at Section~\ref{ss.blender}.
Let $O$ be a periodic orbit in $\Lambda$.

By assumption the orbit of $m$-almost every point $x\in M$ is dense in $M$, and accumulates on $x^u$.
By continuity of the leaves of the strong stable foliation in the $C^1$ topology,
the strong stable manifold $W^{ss}_{loc}(f^n(x))$ for some $n\in \ZZ$
is arbitrarily $C^1$-close to $W^{ss}_{loc}(x^u)$. From the blender property, we deduce that $W^{ss}_{loc}(f^n(x))$
intersects $W^u(y)$ for some point $y\in \Lambda$.
Since $M$ has a global dominated splitting $E_1\oplus E_2$,
if the stable dimension $\dim(E^-(x))$ of $x$ is greater than or equal to the stable dimension $\dim(E_2)$ of $\Lambda$,
the stable manifold of $x$ intersects $W^u(y)$ transversely.
Since the unstable manifold of $O$ is dense in the unstable set of $\Lambda$,
this implies that $x$ belongs to $\phcs(O)$.

Similarly, $m$-almost every point whose unstable dimension is greater than or equal to $E_1$
belongs to $\phcu(O)$. Note that $m$-almost every point has either stable dimension $\geq \dim(E_2)$
or unstable dimension $\geq \dim(E_1)$. Consequently the union $\phcu(O)\cup \phcs(O)$
has full volume. By our last assumption, $\phcu(O)$ and $\phcs(O)$ both have positive $m$-measure.
Theorem~\ref{t.criterion} thus applies and $\phcu(O),\phcs(O), M$ coincide up to a set of zero-volume.
Moreover $m=m|\phc(O)$ is ergodic.
\end{proof}

\subsection{Proof of Corollary D}
Consider a diffeomorphism $f\in \Diff^1_m(M)$
that preserves a non-trivial dominated splitting
$TM= E\oplus F$.
For diffeomorphisms $C^1$-close to $f$ this splitting persists,
and in particular the first case of Theorem~\ref{t=ACW1} does not hold.
It follows that there exists $f_1\in \Diff^1_m(M)$ close to $f$
that is ergodic and non-uniformly Anosov. We can thus change
the dominated splitting so that $\dim(E)$ coincides with the stable dimension
of $m$-almost every point.

We can furthermore require that $f_1$ belongs to the dense $G_\delta$ sets
provided by Theorem~\ref{t.continuity}.
In particular,
for any diffeomorphism $f_2$ in a $C^1$-neighborhood $\cU\subset \diff^1_m(M)$
of $f_1$, the set of non-uniformly hyperbolic points whose unstable dimensions
coincide with $\dim(E)$ has positive volume.
By Theorems~\ref{t.smoothing}, \ref{t.spectral} and~\ref{t.criterion},
one thus can choose
\begin{itemize}
\item a $C^2$ diffeomorphism $f_2$ that is $C^1$-close to $f_1$,
\item a hyperbolic periodic orbit $O$ for $f_2$, such that $m(\phc(O))>0$ and
the unstable dimension of $O$ is $\dim(E)$.
\end{itemize}

Pesin's formula~\cite{Pe} now applies to the normalization $\mu$ of $m|\phc(O)$:
\begin{theorem}[Pesin]
If $f\in \diff^1(M)$ and $\mu$ is an ergodic invariant probability measure absolutely continuous
with respect to a volume of $M$, then the Lyapunov exponents
$\lambda_\mu^1\geq \dots \geq \lambda^d_\mu$ of $\mu$ counted with multiplicity satisfy:
$$h_\mu(f)=\sum_i \max(\lambda^i_\mu, 0).$$
\end{theorem}

For any $\delta>0$, Theorem B' provides us with a $C^2$ diffeomorphism $f_3$ that is $C^1$-close to
$f_2$ and with an affine horseshoe $\Lambda$ whose linear part is constant and equal
to a diagonal matrix $A=\operatorname{Diag}(\exp(\lambda^1),\dots,\exp(\lambda^d))$,
such that
$$h_{top}(\Lambda, f_3)\geq \sum_i \max(\lambda^i, 0)-\delta.$$

Theorem C then implies that there exists $f_4\in\diff^1_m(M)$ that is $C^1$-close to $f_3$
such that the hyperbolic continuation of $\Lambda$ is a $(d-1)$-unstable blender.
Applying again Theorems B' and C to the measure of maximal entropy of $\Lambda$ for $f_4$,
one constructs a diffeomorphism $g$ that is $C^1$-close to $f_4$
(hence to the initial diffeomorphism $f$) such that the continuation of $\Lambda$ is a $(d-1)$-dimensional stable blender $\Lambda$, proving Corollary D.

\subsection{Metric transitivity}
Using that accessibility of the strong distributions for partially hyperbolic diffeomorphisms
is $C^1$-open and dense~\cite{DW}, Brin's argument~\cite{Br} gives:

\begin{theorem}[Brin, Dolgopyat-Wilkinson]\label{t.transitivity}
For any partially hyperbolic diffeomorphisms in an open and dense subset of $\diff^1_m(M)$,
$m$-almost every point has a dense orbit in $M$.
\end{theorem}

\subsection{Proof of Theorem A'}
For $r>1$, consider the $C^1$-open set $\cP\cH_m^r(M)$ of diffeomorphisms $f\in \diff^r_m(M)$
that preserve a partially hyperbolic decomposition $TM=E^s\oplus E^c\oplus E^u$.
By Theorems~\ref{t.smoothing}, \ref{t.transitivity} and Corollary D,
there exists a $C^1$-dense and $C^1$-open subset $\cU\subset \cP\cH_m^r(M)$
of diffeomorphisms $f$ having a horseshoe $\Lambda$ that is both a $(d_{uu}+d_c)$-dimensional unstable blender and
a $(d_c+d_{ss})$-dimensional stable blender  and such that $m$-almost every orbit is dense.
Moreover there exists a dominated splitting $TM=E\oplus F$ such that $\dim(E)$
coincides with the stable dimension of $\Lambda$ and
the set of non-uniformly hyperbolic points whose unstable dimension equals $\dim(E)$ has positive volume.
By Corollary~\ref{c.criterion}, any diffeomorphism in the open set $\cU$ is ergodic.
Since the set $\Nuh_f$ has positive volume, the measure $m$ is hyperbolic and the diffeomorphism $f$
is non-uniformly Anosov. By~\cite[Theorem 8.1]{Pe}, the system $(f,m)$ is Bernoulli.


\section{Horseshoes with simple dominated spectrum}
\label{s.diagonalize}

Our goal in this section is to prove Theorem~\ref{t.diagonalize},
which allows us to extract from a horseshoe $\Lambda$ a subhorseshoe $\widetilde \Lambda$
that has a dominated splitting into one-dimensional subbundles, after an arbitrarily $C^1$-small perturbation.  

Here is the scheme of the proof.
After a $C^1$-small perturbation, we  can assume that the given diffeomorphism is smooth
 in a neighborhood of $\Lambda$. The initial step the proof of this result is to apply Katok's Theorem~\ref{t.katok}
to the measure of maximal entropy of $\Lambda$.
This immediately implies Theorem \ref{t.diagonalize}
when the Lyapunov spectrum of $\mu$ is simple, so our basic task will be
to eliminate multiplicities in the Lyapunov spectrum.

\subsection{Non-triviality of the Lyapunov spectrum for cocycles over subshifts}

In this section we recall some basic results of \cite {BGV}.

Let $\sigma:\Sigma \to \Sigma$ be a \emph{subshift}, i.e., the restriction of the
shift on $\cA^\Z$ (where $\cA$ a finite set) to a transitive invariant compact subset.
For $l \leq r$ integers, we define the
\emph{$(l,r)$-cylinder} containing $x \in \Sigma$ as the set of all $y \in \Sigma$
such that $\pi_j(y)=\pi_j(x)$ for $l \leq j \leq r$, where $\pi_j:\Sigma \to
\cA$ are the coordinate projections.
We say that $x$ and $y$ have the same \emph{stable set}
(resp. \emph{local stable set})
if $\pi_i(x)=\pi_i(y)$ for any large integer $i$
(resp. for any $i\geq 0$).

A subshift is called \emph{Markovian} if there exists a
directed graph $\cG$ with vertices in $\cA$ such that $\Sigma$ consists of all
sequences corresponding to directed bi-infinite paths in $\cG$.
A \emph{subshift of finite type} is a subshift which is
topologically conjugate to a Markovian subshift.

Let $\sigma:\Sigma \to \Sigma$ be a Markovian subshift and let
$A:\Sigma \to \GL(d,\R)$ be continuous.  We say that the cocycle
$(\sigma,A)$ has \emph{stable holonomies} if for every $x,y$ in the same
stable set there exists $H_s(x,y) \in \GL(d,\R)$ such that
\begin{enumerate}
\item $H_s(y,z)\circ H_s(x,y)=H_s(x,z)$,
\item $H_s(\sigma(x),\sigma(y))\circ A(x)=A(y)\circ H_s(x,y)$,
\item $(x,y) \mapsto H_s(x,y)$ is a continuous function restricted to the set of $(x,y) \in
\Sigma \times \Sigma$ such that $y$ belongs to the local stable manifold of
$x$.
\end{enumerate}
We define analogously the unstable holonomies $H_u(x,y)$.
\medskip

We now give a condition for deducing the existence of stable
holonomies.

\begin{proposition}[Lemme 1.12 in~\cite{BGV}]\label{p.bunched}
If there exists $C,\epsilon>0$
such that whenever $x,y\in \Sigma$ belong to the same local stable manifold we have
for every $n \geq 0$,
\begin{equation}\label{e.bunched}
\|A_n(x)\| \|A_n(x)^{-1}\| \|A(\sigma^n(x))-A(\sigma^n(y))\|<
C e^{-\epsilon n},
\end{equation}
then the cocycle admits stable holonomies which satisfy:
\begin{equation}\label{e.holonomy}
H_s(x,y)=\lim_{n\to +\infty} A^{-1}(y)\dots A(\sigma^{n}(y))^{-1}A(\sigma^{n}(x))\dots A(x).
\end{equation}
\end{proposition}

The following is a particular case (for the measure of maximal entropy)
of the criterion for non-degenerate Lyapunov spectrum in \cite{BGV}.

\begin{theorem}[Bonatti\;-\;Gom\'ez-Mont\;-\;Viana] \label{criterion}

Assume that the cocycle $(\sigma,A)$ has stable and unstable holonomies and that its Lyapunov
exponents with respect to the measure of maximal entropy of $(\Sigma,\sigma)$ are all the same. 
Then there exists a continuous family $\mu_x$, $x \in \Sigma$, of probability measures
on $P\R^d$ such that
$$A(x)_*( \mu_{\sigma(x)})=\mu_{x},\text{ } H_s(x,y)_*
(\mu_y)=\mu_x \text{ and } H_u(x,y)_* (\mu_y)=\mu_x.$$

\end{theorem}

The following lemma will allow us to show that the conclusion of Theorem~\ref{criterion}
is not satisfied (hence that the Lyapunov exponents do not coincide).

\begin{lemma} \label{generic}
For $d \geq 2$ and $(B,B')$ in a dense $G_\delta$ subset of $PGL(d,\R)\times PGL(d,\R)$,
there is no probability measure on $P\R^d$ which is
invariant by both $B$ and $B'$.
\end{lemma}
\begin{proof} For $(B,B')$ in a dense $G_\delta$ subset of $PGL(d,\R)\times PGL(d,\R)$,
\begin{itemize}
\item the Oseledets spaces one or two-dimensional,
\item the argument of complex eigenvalues is not a rational multiple of $2\pi$,
\item if the Oseledets splitting of $B$ or $B'$ is not trivial, then
$B$ and $B'$ have distinct Oseledets subspaces,
\item if $d=2$ and $B,B'$ have complex eigenvalues,
they do not belong to the same compact subgroup of $PGL(2,\R)$.
\end{itemize}
The two first items imply that the ergodic $B$-invariant measure on $P\R^d$ are Dirac measures along the
one-dimensional Oseledets subspaces and smooth measures along the
$2$-dimensional Oseledets spaces. The same holds for $B'$.
By the third item, there is no probability measure simultaneously invariant by $B$ and $B'$
if the Oseledets splitting of $B$ or $B'$ is not trivial.

If the Oseledets splitting of $B$ and $B'$ is trivial, then $d=2$ and $B,B'$ have complex eigenvalues.
The set of elements of $PGL(2,\R)$ that preserve the (unique) probability measure on $P\R^2$
that is $B$-invariant  is precisely the compact subgroup of $PGL(2,\R)$ containing $B$.
Consequently $B$ and $B'$ do not preserve the same measure on $P\R^2$.
\end{proof}

\subsection{Eliminating multiplicities in the Lyapunov spectrum}
Let $f:M \to M$ be a $C^1$ diffeomorphism, and let $\Lambda$ be a horseshoe
with a dominated splitting $T_\Lambda M=E_1\oplus \dots \oplus E_\ell$.
Recall~\cite{Anosov-horseshoe} that $\Lambda$ is topologically conjugate to a subshift of finite type $\Sigma$
by a homeomorphism $h\colon \Lambda\to \Sigma$.
Using local smooth charts, the restriction of the derivative cocycle $Df$
to any subbundle $E_j$ can be represented by a continuous
$GL(d,\R)$-cocycle $A_j$ on $\Sigma$.

We say that the bundle $E_j$ is \emph{$\alpha$-pinched} if
there is $n\geq 1$ such that for any $x\in K$
$$\|Df^n(x)|E_j(x)\| \|(Df^n(x)|E_{j+1}(x))^{-1}\|\|Df^n(x)\|^{\alpha}<1,$$
$$\|(Df^{n}(x)|E_j(x))^{-1}\|\|Df^{n}(x)|E_{j-1}(x)\| \|(Df^{n}(x))^{-1}\|^\alpha<1.$$
It has the following well-known consequence (see for instance~\cite{PSW}).
\begin{proposition}
If $E_j$ is $\alpha$-pinched and $f$ is $C^2$, then $E_j$ is $\alpha$-H\"older.
\end{proposition}

We say that the bundle $E_j$ is \emph{$\alpha$-bunched} if
there is $n\geq 1$ such that for any $x\in K$
$$\|Df^n(x)|E_j(x)\| \|(Df^n(x)|E_j(x))^{-1}\|\|Df^n(x)|E^s(x)\|^{\alpha}<1,$$
$$\|Df^{-n}(x)|E^u(x)\|^{\alpha}\|Df^{-n}(x)|E_j(x)\| \|(Df^{-n}(x)|E_j(x))^{-1}\|<1.$$

Note that if $f$ is $C^2$ and if $E_j$ is $\alpha$-pinched and $\alpha$-bunched,
then the condition~\eqref{e.bunched} is satisfied and by Proposition~\ref{p.bunched},
$A_j$ has stable and unstable holonomies.

\begin{theorem} \label{pertu}
Let $f$ be a $C^k$ diffeomorphism
and $\Lambda$ a horseshoe
with a dominated splitting $TM=E_1 \oplus \cdots \oplus E_\ell$.
If $E_j$ is $\alpha$-pinched and $\alpha$-bunched with $\dim E_j\geq 2$,
then in every $C^k$-neighborhood of $f$, there exists $g$ with the following
property. 
The Lyapunov exponents of $Dg$ along $E_j(g)$ with respect to the measure of maximal entropy of
the continuation $\Lambda_g$ are not all equal.

Moreover,  if $f$ is volume preserving, $g$ can be chosen volume preserving as well.
\end{theorem}

\begin{proof}
The $\alpha$-pinching and $\alpha$-bunching are robust.
Up to a $C^k$-small perturbation, one can thus assume that $f$ is smooth, and
hence the cocycle $A_j$ associated to $Df|E_j$ admits stable and unstable holonomies $H_s,H_u$.
Let $p\in \Sigma$ be a $n$-periodic point
and $q\in \Sigma$ a homoclinic point of $p$
(so that $\sigma^{\pm \ell n}(q) \to p$ as $\ell \to \infty$).
We set $d_j=\dim(E_j)$.

The G$_\delta$ set $\cG\subset PGL(d_j,\R)\times PGL(d_j,\R)$ of Lemma \ref {generic}
can be obtained as a union
$$\cG=\bigcup_{B\in \cG_1} \{B\}\times \cG_B,$$
where $\cG_1$ and each $\cG_B$ is a dense G$_\delta$ subset of $PGL(d_j,\R)$.
Perturbing $f$ near $h^{-1}(p)$, if necessary,
we may assume that $B=A(\sigma^{n-1}(p))\circ\dots\circ A(p)$ belongs to $\cG_1$.
Consider another perturbation of
$f$ near $h^{-1}(q)$ and away from the closure of
$\{f^\ell(h^{-1}(q))\}_{\ell \in \Z \setminus \{0\}}$, such that $g\circ h^{-1}(q)=f\circ h^{-1}(q)$.
Consider $N\geq 1$ large such that
$\sigma^N(q)$ and $\sigma^{-N}(q)$ belong to the local stable manifold 
and to the local unstable manifold of $p$ respectively.
The formula~\eqref{e.holonomy} shows that the holonomies
$H_u(p,\sigma^{-N}(q))$ and $H_s(\sigma^N(q),p)$ are not modified by the perturbation near $q$.
One can thus assume that after the perturbation the following map belongs to $\cG_B$:
$$B'=H_s(\sigma^N(q),p)\circ \big(A(\sigma^{N-1}(q))\circ\dots A(q)
\circ\dots A(\sigma^{-N}(q))\big)\circ H_u(p,\sigma^{-N}(q)).$$
Since there is no probability measure on $P\R^{d_j}$ that is simultaneously preserved by $B$ and $B'$,
 Theorem~\ref{criterion} shows that the Lyapunov exponents along $E_j$ for the measure of maximal entropy
on $\Lambda$ cannot all coincide.
\end{proof}

\subsection{Proof of Theorem \ref{t.diagonalize}}
Up to an arbitrarily small perturbation, we can assume that $f$ is smooth in a neighborhood of $\Lambda$.
Let $\ell$ be the number of distinct Lyapunov exponents of $\mu$.
Using Theorem~\ref{t.katok}, we can replace $\Lambda$ by a subhorseshoe $\Lambda_1$
endowed with a dominated splitting into $\ell$ subbundles and whose topological entropy
is arbitrarily close to the entropy of $\Lambda$.
If $\ell=\dim(M)$, we are done.

If $\ell<\dim(M) = d$, we apply Theorem~\ref{pertu} to obtain a perturbation $f_1$ of $f$
for which the measure of maximal entropy on the
continuation $\Lambda'_1$ of $\Lambda_1$ has $\ell_1>\ell$ distinct Lyapunov
exponents.  Theorem~\ref{t.diagonalize} follows after repeating this procedure at most $d-\ell$ times.
\qed


\section{A linear horseshoe by perturbation}\label{s.linearize}

In this section we prove Theorem~B.

\subsection{Partially hyperbolic horseshoes with essential center bundle}

If $\Lambda$ is a horseshoe, it admits a unique measure of maximal entropy $\mu$.
We refer to~\cite{bowen} for its properties. In particular:
\begin{itemize}
\item The measure $\mu$ may be disintegrated along \emph{every} unstable manifold $W^u(x)$, $x\in \Lambda$ as a (non-finite) measure $\mu^u$, which is well-defined up to a multiplicative constant.
Hence the notion of measurable sets $A,B\subset W^u(x)$ with positive
$\mu^u$-measure is well defined, as is their ratio $\mu^u(A)/\mu^u(B)$.
\item With respect to the disintegration $\mu^u$, the map
$f$ has constant Jacobian along unstable leaves:
for any measurable sets $A,B\subset W^u(f(x))$, the ratios
$\mu^u(A)/\mu^u(B)$ and $\mu^u(f^{-1}(A))/\mu^u(f^{-1}(B))$ are equal.
\item The disintegration $\mu^u$ is invariant under stable holonomy.
If $\rho$ is small enough,
for any $x,y\in \Lambda$ with $y\in W^s(x,\rho)$ the stable holonomy defines
a map $\Pi^s_{x,y}$ from $W^u(x,\rho)$ to $W^u(y)$: the point
$\Pi^s_{x,y}(z)$ is the unique intersection point between $W^s_{loc}(z)$
and $W^u_{loc}(y)$.
Then for any two measurable sets $A,B\subset W^u(x,\rho)$, the ratios 
$\mu^u(A)/\mu^u(B)$ and $\mu^u(\Pi^s_{x,y}(A))/\mu^u(\Pi^s_{x,y}(B))$ are equal.
\end{itemize}

In general, the strong unstable leaves have zero $\mu^u$-measure, as the next proposition makes precise.

\begin{proposition} \label{p=dimension}
Let $\Lambda$ be a horseshoe for a $C^1$-diffeomorphism $f$ with a partially hyperbolic splitting:
$$
T_\Lambda  M = E^{uu} \oplus E^{c} \oplus E^s,
$$
where $E^s$ is the stable bundle and $E^u = E^{uu}\oplus E^{c}$  is the unstable
bundle in the hyperbolic splitting for $f\vert_{\Lambda}$.
Then the following dichotomy holds.
\begin{enumerate}   
\item Either $\mu^u(W^{uu}(x)) = 0$, for every $x\in \Lambda$, where $\mu^u$ is the disintegration of the measure of maximal entropy along $W^{u}(x)$, or
\item $\Lambda\cap W^u(x)\subset W^{uu}(x)$ for every $x\in \Lambda$.\end{enumerate}
\end{proposition}
In the second case, note that the local stable and strong unstable laminations are jointly integrable
and that the Hausdorff dimension of $\Lambda\cap W^{u}(x)$ is therefore less than or equal to $\dim(E^{uu})$.

When the first case holds, we say that the center bundle $E^{c}$ of
$\Lambda$ is \emph{essential}.
\bigskip

\begin{proof}[Proof of Proposition~\ref{p=dimension}]
Consider a Markov partition $\cC=\{C_0,\dots,C_s\}$ of $\Lambda$
into small compact disjoint rectangles: in particular, there exists $\rho$ such that
the local manifolds $W^u(x,\rho)$ and $W^s(y,\rho)$
intersect at a unique point whenever $x,y$ belong to the same rectangle $C_i$.
We denote by $C(x)$ the rectangle containing the point $x\in \Lambda$.
For $n\geq 1$, we also introduce the iterated Markov partition $\cC^n$, which is the
collection of rectangles of the form $C_{i_0}\cap f^{-1}(C_{i_1})\cap\dots\cap
f^{-(n-1)}(C_{i_{n-1}})$.

\begin{lemma}
If the second condition of the proposition does not hold, then
there exists $n\geq 2$ satisfying the following.
For any $x\in \Lambda$ there exists a sub rectangle $C'\subset C(x)$ in $\cC^n$
such that $W^{uu}_{loc}(x)\cap C'=\emptyset$.
\end{lemma}
\begin{proof}
Assume that the second condition of the proposition does not hold:
there exist two points $z_1,z_2$ in the same unstable manifold
such that $W^{uu}(z_1)$ and $W^{uu}(z_2)$ are different.
Taking a negative iterate if necessary, we may assume that
$z_1,z_2$ belong to the same local unstable manifold and the same rectangle $C$.
In particular, there exists $m>1$ and two subrectangles
$C_1,C_2\in \cC^m$ with $z_1\in C_1$ and $z_2\in C_2$
such that the plaque $W^u_{loc}(z_1)$ satisfies the following property:
for any $x_1\in C_1\cap W^{u}_{loc}(z_1)$ and $x_2\in C_2\cap W^{u}_{loc}(z_1)$,
the manifolds $W^{uu}_{loc}(x_1)$ and $W^{uu}_{loc}(x_2)$ are disjoint.
By compactness, the local unstable manifold of any point $z$ close to $z_1$
satisfies the same property.

Now consider any point $x\in \Lambda$.
Since $\Lambda$ is locally maximal and transitive,
there exists a point $z$ close to $z_1$ having a backward iterate
$f^{-p}(z)$ in $W^u_{loc}(x)\cap C(x)$.
It follows that $C(x)$ contains two rectangles $C'_1,C'_2\in \cC^{m+p}$
satisfying: for any $x_1\in C'_1\cap W^{u}_{loc}(x)$ and $x_2\in C'_2\cap W^{u}_{loc}(x)$,
the manifolds $W^{uu}_{loc}(x_1)$ and $W^{uu}_{loc}(x_2)$ are disjoint.
In particular, $W^{uu}_{loc}(x)$ is disjoint from $C'_1$ or $C'_2$
and the lemma holds for the point $x$ and any integer $n$ larger than $m+p$.
Note that it also holds for any point $x'\in \Lambda$ close to $x$.
By compactness we obtain that there exists $n\geq 1$ such that
the lemma holds for all $x\in \Lambda$.
\end{proof}

We now continue with the proof of the proposition. We assume that the second case does not hold
and consider an integer $n$ as in the previous lemma.
For any $x\in \Lambda$ and $m\geq 1$ we denote by $T_m(x)$
the union of the rectangles in $\cC^m$ that are contained in $C(x)$ and meet
$W^{uu}_{loc}(x)$.

Since $\mu$ has full support $\Lambda$ and since its disintegration is invariant under
stable holonomy,
there exists $\tau>0$ such that for any $x\in\Lambda$
and any $C'\in \cC^n$ contained in $C(x)$, we have
$$\mu^u(C'\cap W^u_{loc}(x))>\theta.\mu^u(C(x)\cap W^u_{loc}(x)).$$
It follows that
$$\mu^u(T_n(x)\cap W^u_{loc}(x))<(1-\theta).\mu^u(C(x)\cap W^u_{loc}(x)).$$
Since $\mu^u$ has constant jacobians along the unstable leaves,
we have
$$\mu^u(T_{(j+1).n}(x)\cap W^u_{loc}(x))<(1-\theta).\mu^u(T_{j.n}(x)\cap W^u_{loc}(x)).$$
The measure $\mu^u(W^{uu}_{loc}(x)\cap C(x))$ is the limit of
$\mu^u(T_{j.n}(x)\cap W^u_{loc}(x))$ which is exponentially small.
Thus $\mu^u_{loc}(W^{uu}(x))$ has zero $\mu^u$-measure, for all $x\in \Lambda$.
\end{proof}

\bigskip

\subsection{A reverse doubling property of partially hyperbolic horseshoes}
In this subsection and the following ones, we will consider a diffeomorphism
$f$ and a horseshoe $\Lambda$ satisfying the following hypothesis:
\begin{itemize}
\item[(H)] $f$ is a $C^{1+\alpha}$-diffeomorphism for some $\alpha>0$
and the horseshoe $\Lambda$ has a partially hyperbolic splitting:
$$
T_\Lambda  M = E^{uu} \oplus E^{c} \oplus E^s,
$$
where $E^s$ and $E^u = E^{uu}\oplus E^{c}$ are the stable and unstable
bundles in the hyperbolic splitting for $f\vert_{\Lambda}$ and where $E^{c}$
is one-dimensional and essential.
\end{itemize}
\bigskip

\subsubsection{The reverse doubling property}
We prove a geometric inequality of independent interest
for the disintegration of the  measure  of maximal entropy along unstable leaves.
A measure satisfying this inequality is sometimes said to have
the \emph{reverse doubling property} -- for a discussion of this property,  see~\cite{HMY}.
We will later need a more technical version of this property
(see Lemma~\ref{l.density}) that will be proved analogously.

\begin{theorem}[Reverse doubling property]\label{t.doubling}
For any diffeomorphism $f$ and any horseshoe $\Lambda$
satisfying the property (H), there exist $\rho,\eta>0$ such that
for any $x\in \Lambda$ and $r\in (0,\rho)$,
\begin{equation}\label{e.reversedoubling}
\mu^u\left(W^{u}(x,\eta r)\right)<\frac 1 2 \mu^u\left(W^{u}(x,r)\right).
\end{equation}
\end{theorem}
\bigskip

We will at times be interested in replacing $\Lambda$ by subhorseshoes that have a large period; i.e.,
that have a partition $K\cup f(K)\cup \dots f^{N-1}(K)$ into disjoint compact sets, for some
large integer $N$.  The next result states that these subhorseshoes can be extracted to have large entropy and a uniform reverse doubling property.

\begin{theorem}[Reverse doubling property and large period]\label{t.doubling-uniform}
For any diffeomorphism $f$ and any horseshoe $\Lambda$
satisfying  (H)  and for any $\varepsilon>0$,
there exist $\eta,\rho>0$ with the following property.

For any $N_0\geq 1$ there exist $N\geq N_0$ and a subhorseshoe $\Lambda_N\subset \Lambda$ that:
\begin{itemize}
\item admits a partition into compact subsets $\Lambda_N=K\cup\dots f^{N-1}(K)$,
\item has entropy larger than $h_{top}(\Lambda,f)-\varepsilon$,
\item for any $x\in \Lambda_N$ and $r\in (0,\rho)$, we have:
\begin{equation*}
\mu^u\left(W^{u}(x,\eta r)\right)<\frac 1 2 \mu^u\left(W^{u}(x,r)\right).
\end{equation*}
\end{itemize}
\end{theorem}
\bigskip

\subsubsection{Split Markov partitions}
The reverse doubling property will be obtained from a special construction of Markov partitions
satisfying a geometrical property that we introduce now.
This construction uses strongly that the unstable bundle
splits as a sum $E^u=E^{uu}\oplus E^{c}$ with $\dim(E^{c})=1$.

Let us fix $\rho_0>0$ small. Since $\Lambda$ has a local product structure,
for any $x,y\in \Lambda$ close, the intersection
$[x,y]:=W^u(x,\rho_0)\cap W^s(y,\rho_0)$ is transverse,  and consists of a single point that belongs to $\Lambda$.
A \emph{rectangle} of $\Lambda$ is a closed and open subset with diameter smaller than $\rho_0$
that is saturated by the local product: for any $x,y$ in a rectangle, $[x,y]$ also belongs to the rectangle.
\medskip

\begin{definition}
A \emph{split rectangle} of $\Lambda$ is a rectangle $R$ that can be expressed as the disjoint union of two subrectangles
$R_-,R_+$ such that (see Figure~\ref{f.split}):
\begin{itemize}
\item the decomposition is saturated by local stable manifolds:
any two points $x,y\in R$ with $y\in W^s(x,\rho_0)$ belong to a same subrectangle
$R_{-}$ or $R_{+}$; and
\item inside unstable leaves, the subrectangles $R^-,R^+$ are bounded by strong unstable leaves: for any $x\in R$, there exists $z\in W^u(x,\rho_0)$ such that
$R^-\cap W^u(x,\rho_0)$ and $R^+\cap W^u(x,\rho_0)$ are contained in two different connected
components of $W^{u}(x,\rho_0)\setminus W^{uu}(z,2\rho_0)$.
\end{itemize}
A \emph{split Markov partition} of $\Lambda$ is a collection of pairwise disjoint split rectangles
$\{R_i=R_{i,-}\cup R_{i,+},\; i=1,\dots,m\}$ such that both $\{R_i\}$
and $\{R_{i,-}\}\cup \{R_{i,+}\}$ are Markov partitions.
\end{definition}

\begin{figure}
\includegraphics[scale=0.32]{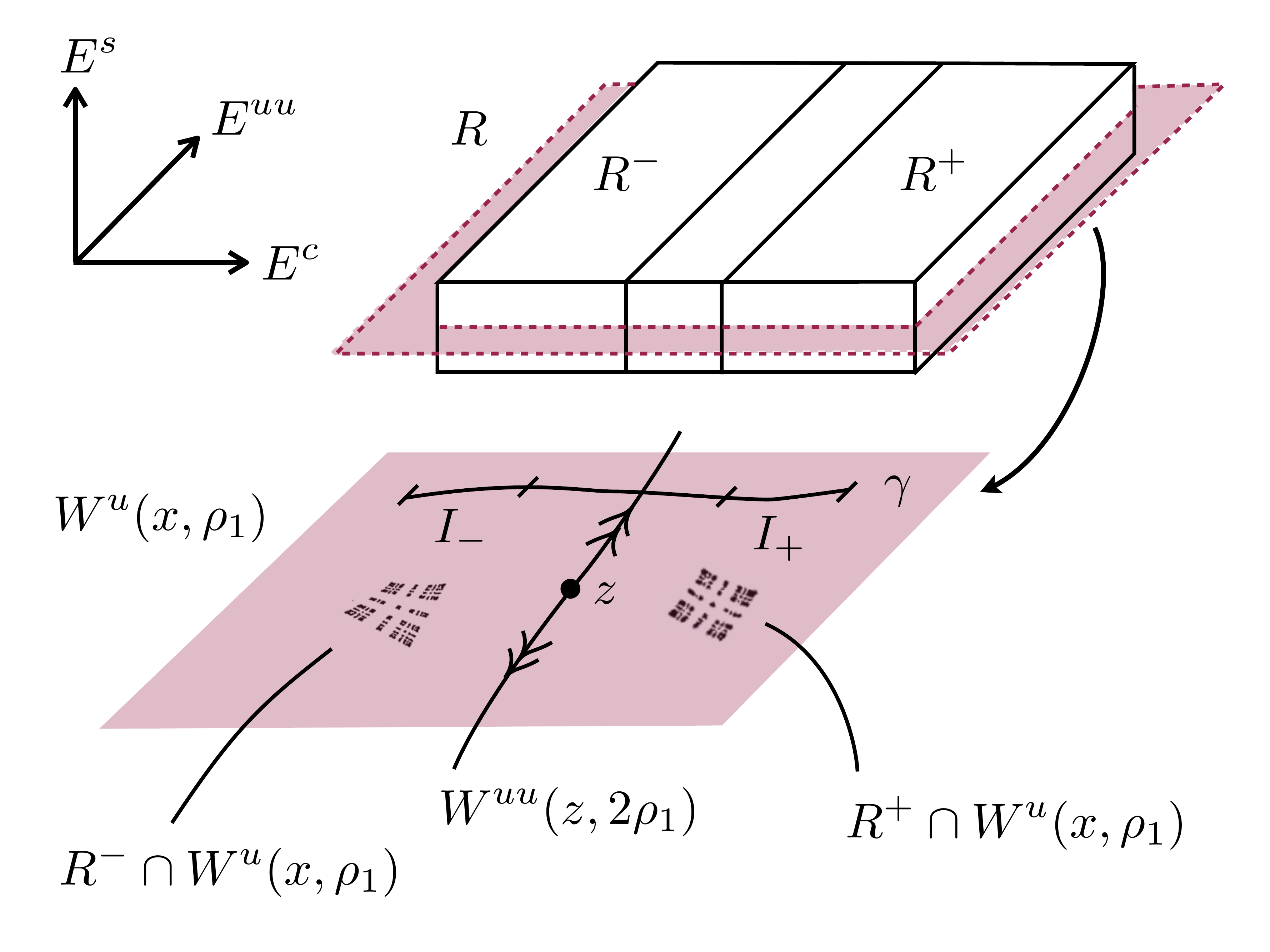}
\caption{\label{f.split} A split rectangle and its intersection with an unstable plaque.}
\end{figure}

The next result says that one can extract subhorseshoes with large period and with a split Markov partition.

\begin{proposition}\label{p.split}
For any diffeomorphism $f$ and any horseshoe $\Lambda$
satisfying (H), for any $\rho',\varepsilon>0$,
there exists a horseshoe $\Lambda'\subset \Lambda$ endowed with a split Markov partition $\{R_{1,\pm},\dots,R_{m,\pm}\}$ whose rectangles have diameter smaller than $\rho'$
and such that the following property holds.

For every $N_0\geq 0$, there exist $N\geq N_0$
and a horseshoe $\Lambda_N\subset \Lambda'$ such that:
\begin{itemize}
\item $\Lambda_N$ admits a decomposition into compact subsets
$$\Lambda_N=K\cup f(K)\cup\dots\cup f^{N-1}(K),$$
\item the topological entropy $h_{top}(\Lambda_N,f)$ is larger than $h_{top}(\Lambda,f)-\varepsilon$,
\item for each $i$ and $x\in \Lambda_N\cap R_i$, the disintegration $\mu^u$ of the measure of maximal entropy
of $\Lambda_N$ along the unstable leaves gives the same weight to
$W^{u}(x,\rho_0)\cap R_{i,+}$ and $W^{u}(x,\rho_0)\cap R_{i,-}$.
\end{itemize}
\end{proposition}
\begin{proof}
Consider a Markov partition $C_0,\dots,C_s$ of $\Lambda$ into disjoint compact rectangles
with diameter smaller than $\rho'$.
We can furthermore require that
the maximal invariant set $A$ in the smaller collection
$C_1\cup \dots\cup C_s$ has entropy larger than $h_{top}(\Lambda,f)-\varepsilon/4$.
A word $i_1,\dots,i_n$ in $\{0,\dots,s\}^n$ is \emph{admissible} if
$f^{n-1}(C_{i_1})\cap f^{n-2}(C_{i_2})\cap\dots\cap C_{i_n}\neq \emptyset$.
\medskip

 Fix some point $x_0\in C_0$.
Since $E^{c}$ is essential, there exist $x^-,x^+\in C_0\cap W^u(x_0,\rho_0)$
such that $W^{uu}(x^\pm,2\rho_0)$ are disjoint.
For $\ell\geq 1$ let
$$(i_{-\ell},\dots, i_{-1},0, i_{1}^-,\dots,i_{\ell}^-),\;
(i_{-\ell},\dots, i_{-1},0, i_{1}^+,\dots,i_{\ell}^+)\; \in \{0,\ldots,s\}^{2\ell+1}$$
denote the itinerary of $f^{-\ell}(x^+),\dots,f^{\ell}(x^+)$
and $f^{-\ell}(x^-),\dots,f^{\ell}(x^-)$ in the partition $C_0,\dots,C_s$.
If $\ell$ is large enough, the rectangles
$$R_{-}=f^{\ell}(C_{i_{-\ell}})\cap\dots\cap C_0\cap f^{-1}(C_{i_1^-})\cap f^{-\ell}(C_{i_\ell^-}),$$
$$R_{+}=f^{\ell}(C_{i_{-\ell}})\cap\dots\cap C_0\cap f^{-1}(C_{i_1^+})\cap f^{-\ell}(C_{i_\ell^+}),$$
are small neighborhoods of $x^-,x^+$ in $\Lambda$.
The union $R=R_-\cup R_+$ is a split rectangle.
Note that one can modify $x^-,x^+$ and take $\ell$ so that $i_\ell^-=i_\ell^+$.
\medskip

We then consider some integer $L$ large and the admissible words of length $L$
of the form $i_1^-,\dots,i_\ell^-,w,i_{-\ell},\dots,i_{-1},0$ or $i_1^+,\dots,i_\ell^+,w,i_{-\ell},\dots,i_{-1},0$
with admissible words $w$ in $\{1,\dots,s\}$ of length $L-2\ell-1$.
The bi-infinite words obtained by concatenation of these words define the horseshoe $\Lambda'$.
There exists $L$ arbitrarily large such that the entropy of $\Lambda'$ is arbitrarily close to the entropy of $\Lambda$, and hence is larger than $h_{top}(\Lambda,f)-\varepsilon/3$.

If $L$ has been chosen large enough, we note that $R\cap \Lambda'$ is disjoint from its $L-1$
first iterates. Hence, any segment of orbit of $\Lambda'$ of length $L$ meets $R$ at one point exactly.
One can thus define a Markov partition by taking rectangles of the form
$$f^{-n}(R\cap \Lambda')\cap f^{-n+1}(C_{j_1})\cap \dots \cap f^{-1}(C_{j_{n-1}})\cap C_{j_n},$$
with $n$ varying between $0$ and $L$.
These rectangles are compact, disjoint and naturally split by $f^{-n}(R_{\pm})$.
We thus get a collection of disjoint split rectangles $R_{1,\pm},\dots,R_{m,\pm}$ of $\Lambda'$.
Both partitions $\{R_i\}$ and $\{R_{i,-}\}\cup \{R_{i,+}\}$ are Markov by construction.

Since $i_\ell^-=i_\ell^+$, the set of  itineraries of the forward orbits from
$R_-$ and $R_+$ are equal after time $\ell$.
This implies that the disintegration of the maximal entropy measure of $\Lambda'$
gives the same weight to $W^{u}(x,\rho_0)\cap R_{+}$ and $W^{u}(x,\rho_0)\cap R_{-}$
for each $x\in R\cap \Lambda'$. Similarly, it gives the same weight to $W^{u}(x,\rho_0)\cap R_{i,+}$ and $W^{u}(x,\rho_0)\cap R_{,i-}$
for each $x\in R_i$.
\medskip

Fix a Markov rectangle $R'$ that meets $R$ after $L-\ell$ iterates:
if $L$ is large enough, the set of points of $\Lambda'$ whose orbit does not intersect $R'$
has entropy larger than $h_{top}(\Lambda,f)-\varepsilon/2$.
For $N$ large (and a multiple of $L$), 
the subhorseshoe $\Lambda_N$ corresponds to itineraries
that visit $R'$ exactly every $N$-iterates. Since $N$ is large,
the entropy of $\Lambda_N$ is larger than $h_{top}(\Lambda,f)-\varepsilon$.
This horseshoe is $N$-periodic: setting $K:=R'\cap \Lambda_N$,
it can be decomposed as
$$\Lambda_N=K\cup f(K)\cup\dots\cup f^{N-1}(K).$$
Again by symmetry of the construction the disintegration of the maximal entropy measure on $\Lambda'$
gives the same weight to $W^{u}(x,\rho')\cap R_{i,+}$ and $W^{u}(x,\rho')\cap R_{i,-}$
for each $i$ and each $x\in R_i\cap \Lambda_N$.
This gives the proposition.
\end{proof}
\bigskip

\subsubsection{Proof of the reverse doubling property}
We now prove Theorem~\ref{t.doubling-uniform} (the proof of Theorem~\ref{t.doubling} is similar).
\begin{proof}[Proof of Theorem~\ref{t.doubling-uniform}]
Since there is a dominated splitting between $E^{uu}$ and $E^{c}$,
there exists a thin cone field $\cC^{c}$ defined on a small neighborhood of $\Lambda$
and containing the bundle $E^{c}$ such that for any $x\in \Lambda$, and any small $C^1$-curve $\gamma$
in $W^{u}(x)$ containing $x$ and tangent to $\cC^{c}$, the backward iterates
$f^{-n}(\gamma)$ are still tangent to $\cC^{c}$. By a classical distortion argument,
since $\gamma$ is uniformly contracted
by backward iterations, and since $f$ is $C^{1+\alpha}$, there exists some uniform constant $C>0$ such that
for any unit vectors $v,v'$ tangent to $\gamma$ at  points $y,y'\in \gamma$ we have
$$C^{-1}\|Df^{-n}(y')(v')\|\leq\|Df^{-n}(y)(v)\|\leq C \|Df^{-n}(y')(v')\|.$$
Let's apply  Proposition~\ref{p.split} to $f$, $\Lambda$, $\varepsilon$
and a small constant $\rho'$:
we obtain a split Markov partition $\{R_{i,\pm}\}$ of $\Lambda'\subset \Lambda$.
For each $R_i$ and each $x\in R_i$, there exists a small curve $\gamma\subset W^u(x,\rho_0)$
tangent to $\cC^{c}$ and two disjoint compact intervals $I_+,I_-\subset \gamma$ such that
for any $y\in R_i\cap W^u(x,\rho_0)$ the strong manifold $W^{uu}(y,\rho_0)$
intersects $\gamma$ at a point of $I_+$ if $y\in R_{i,+}$ and at a point of $I_-$ otherwise.
(See Figure~\ref{f.split}.) If $\rho'$ is small, the curve $\gamma$ is also small.

There exist two constants $\widehat L,L>0$, which do not depend on $x\in \Lambda'$
or on $\gamma\in W^u(x,\rho_0)$, such that:
\begin{itemize}
\item the length of $\gamma$ is smaller than $\widehat L$,
\item the distance between $I_-$ and $I_+$ in $\gamma$ is larger than $L$.
\end{itemize}
Choose constants $\rho, \eta>0$ small.
For any a horseshoe $\Lambda_N$ given by Proposition~\ref{p.split},
we have to prove the reverse doubling property for these constants.

We fix $x\in \Lambda_N$ and $r\in (0,\rho)$.
For each $\zeta\in B(x,\eta r)\cap W^u(x,\rho_0)\cap \Lambda'$
we consider the sequence $(i_{k})$ such that $f^{k}(\zeta)\in R_{i_{k}}$
for each $k\geq 0$ and then define the domain
$$\Delta_{\zeta,k}:=
f^{-k}(R_{i_k})\cap f^{-k+1}(R_{i_{k-1}})\cap\dots\cap R_{i_0}\cap W^u(x,\rho_0);$$
it splits in two pieces $\Delta_{\zeta,k,\pm}:=\Delta_{\zeta,k}\cap f^{-k}(R_{i_k,\pm})$.
We also consider $\Delta_\zeta$, the largest domain
$\Delta_{\zeta,k}$ contained in $B(x,r)$, and $\Delta_{\zeta,\pm}$,
the corresponding domains $\Delta_{\zeta,k,\pm}$.
Note that if $\rho$ has been chosen small, then $k$ is large.

Next choose $\gamma, I_-,I_+$ associated to $R_{i_k}$ and $W^{u}(f^k(\zeta),\rho_0)$ as above.
From the domination between $E^{uu}$ and $E^{c}$,
the domains $\Delta_{\zeta,k}$ and $\Delta_{\zeta,k,\pm}$ are contained in the unions
\begin{equation}\label{e.thickedcurve}
\Delta_{\zeta,k}\subset \bigcup_{z\in \gamma}f^{-k}(W^{uu}(z,\rho_0)),\quad
\Delta_{\zeta,k,\pm}\subset \bigcup_{z\in I_\pm}f^{-k}(W^{uu}(z,\rho_0)).
\end{equation}
The diameter of the strong unstable disk $f^{-k}(W^{uu}(z,\rho_0))$
is much smaller than the length of the curves $f^{-k}(\gamma)$, $f^{-k}(I_\pm)$.
Hence the domains $\Delta_{\zeta,k}$, $\Delta_{\zeta,k,\pm}$ are
strips that are very thin in the strong unstable direction.

From the distortion estimate,
the distance between $f^{-k}(I^-)$ and $f^{-k}(I^+)$
is larger than 
$$C^{-2}\frac L {\widehat L} \length(f^{-k}(\gamma)).$$
The holonomy along the strong-unstable foliation
between central curves inside the unstable plaques
is uniformly Lipschitz.
As a consequence there exists a uniform constant $\kappa>0$ such that
if $\gamma_k,\gamma_{k+1}$ are central curves
associated to the domains $\Delta_{\zeta,k}$ and $\Delta_{\zeta,k+1}$, then
we have
\begin{equation}\label{e.length}
\length(f^{-(k+1)}(\gamma_{k+1}))\geq (2\kappa)\length(f^{-k}(\gamma_{k})).
\end{equation}
By maximality of $\Delta_\zeta\subset B(x,r)$, the length of the associated curve
$f^{-k}(\gamma)$ is thus larger than $\kappa r$.
It follows that the distance between $f^{-k}(I^-)$ and $f^{-k}(I^+)$
is larger than $C^{-2} \kappa r L/\widehat L$.

If $\eta$ is chosen smaller
than $C^{-2}\kappa L/\widehat L$, then only one domain
$\Delta_{\zeta,-}$ or $\Delta_{\zeta,+}$ can intersect $B(x,\eta r)$.
Since $f$ has constant Jacobian along the unstable leaves for the measure $\mu^u$,
both preimages have the same $\mu^u$-measure. This proves that
$$\mu^u(\Delta_\zeta\cap B(x,\eta r))\leq \frac 1 2 \mu^u(\Delta_\zeta).$$
By construction the domains $\Delta_\zeta$ for $\zeta\in B(x,\eta r)\cap W^u(x,\rho_0)\cap \Lambda'$
are disjoint or equal. We thus obtain a finite family $Y$
such that the domain $\Delta_\zeta$, $\zeta\in Y$ are pairwise disjoint, and
their union contains $B(x,\eta r)\cap W^u(x,\eta r)\cap \Lambda'$
and is contained in $B(x,r)$. This proves that
\begin{equation*}
\begin{split}
\mu^u(B(x,\eta r)\cap W^u(x,\rho_0))&= \sum_{\zeta\in Y} \mu^u(\Delta_\zeta\cap B(x,\eta r))\\
&\leq \sum_{\zeta\in Y} \frac 1 2 \mu^u(\Delta_\zeta)\leq \frac 1 2 \mu^u(B(x, r)\cap W^u(x,\rho_0)).
\end{split}
\end{equation*}
This gives the desired estimate.
\end{proof}

\bigskip

\subsection{Extraction of sparse horseshoes and proof of Theorem~B }
We first prove the existence of a subhorseshoe with large entropy
and nice geometry.

\begin{theorem}[Sparse subhorseshoe]\label{t=ballselection}
For any diffeomorphism $f$, any horseshoe $\Lambda$
satisfying (H), and any $\varepsilon>0$, there exist $\chi, \lambda \in (0,1)$
with the following property.

For any $\rho>0$ there exist compact subsets $X_1,\ldots, X_n$
such that
\begin{itemize}
\item each $X_i$ has diameter in $[\lambda\rho,\rho]$ and meet $\Lambda$;
\item the neighborhoods $\widehat X_i$ of size  $\chi.\diam(X_i)$ of the $X_i$
are pairwise disjoint and homeomorphic to balls;
\item the entropy of the restriction of $f$ to the maximal invariant set in
$\bigcup_{i=1}^n X_i$ is larger than $h_{top}(\Lambda,f)-\varepsilon$.
\end{itemize}
\end{theorem}
The proof of this theorem is postponed until Section~\ref{ss.proof-extraction}
\begin{proof}[Proof of Theorem~B  from Theorem~\ref{t=ballselection}]
Theorem~B  will be proved in several steps:
we will first show that $\Lambda$ admits an extracted horseshoe that can be locally linearized; we 
then show that this subhorseshoe admits a dominated decomposition
into one-dimensional subbundles and a global chart where these bundles are constants.
At last we prove that after another extraction there is a subhorseshoe which is globally
linear.
\medskip

Fix $\varepsilon>0$.
Note that by Theorem~\ref{t.diagonalize},  we can perform if necessary an arbitrarily $C^r$-small perturbation and replace $\Lambda$
by a subhorseshoe whose topological entropy is arbitrarily close to the initial entropy in order to ensure that
$\Lambda$ has a dominated splitting $T_\Lambda M=E_1\oplus \dots \oplus E_d$ into one dimensional subbundles.

One can by a small $C^1$-perturbation replace $f$ by a diffeomorphism
that is $C^{\infty}$ in a small neighborhood of $\Lambda$.
This is more delicate in the conservative setting:
one uses~\cite{Z} for symplectic diffeomorphisms, \cite{DM} for volume preserving ones in the case $r>1$ and~\cite[Theorem 5]{A}  for volume preserving ones in the case $r=1$.

By an arbitrarily $C^r$-small perturbation we can also assume that
$E^{c}$ is essential, so that the property (H) holds.
Indeed if $p,q$ are two periodic
points close in $\Lambda$, it is possible to  break the intersection between the local stable manifold of $p$
and the local strong unstable manifold of $q$ (tangent to $E^{uu}(q)$); the intersection
between $W^u_{loc}(q)$ and $W^s_{loc}(p)$ belongs to $\Lambda\cap (W^{u}(q)\setminus W^{uu}(q))$, proving that $\Lambda\cap W^u(q)$ is not contained in $W^{uu}(q)$.
\medskip

Now apply Theorem~\ref{t=ballselection}, using $\varepsilon/4$: 
we obtain $\chi\in (0,1)$.
We also fix a global chart $\varphi\colon U\to \RR^d$ from a neighborhood $U$ of $\Lambda$,
so that we can work in $\RR^d$; if one considers a volume or a symplectic form,
one can take it to be constant in the chart.
For $\rho>0$ that will be chosen small enough, we get a family of compact sets
$X_1,\dots, X_n$ with diameter in $[\lambda.\rho,\rho]$
which meet $\Lambda$ and we introduce for each $X_i$
its $\chi.\diam(X_i)/2$-neighborhood $Q_i$ and its $\chi.\diam(X_i)$- neighborhood
$\widehat Q_i$.
The sets $\widehat Q_{i}$ are pairwise disjoint, have a diameter
smaller than $(1+2\chi).\rho$
and $\widehat Q_{i}$ contains the
$\frac \chi 4.\diam(Q_{i})$-neighborhood of $Q_{i}$.
\medskip

For each $i$, we choose a linear map $A_{i}$ in $GL(\RR,d)$ close to
the $Df(x)$ for $x\in \widehat Q_{i}$.
We will perturb $f$ in the sets $\widehat Q_i$ with the following lemma.

\begin{lemma}[Local linearization]\label{local-linearize}
For any $d\geq 1$ and any $\theta, \chi'>0$, there exists $\eta>0$
with the following property. Consider:
\begin{itemize}
\item a compact set $Q$ of the unit ball $B(0,1)\subset \RR^d$ which is homeomorphic
to a ball and has diameter equal to $1$,
\item a $C^\infty$-map $f\colon B(0,2)\to \RR^d$ which is a diffeomorphism onto its image
and satisfies $\|Df-\id\|<\eta$,
\end{itemize}
then,  there exists a $C^\infty$-map $g\colon B(0,2)\to \RR^d$ which is a diffeomorphism onto its image and:
\begin{itemize}
\item coincides with $f$ outside
the $\chi'$-neighborhood $\widehat Q$ of $Q$,
\item coincides with a translation $z\mapsto z+b$ in $Q$,
\item satisfies $\|Dg-Df\|\leq \theta$ everywhere.
\end{itemize}
If $f$ preserves the standard symplectic form or if $f$ preserves the standard Lebesgue volume
and satisfies $\|D^2f\|\leq \eta$, then $g$ can be chosen to preserve it as well.
\end{lemma}

\begin{proof}
Note that it is enough to prove the lemma for functions $f$ such that $f(0)=0$,
hence that are $C^1$-close to the identity.
Let us fix $d,\theta,\chi'$ and a compact set $Q_0\subset B(0,1)$ which is homeomorphic to a ball and
has diameter equal to $1$.
One chooses two disjoint, closed connected neighborhoods $\Delta_0,\Delta_1$ of
$\{y: d(y,Q_0)\geq \chi'\}$ and $Q_0$ respectively.

We consider a smooth map $\sigma\colon \RR^d \to [0,1]$ that coincides with $0$
on $\Delta_0$ and with $1$ on $\Delta_1$.

For any $C^\infty$-map $f\colon B(0,2)\to \RR^d$ which is a diffeomorphism onto its image,
one defines $g=\sigma\cdot f + (1-\sigma)\id$. Since $f(0)=0$,
if $Df$ is close to the identity, $g$ is a diffeomorphism onto its image
and $\|Dg-Df\|\leq \theta$ everywhere.
It coincides with $f$ on $\Delta_0$ and with the identity on $\Delta_1$,
hence satisfies the conclusions of the lemma.

When $f$ preserves Lebesgue, and $\|D^2f\|$ is small,
one builds a first diffeomorphism $g_0$ as before, which is $C^2$-close to the identity.
In particular the volume form $\widetilde m:=(g_0)_* \vol$ is $C^1$-close to the standard Lebesgue
volume $\vol$ and coincides with it on $g_0(\Delta_0)$ and on $g_0(\Delta_1)$.
By applying~\cite{Moser}, one obtains a conservative diffeomorphism
$h$ of $\RR^d$ which sends $\widetilde \vol$ on $\vol$.
Consequently the diffeomorphism $g:=h\circ g_0$ preserves the volume.
Since $h$ is obtained by integration of the form $\widetilde m-m$,
one can check that $h$ is a translation on each connected component of $g_0(\operatorname{Interior}(\Delta_0\cup \Delta_1))$ and can be chosen to coincide with the identity outside $B(0,1)$.
Since $m$ and $\widetilde m$ are $C^1$-close, $\|Dh-\id\|$ is small,
hence $g$ is $C^1$-close to the identity as well.
Since $Q_0$ is contained in a connected component of
$\operatorname{Interior}(\Delta_1)$, the diffeomorphism $g$ is a translation on $Q_0$
by construction.

In the symplectic case, $g$ is built in a different way.
The diffeomorphism $f$ is obtained from a generating function $h_f$ that is $C^2$-close to the generating function $h_I$ of the identity (since $f$ is $C^1$-close to the identity).
The map $g$ can be obtained from the generating function
$\sigma.h_f+(1-\sigma).h_I$.

Note that the triple $(\Delta_0,\Delta_1,\eta)$ is still valid for any connected compact set $Q$
that is Hausdorff close to $Q_0$.
Since the set of compact subsets of $B(0,1)$ whose diameter is equal to $1$
is compact in the Hausdorff topology,
there exists $\eta>0$ which is valid for all compact sets with diameter equal to $1$.
\end{proof}

Lemma~\ref{local-linearize} may be restated as follows.
\begin{corollary}\label{c.local-linearize}
For any $C^\infty$-diffeomorphism $f$ of the ball $B(0,2)\subset \RR^d$,
and any $\theta, \chi'>0$, there exists $\rho,\eta>0$
with the following property.

For any compact set $Q\subset  B(0,1)$ homeomorphic to the ball and with diameter smaller than $\rho$,
for any $A\in GL(\RR, d)$ such that $\|A-Df(x)\|<\eta$ for some $x\in Q$,
there exists $b\in \RR^d$ and a $C^\infty$-diffeomorphism $g$ which:
\begin{itemize}
\item coincides with $f$ outside
the $\chi'.\diam(Q)$-neighborhood $\widehat Q$ of $Q$,
\item coincides with $z\mapsto A\cdot z+b$ in $Q$,
\item satisfies $\|Dg-Df\|\leq \theta$ everywhere.
\end{itemize}
If $f$ and $A$ preserve a given volume or symplectic form, then $g$ can be chosen to preserve it as well.
\end{corollary}
\begin{proof}
It is enough to apply the previous lemma
to the diffeomorphism $\widetilde f\colon z\mapsto \frac 1 r (A^{-1}\circ f)(r\cdot (z+z_0))$,
where $z_0$ belongs to $Q$ and $r=\diam(Q)$.
If $\rho,\eta$ have been chosen small enough, $D\widetilde f(x)$ is close to the identity for $x\in B(0,2)$.

When $f$ preserves a given volume or symplectic form, one can change the coordinates so that
the volume or the symplectic form coincide with the standard ones.
Moreover one notices that $\|D^2\widetilde f\|$ uniformly converges to $0$ on $B(0,2)$
when the diameter $r$ goes to $0$.
The previous lemma can thus still be apply
in these settings provided $\rho$ is small enough.
\end{proof}

Applying the previous corollary independently for each $Q_i$
provides us with a diffeomorphism $g$ which coincides with an affine map
$z\mapsto A_{i}\cdot z+b_i$ on each $Q_{i}$ and with $f$ outside the union of the
$\widehat Q_i$.
The tangent maps of $f$ and $g$ are $C^0$-close, and since the size of the
connected components of the support of the perturbation is smaller than $(1+2\chi).\rho$,
the $C^0$ distance between $f$ and $g$ is also small.
Consequently $g$ belongs to the neighborhood $\cU$ of $f$.
Moreover if $f$ preserves a volume or a symplectic form,
by choosing the linear maps $A_i$ to be conservative, the perturbation $g$ still preserves the form
and is locally affine in the union $V:=\bigcup_{i}Q_{i}$.

Consider a transitive hyperbolic set $\Lambda'$ with entropy larger than $h_{top}(\Lambda, f)-\varepsilon/4$
and contained in $\bigcup_i X_i$ as given by Theorem~\ref{t=ballselection}.
Since $f$ and $g$ are $(1+2\chi).\rho$-close in the $C^0$ topology, the shadowing lemma
implies that the hyperbolic continuation $\Lambda'_g$ of $\Lambda'$ for $g$ is contained in the
$\gamma\rho$-neighborhood of $\Lambda'$, where $\gamma$ is arbitrarily close to $0$
if the distance between $Df$ and $Dg$ is chosen small enough. In particular,
$\Lambda'_g$ is contained in the union $V$, hence is locally affine.
Moreover $\Lambda'_g$
has entropy larger than $h_{top}(\Lambda,f)-\varepsilon/4$.
By Theorem~\ref{t.katok}, there exists a horseshoe $\widetilde \Lambda$
with entropy larger than $h_{top}(\Lambda,f)-\varepsilon/2$
contained in $V$, hence locally affine as required.
This gives the first part of Theorem~B .
\bigskip

We now explain how to modify the previous construction so that the splitting
$T_{\widetilde \Lambda}M=E_1\oplus\dots\oplus E_d$
 is locally constant.
We first introduce a family of disjoint open sets $\Delta_{1},\dots,\Delta_{s}$ which
cover $\Lambda$ and have small diameters.
We choose a point $x_k\in \Delta_{k}\cap\Lambda$ in each of them.
The sets $X_1,\dots,X_n$ given by  Theorem~\ref{t=ballselection}
are only chosen after,
with diameter small enough so that
each set $Q_{i}$ and each image $f(Q_{i})$ is contained in one of the sets $\Delta_k$.
We linearize in each domain as before but require that
$g$ coincide on $Q_{i}$ with the affine map
$z\mapsto A_{i}\cdot z+b_i$ where
$A_{i}=\Pi_{i,k,\ell}\circ Df(x_k)$, such that
\begin{itemize}
\item $k$ and $\ell$ are defined by the conditions $Q_{i}\subset \Delta_{k}$ and $f(Q_{i})\subset \Delta_{\ell}$,
\item the linear maps $\Pi_{i,k,\ell}$ is close to the identity and sends the splitting
$E_1(x_k)\oplus\dots\oplus E_d(x_k)$ to the splitting
$E_1(x_\ell)\oplus\dots\oplus E_d(x_\ell)$.
\end{itemize}
When $f$ preserves the volume or the symplectic form, we choose
$\Pi_{i,k,\ell}$ to preserve it as well. In the symplectic case this is possible
since the two planes of the form $E_m\oplus E_{d-m}$ are pairwise symplectic-orthogonal
(see~\cite{BV-Howfrequently}).

The end of the construction is unchanged.
After perturbation, the dominated splitting on $\widetilde \Lambda$
for the map $g$,
coincides in each set $\Delta_{k}$ with
$E_1(x_{k})\oplus\dots\oplus E_d(x_{k})$ and hence  is locally constant.
\bigskip

At this step we have reduced the proof of Theorem~B  to the case of a local diffeomorphism $f$ of 
a hyperbolic set $\Lambda$ in a subset $U$ of $\RR^d$, such that
$f$ is locally affine on $U$, and the splitting of $\RR^d$ into the coordinates axes
is $f$-invariant.
In the conservative case, the volume, or the symplectic form is chosen to coincide with the standard
Lebesgue volume of $\RR^d$
or with the standard symplectic form of $\RR^{2\times \frac d 2}$.
It remains to prove that after a new extraction and
a new perturbation, the linear part can be made constant.

Choose a Markov partition of $\Lambda$ into small disjoint rectangles
$R_0,\dots,R_\ell$. Since the $R_i$ are small, the diffeomorphism $f$ is affine on a neighborhood of $R_i$: there exists a linear map
$A_i$ such that the diffeomorphism $f$ has the form
$z\mapsto A_i.(z-x)+f(x)$.
Since the coordinates axes are preserved, $A_i$ is diagonal
with diagonal coefficients $a_{i,1},\dots,a_{i,d}$.

For $N\geq 1$ large enough, consider the sub-horseshoe
$\Lambda_N$ of points of $\Lambda$ which visit $R_0$ exactly every $N$ iterates:
its topological entropy is larger than $h_{top}(\Lambda,f)-\varepsilon/4$.
Moreover the return map $f^N$ on $R_0\cap \Lambda_N$
decomposes into branches labelled by compatible itineraries in $R_1,\dots,R_\ell$,
which are affine. Each itinerary is associated to
a diagonal matrix whose coefficients are sums of $N$ numbers chosen among
$a_{i,k}$, with $0\leq i\leq \ell$ and $1\leq k\leq d$.
The number of such matrices grows at most polynomially in $N$.
On the other hand, the number of itineraries of length $N$
starting in $R_0\cap \Lambda_N$ grows faster than
$e^{N(h_{top}(\Lambda,f)-\varepsilon/4)}$.

It follows from the pigeonhole principle that there exists a diagonal matrix $A$
and a set containing at least $e^{N(h_{top}(\Lambda,f)-\varepsilon/2)}$
itineraries of length $N$ which start in $R_0\cap \Lambda_N$ and 
which all have the same associated diagonal matrix $A^N$.
The set of points in $\Lambda_N$ whose orbit follows these itineraries
is a sub horseshoe $\widetilde \Lambda$ whose topological entropy is larger
than $h_{top}(\Lambda,f)-\varepsilon/2$ as required.
It may be decomposed as a disjoint union of compact sets:
$\widetilde \Lambda=K\cup f(K)\cup\dots\cup f^{N-1}(K)$.
The diffeomorphism $f^N$ on $K$ is affine with the constant linear part $A^N$.

Now cover $K$ by small disjoint open sets $\Delta_1,\dots,\Delta_m$
is such a way that $f^\ell$ is affine on each $\Delta_j$, with $0\leq \ell\leq N$
and $1\leq j\leq m$, and has a linear part $B_{j,\ell}$.
We change the chart on each set $f^\ell{\Delta_j}$
by the composition of $A^\ell B_{j,\ell}^{-1}$ with a translation.
For this new chart, the diffeomorphism $f$ has the form
$z\mapsto A.(z-x)+f(x)$ near each point $x\in \widetilde \Lambda$.
This completes the proof of  Theorem~B .
\end{proof}

\bigskip

\subsection{Cube families}
In this subsection and the following ones we prove Theorem~\ref{t=ballselection}.
We thus consider a diffeomorphism $f$ and a horseshoe $\Lambda$
satisfying (H). We describe some preliminary constructions.

\subsubsection{The affine chart $\varphi$ - the scale $\rho_0$ - the metric on $\Lambda$.}
For each point $x\in \Lambda$, we will use smooth charts
$\phi_x$ from a neighborhood $U_x$ of $x$ in $M$ to $\RR^d$
whose derivative $D_{x}\phi_i$ at $x$ sends the splitting
$E^{uu}(x)\oplus E^{c}(x)\oplus E^{s}(x)$ to the splitting
\begin{equation}\label{e=rsplit}
\RR^d = \RR^{d_u-1} \oplus \RR \oplus \RR^{d_{s}}.
\end{equation}
Since the angle between the spaces $E^{uu}$, $E^c$, $E^s$ is bounded away from zero,
one can assume that the $C^1$-norm of the charts $\phi_x$ is uniformly bounded by some $D>0$.

In the case $f$ preserves a volume $m$ or a symplectic form $\omega$,
we may assume that $\phi_*m$ (resp. $\phi_*\omega$) is the standard symplectic form
on $\RR^{2\times \frac d2}$. Indeed, for the volume preserving case
one uses~\cite{DM}.
In the symplectic case, we first use Darboux's theorem in order to rectify the
symplectic form. By~\cite{BV-Howfrequently}, the horseshoe $\Lambda$ admits
a finer dominated splitting into Lagrangian subbundles
$$T_\Lambda=E^{uu}\oplus E^c_1\oplus E^c_2\oplus E^{ss}$$
such that $E^c=E^c_1$, $E^s=E^c_2\oplus E^{ss}$, $\dim(E^{uu})=\dim(E^{ss})$
and $\dim(E^c_1)=\dim(E^c_2)$.
The subspaces $E^{uu}\oplus E^{ss}$ and $E^c_1\oplus E^c_2$ are symplectic orthogonal.
All this implies that the splitting $T_xM=E^{uu}(x)\oplus E^{c}(x)\oplus E^{s}(x)$
may be sent by a bounded symplectic linear map to the standard splitting~\eqref{e=rsplit}.

Since $\Lambda$ is totally disconnected,
for any scale $\rho_0>0$ there exist a neighborhood $U$
of $\Lambda$ having finitely many connected components $U_k$,
each of them has diameter smaller than $\rho_0$, contains
a point $x_k\in\Lambda$, and is included in the domain of the chart
$\phi_{x_k}$. We may also assume that the images $\phi_{x_k}(U_k)$ are pairwise disjoint,
so that the map $\varphi\colon U\to \RR^d$ which coincides with $\phi_{x_k}$ on $U_k$
is also a chart.
\medskip

Two metrics appear: the initial metric on $M$
and the chart metric induced by $\varphi$ from the standard metric in $\RR^d$.
At this stage, the chart $\varphi$ depends on $\rho_0$ and has not been fixed.
However the initial and charts metrics are comparable up to the constant $D$ which does not depend on $\varphi$. In the following we will mainly consider the chart metric.
\medskip

\subsubsection{First extraction - the Lyapunov exponents - the split Markov partition.}
\label{ss.firstextraction}
Replacing if necessary $\Lambda$ by a subhorseshoe whose entropy is arbitrarily close to the entropy of the initial horseshoe, one can assume that:
\begin{itemize}
\item (Theorem~\ref{t.katok}.) There exist constants
$-\lambda \leq -\nu < 0 <  \gamma < \widehat\gamma < \widehat\nu \leq \widehat\lambda$
such that for any invariant probability measure on $\Lambda$ the Lyapunov exponents
along $E^{uu}$, $E^{c}$, and $E^{s}$ belong to $(\widehat\nu,\widehat\lambda)$,
$(\gamma, \widehat\gamma)$ and $(-\lambda, -\nu)$ respectively.

One chooses $N_0$ large enough so that for any $N\geq N_0$,
if $u$ is a unit vector in $E^{uu}$, $E^{c}$ or $E^{s}$, then
$\frac 1 N \log \|Df^N(u)\|$ belongs to the corresponding interval.

\item (Proposition~\ref{p.split}.)
For any $\varepsilon'>0$,
there are subhorseshoes that are disjoint unions of the form
$\Lambda_N=K\cup\dots f^{N-1}(K)$ with $N\geq N_0$
and entropy larger than $h_{top}(\Lambda,f)-\varepsilon'$,
and there exists a split Markov partition $\{R_{1,\pm},\dots R_{m,\pm}\}$
whose rectangles have a small diameter
and such that $\mu^u$ gives the same weight to $R_{i,-}$ and $R_{i,+}$
inside each unstable plaque.
\end{itemize}

\subsubsection{Sheared cubes - the shear $\sigma$ - the scale $2^{-k}$.}\label{ss.cube}
Let $\{e_1,\dots,e_{(d_u-1)}\}$, $\{e_{d_u}\}$, $\{f_1,\dots,f_{d_s}\}$
be orthonormal bases for the three factors of the decomposition
\eqref{e=rsplit}.  For a shear $\sigma\in (0,1)$ we define the linear transformation
$L_\sigma\colon \RR^d\to \RR^d$ by
$L_\sigma(e_j)=e_j$ if $1\leq j<d_u$, $L_\sigma(f_j)=f_j$ if $1\leq j\leq d_s$ and
$$L_\sigma(e_{d_u})=\sum_{j=1}^{d_u} e_j+ \sigma\cdot \sum_{j=1}^{d_s} f_j.$$
We denote by $P_\sigma$ the image $L_\sigma(P)$ of the unit parallelepiped centered at $0$:
$$P=\{t_1e_1+\dots+t_{d_u}e_{d_u}+t_{(d_u+1)}f_1+\dots+t_df_{d_s},\;
t_i\in[-1/2,1/2]\}.$$
The reason we need a shear will appear
in Lemmas~\ref{l.density} and~\ref{l.strip}.

\begin{figure}[h]
\includegraphics[scale=0.25]{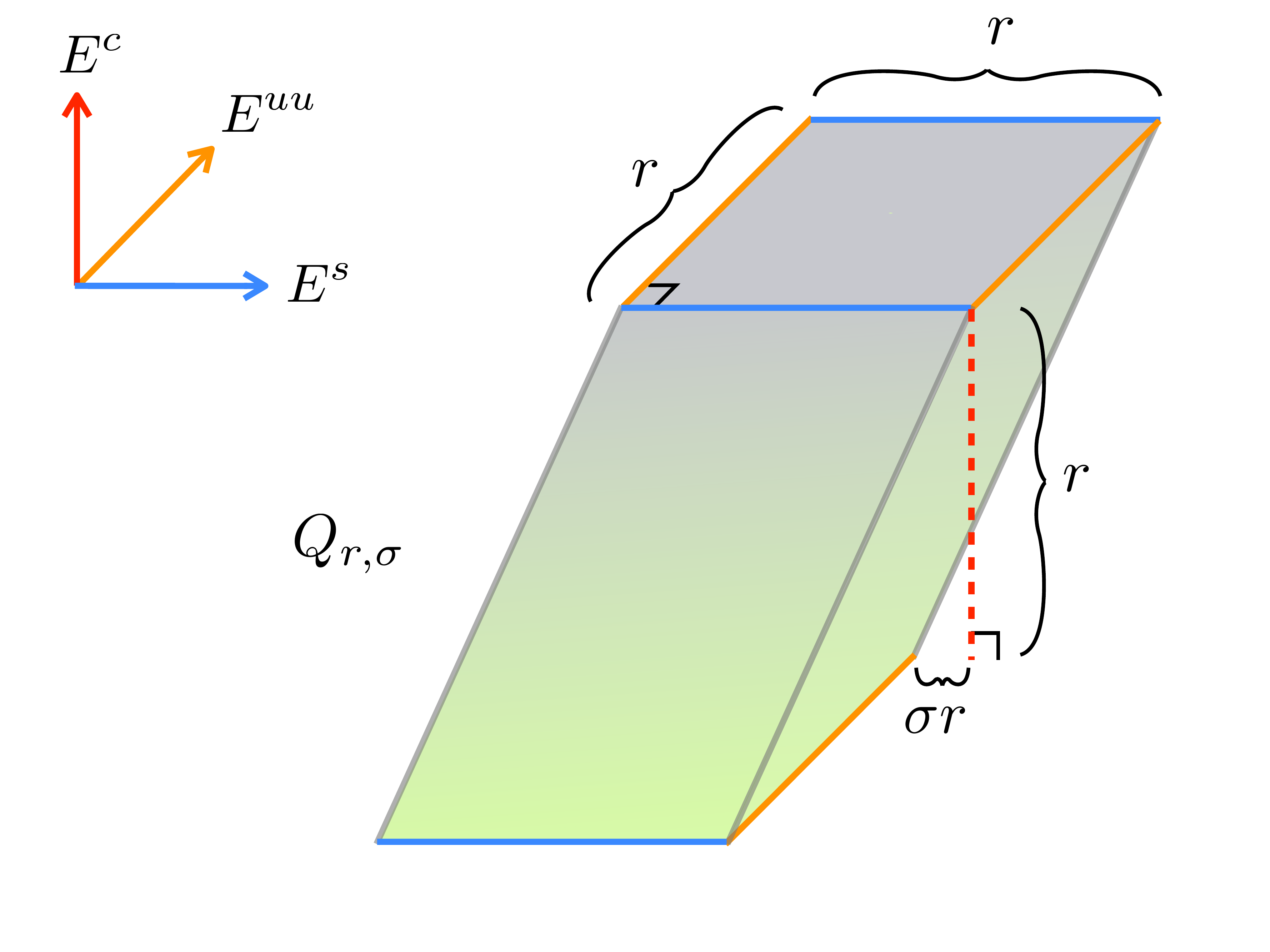}
\caption{\label{f.cube} The dimensions of a cube $Q_{r,\sigma}(y)$ of diameter $r$ and shear $\sigma$.}
\end{figure}

The \emph{unstable faces} are the faces of $P_\sigma$ spanned by the same vectors as $P_\sigma$,
but one among $e_1,\dots,e_{d_u-1},L_\sigma(e_{d_u})$.
The other faces are the \emph{stable faces} and are
spanned by the same vectors as $P_\sigma$, but one among $f_1,\dots,f_{d_s}$.
The \emph{unstable boundary} $\partial^uP_\sigma$ (resp. \emph{stable boundary}) is the union of the unstable (resp. stable) faces.

Fix $\rho_1\ll\rho_0$. For $r\in (0,\rho_1)$  and $y\in U$ at distance less than $\rho_1$ from $\Lambda$, we define the \emph{(sheared) cube} (see Figure~\ref{f.cube})
$$Q_{r,\sigma}(y) = \varphi^{-1}\left( r P_\sigma + \varphi(y)\right) =\varphi^{-1}\left( r L_\sigma P + \varphi(y)\right).
$$
The point $y$ is called the {\em center}  of the {\em cube} $Q_{r,\sigma}(y)$ and $r$ is its {\em diameter}.  For any cube $Q$ with center $y$ and diameter $r$, and any $\kappa \in (0,2)$,
we denote by $\kappa Q$ the cube centered at  $y$ of diameter $\kappa r$.
The stable and unstable boundaries of $Q$ are defined analogously to the boundaries of $P_\sigma$.

For $\eta \in (0,1)$, we define linear transformations $H^s_\eta, H^u_\eta: \RR^d\to \RR^d$
by $H^s_\eta(f_i) = (1-\eta) f_i$, $H^s_\eta (e_i) = e_i$, $H^u_\eta(f_i) = f_i$ and $H^u_\eta(e_i) = (1-\eta) e_i$.
For $Q = Q_{\rho,\sigma}(y)$, we define the {\em stable  and unstable $\eta$-boundaries} of
$Q$ as follows:
$$\partial^s_\eta Q=Q\setminus \varphi^{-1}\left( r L_\sigma H^s_\eta P + \varphi(y)\right)
\text{ and }\partial^u_\eta Q=Q\setminus \varphi^{-1}\left( r L_\sigma H^u_\eta P + \varphi(y)\right).$$
This partitions $Q$ into $(1-\eta)  Q$ and $\partial^s_\eta Q \cup \partial^u_\eta Q$.
Note also that $\partial ^s Q=\bigcap_{\eta>0} \partial ^s_\eta Q$.

Each cube $Q = Q_{r,\sigma}(y)$ has a set of  {\em unstable neighbor cubes} $\cN^u(Q)$
of cardinality $3^{d_u}$.
The set $\cN^u(Q)$ consists of the cubes of diameter $r$ and shear $\sigma$
that are produced from $Q$
by a translation by
$$r\;L_\sigma\left(n_1\;e_1\;+\;\dots\;+\;n_{d_u-1}\;e_{d_u-1}\;+\;n_{d_u}\;e_{d_u} \right),$$
where each $n_i$ is taken in $\{-1,0,1\}$ (see Figure~\ref{f.cube-neighbor}).

\begin{figure}[h]
\includegraphics[scale=0.25]{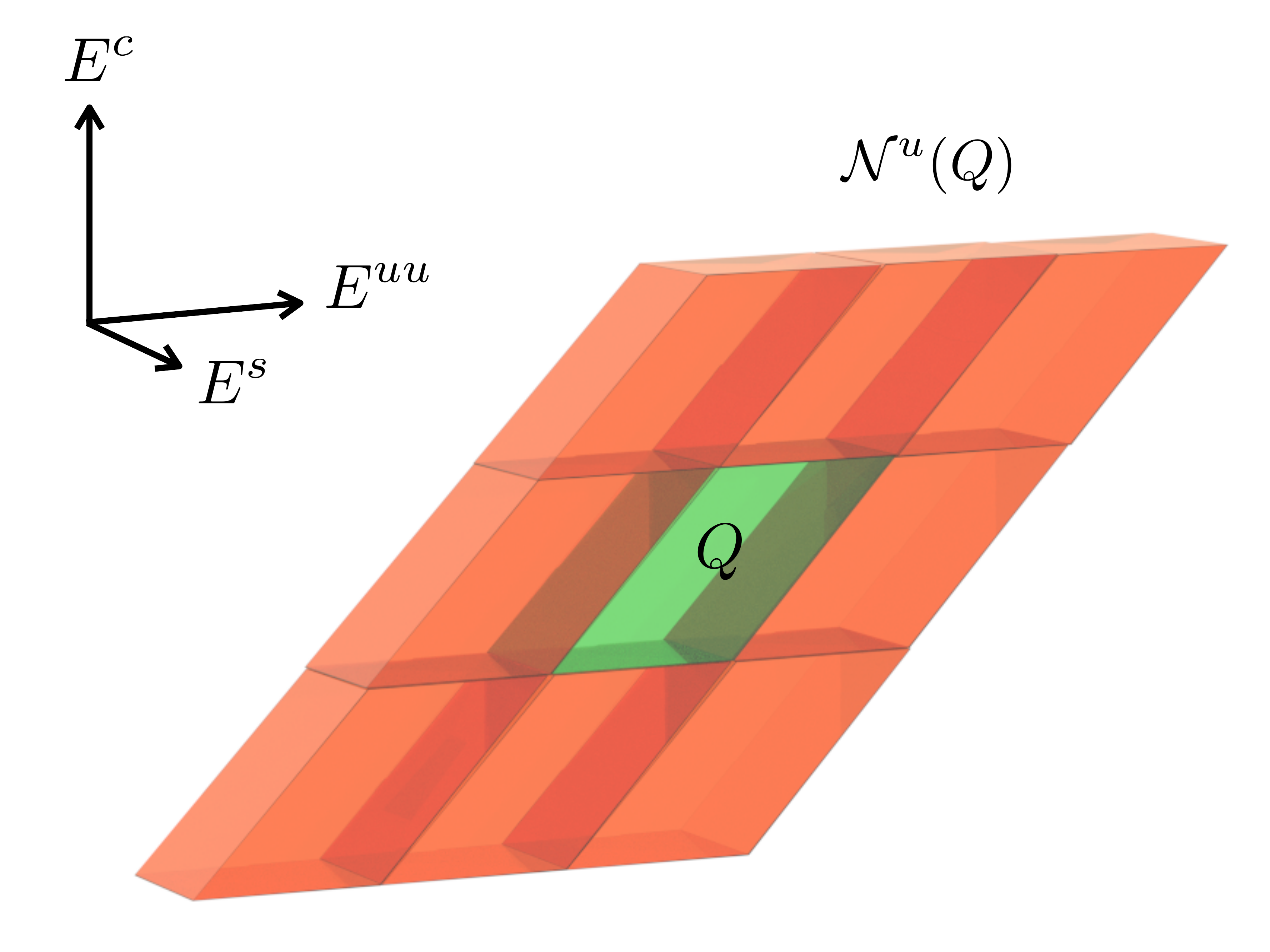}
\caption{\label{f.cube-neighbor} The unstable neighbor cubes of $Q$ (which include $Q$ itself).}
\end{figure}

If $\rho_0$ is small, then the local unstable manifolds $W^u(z,\rho_0)$
are $C^1$-close to planes spanned by $e_1,\dots,e_{d_u}$, so that
if $(1-3\sigma)  Q\cap W^u(z,\rho_0)\neq \emptyset$,
then the local unstable manifold
$W^u(z,\rho_0)$ does not meet the stable boundary of the $Q'\in \cN^u(Q)$,
and
$$
3  Q\cap W^u(z,\rho_0) = \bigcup_{Q'\in \cN(Q)} Q'\cap W^u(z,\rho_0).
$$

Choose $k>0$ large.
We construct the {\em cube family}  $\cQ = \cQ_{\varphi,k,\sigma}$
(at scale $2^{-k}$ with shear $\sigma$) as the collection:
$$\cQ_{\varphi,k,\sigma} =
\left\{ Q_{2^{-k},\sigma } (x) : 
\quad d(x, \Lambda)<\rho_1 \text{ and }
\varphi(x)\in2^{-k}\; L_\sigma(\ZZ^d)\right\}.$$\medskip

\subsubsection{Cube transitions - the boundary size $\beta$.}
Fix some $\beta>0$ close to $0$.
For $Q,Q'\in \cQ $, we say that \emph{there is a transition from $Q$ to $Q'$} (which we denote by
$Q\rightarrow Q'$) if
$f^N(Q)$ intersects $Q'$, whereas
$Q'\cap f^N(\partial^u_{\beta/2} Q)$ and $\partial^s _{\beta/2} Q' \cap f^N(Q)$
are empty (see Figure~\ref{f.sbad}).
\medskip

In the following, we consider the disintegration $\mu^u$ of the measure of maximal entropy
of a horseshoe $\Lambda_N$ along the unstable manifolds. We then define the measure $\mu^u_x$ induced by $\mu^u$
on the plaque $W^u(x,\rho_0)$ for each $x\in K$.
We will reduce the proof of Theorem~\ref{t=ballselection} to the following proposition, which is proved in Section~\ref{ss.proof-selectcube}.

\begin{proposition}\label{p=selectcubes}
Consider $f$ and $\Lambda$ as introduced in Section~\ref{ss.firstextraction}.

For all $\varepsilon > 0$, there exist $\rho_0,\sigma,\beta,N_0>0$ and a chart $\varphi\colon U\to \RR^d$ with $\Lambda\subset U$,
such that if $\Lambda_N\subset \Lambda$
is a subhorseshoe associated to an integer $N\geq N_0$ and if $\mu^u_x$ denotes the measure induced on $W^u(x,\rho_0)$ by the disintegration of its measure of maximal entropy along the unstable leaves, then the following holds.

There exists $k_0$ such that for all $k\geq k_0$, any cube $Q$ in the family $\cQ = \cQ_{\varphi,k,\sigma}$
of the chart $\varphi$ and any point $x\in \Lambda_N\cap (1-2\beta)  Q$, we have:
$$
\sum_{Q \rightarrow Q' } \mu^u_x\left( f^{-N}((1-2\beta)  Q')\right ) 
\geq e^{-\varepsilon} \cdot \mu^u_x\left((1-\beta)  Q\right).
$$
\end{proposition}
\bigskip

\subsection{Proof of Theorem~\ref{t=ballselection} from Proposition~\ref{p=selectcubes}}
\label{ss.proof-extraction}
Consider a horseshoe $\Lambda$ as in the assumptions of Theorem~\ref{t=ballselection}
and $\varepsilon>0$.
Fix a Markov partition $R_0,\dots,R_s$ of $\Lambda$
such that any point $x\in \Lambda$ is uniquely determined by its
itinerary in the collection of rectangles $R_i$.

 Proposition~\ref{p=selectcubes} applied with this value of $\varepsilon$
provides us with a chart $\varphi\colon U\to \RR^d$ satisfying $\Lambda\subset U$,
with a boundary size $\beta>0$ and a shear $\sigma$.
We also obtain an arbitrarily large integer $N\geq 1$
and a subhorseshoe $\Lambda_N$
such that $h_{top}(\Lambda_N,f)>h_{top}(\Lambda',f)-\varepsilon/2$
and which decomposes as a disjoint union $\Lambda_N=K\cup f(K)\cup\dots\cup f^{N-1}(K)$.

The initial and chart metrics are equivalent up to a uniform constant.
One can thus end the proof with the initial metric on $M$.
We will choose $\chi>0$ such that for any cube $Q$ in a family
$\cQ_{\varphi,k,\sigma}$, the $\chi.\diam(Q)$-neighborhood of the cube $(1-\beta/2)Q$
is contained in $Q$.
There exists also $\lambda\in (0,1)$ such that
for any $k$ large and any two cubes $Q,Q'\in \cQ_{\varphi,k,\sigma}$,
the quantity $10\lambda\diam(Q)$ is bounded by $\diam(Q')$.
In this way, for any $\rho>0$ sufficiently small, we can choose $k\geq 1$ large such that the cubes
$Q\in \cQ=\cQ_{\varphi,k,\sigma}$ have diameter in $[\lambda.\rho,\rho]$.
Note also that all points in the same cube have the same itinerary during $N$ iterates
with respect to the Markov partition.
Let $X_1,\dots,X_n$ be the collection of cubes $(1-\beta/2)  Q$ for $Q\in \cQ$ that meet $K$
and let $\widehat X_i$ denote the $\chi\rho$-neighborhood.
If $\rho$ has been chosen sufficiently small, then
the sets $f^\ell(\widehat X_j)$ for $0\leq \ell<N$ and $1\leq j\leq n$ are pairwise disjoint.

Let $V$ be the union of the cubes
$X_1,\dots,X_n$,
and let $V_\ell=\bigcap_{j=1}^{\ell} f^{-j.N}(V)$.
Fix any point $x\in K$ and $Q\in \cQ$ such that
$x\in (1-2\beta)  Q$. Since $\mu^u$ has full support, we have
$$\mu^u_x((1-\beta)  Q)>0.$$

Note that if there is a transition $Q\rightarrow Q'$ then for any point
$x'\in K$ belonging to $Q'\cap f^N(W^u(x,\rho_0))$
we have $W^{u}(x',\rho_0)\subset f^N(W^{u}(x,\rho_0))$ and
moreover, the non-empty connected set
$f^{-N}(Q')\cap W^u(x,\rho_0)$ is contained in
$(1-\beta/2)  Q$.
Since $f$ has constant Jacobian along the unstable leaves for
the measures $\mu^u$, by applying inductively the proposition we obtain that for each $\ell\geq 1$,
\begin{equation*}
\begin{split}
\sum_{Q \rightarrow Q_1\rightarrow\dots\rightarrow Q_\ell}
\mu^u_x&\left( f^{-\ell N}({\textstyle (1-\frac \beta 2)}  Q_\ell)\cap
f^{-(\ell-1) N}((1-{\textstyle\frac \beta 2})  Q_{(\ell-1)})
\dots\cap (1-{\textstyle\frac \beta 2})  Q
\right )\\
&\geq e^{-\varepsilon\ell} \cdot \mu^u_x\left((1-\beta)  Q\right).
\end{split}
\end{equation*}

Hence
$$
\mu^u_x\left(V_\ell\cap (1-\beta)  Q)\right)>
e^{-\ell\varepsilon}\mu^u_x\left((1-\beta)  Q\right).
$$
Integrating over the different plaques $W^u(x,\rho_0)$
with $x\in (1-2\beta)  Q\cap \Lambda_N$,
there exists $C_1>0$ uniform in $\ell$ such that the measure of maximal entropy $\mu$
on $\Lambda_N$ satisfies:
\begin{equation}\label{e.measure-bound}
\mu\left(V_\ell\cap (1-\beta)  Q)\right)>
C_1e^{-\ell\varepsilon}.
\end{equation}

\begin{figure}
\includegraphics[scale=0.32]{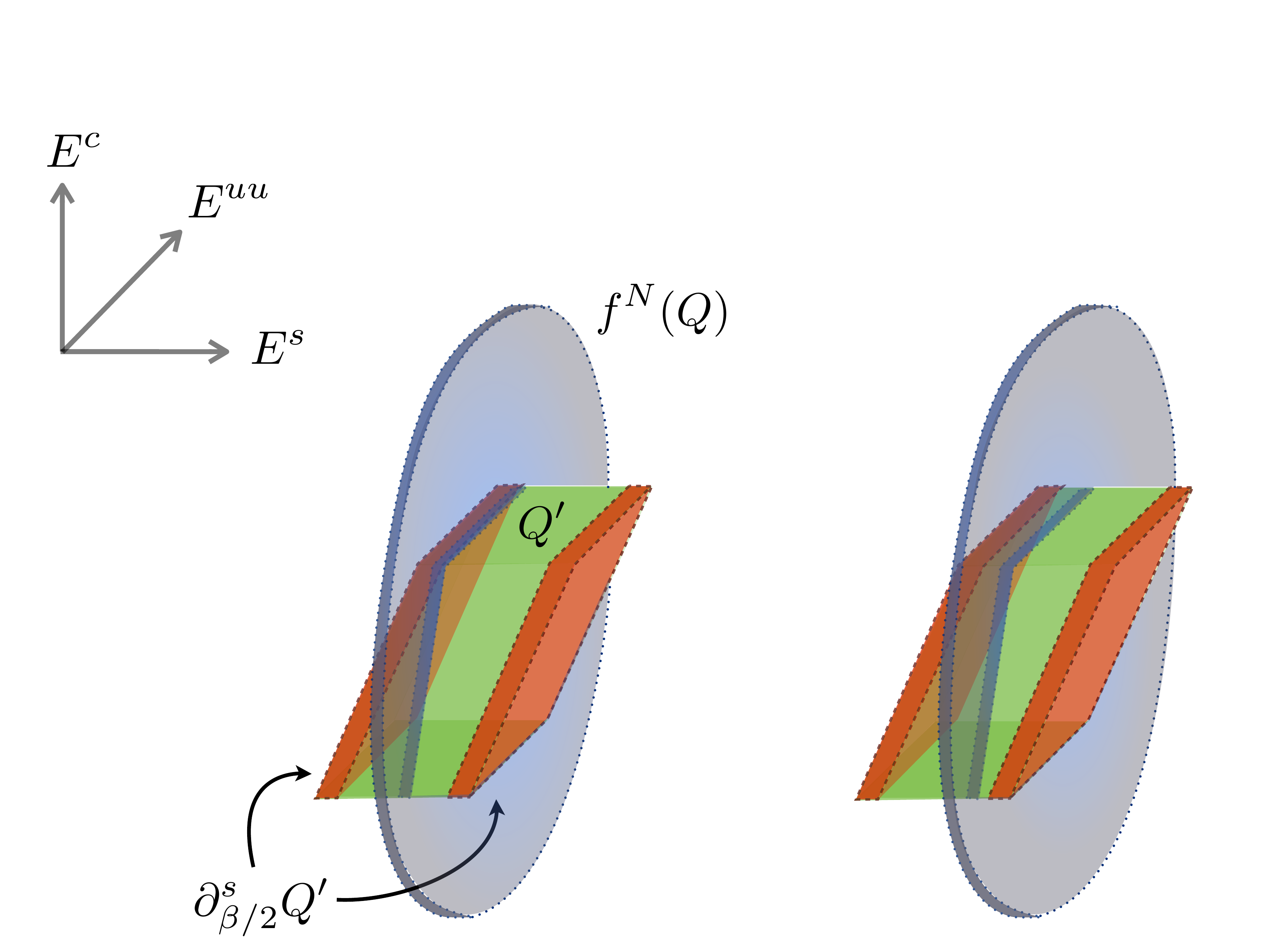}
\caption{\label{f.sbad} An s-bad cube $Q'$ (left) and a transition $Q\to Q'$ (right).}
\end{figure}

Since $\mu$ is the Gibbs state for the potential $\psi=0$ on $\Lambda_N$,
the measure of points that follow a fixed itinerary of length $q$ (with respect to the
Markov partition)
is smaller than $C_2\exp(-q\cdot h_{top}(\Lambda_N,f))$ for some $C_2>0$ uniform in $q$,
see~\cite{bowen}.
Using~\eqref{e.measure-bound}, we deduce that the number
of different itineraries of length $q=\ell\cdot N$ starting from $(1-\beta)  Q$
and contained in $V_{\ell}$ is larger than $C_3\exp(q\cdot (h_{top}(\Lambda_N,f)-\varepsilon/N))$.
Passing to the limit as $q$ goes to $+\infty$  proves that the topological
entropy of the maximal invariant set in $V_\ell\cup f(V_\ell)\cup\dots\cup f^{N-1}(V_\ell)$
is larger than $h_{top}(\Lambda_N,f)-\varepsilon/N$, hence larger than $h_{top}(\Lambda,f)-\varepsilon$.

We have proved that, for each $\varepsilon>0$,
there is a decomposition $\Lambda=K\cup f(K)\cup\dots\cup f^{N-1}(K)$
such that the conclusion of the theorem holds for the horseshoe $K$ and the diffeomorphism
$f^N$.
Since $K, f(K),\dots, f^{N-1}(K)$ are pairewise disjoint, the conclusion holds also
for $\Lambda$ and $f$ by reducing $\lambda,\chi\in (0,1)$: considering
a family of connected cubes $X_1,\dots,X_n$ with small diameter associated to $K$ and $f^N$,
one gets a family $f^\ell(X_j)$, $0\leq \ell<N$, $1\leq j\leq n$, which satisfies
the required properties for $f$ and $\Lambda$.
This completes the proof of Theorem~\ref{t=ballselection}. \eproof
\medskip

\subsection{Proof of Proposition~\ref{p=selectcubes}}\label{ss.proof-selectcube}
The proof uses the chart metric.
Consider any horseshoe $\Lambda_N$ with $N$ larger than some $N_0\geq 1$,
a chart $\varphi$ whose connected components have  diameter smaller than some $\rho_0>0$,
a cube $Q\in \cQ=\cQ_{\varphi,\sigma,k}$ and any $x\in \Lambda_N\cap (1-2\beta)  Q$.

We say that a cube $Q'\in \cQ$ is \emph{s-bad} (with respect to the image $f^N(Q)$)
if  its stable boundary $\partial^s_{\beta/2} Q'$ intersects $f^N(Q)$
(see Figure~\ref{f.sbad}).

We will assume that  the images of the unstable $\beta$-boundaries
is larger than the size of the cubes:
\begin{equation}\label{e.boundbeta1}
\beta\exp(N_0\widehat \nu)>10.
\end{equation}

Any cube $Q'\in \cQ$ that intersects
$f^N((1-\frac 3 4\beta)  Q\cap W^u(x,\rho_0))\cap \Lambda_N$
satisfies one of the following cases:
\begin{itemize}
\item either there exists a transition $Q\rightarrow Q'$,
\item or $Q'$ is s-bad.
\end{itemize}
(See Figure~\ref{f.transitions}.)
Indeed $Q'$ has size $2^{-k}$ and by~\eqref{e.boundbeta1}
cannot intersect $f^N(\partial^u_{\beta/2}Q)$,
which is at distance larger than $\exp(N\widehat \nu)\frac \beta {10}2^{-k}$ from $f^N((1-\frac 3 4 \beta)  Q)$.
\medskip

\begin{figure}
\includegraphics[scale=0.38]{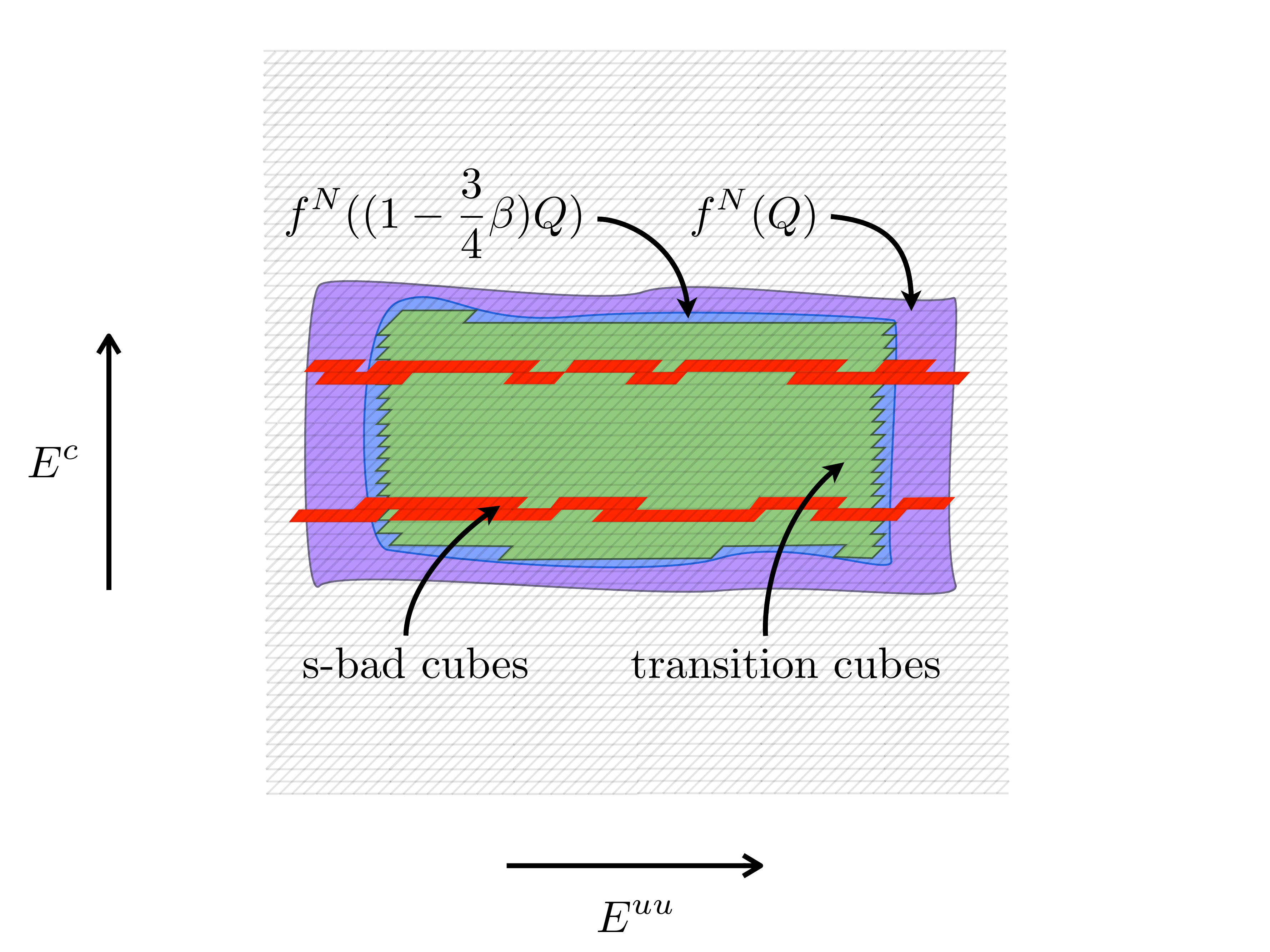}
\caption{\label{f.transitions} Transition cubes in $f^N(W^u(x,\rho_0))$.}
\end{figure}

In the following the main point is to bound the measure of s-bad cubes
that intersect $f^N((1-\beta)  Q\cap W^u(x,\rho_0))\cap \Lambda_N$
\medskip

\subsubsection{Measure of cube boundaries.} For the particular geometry of the cubes that we've defined,
the reverse doubling inequality~\eqref{e.reversedoubling} can be improved.

\begin{lemma}\label{l.density}
For every $\delta>0$, there is $\eta\in(0,1)$ such that for any $\sigma\in (0,1)$
the following property holds if $\rho_0$ is sufficiently small.

For any $N\geq N_0$, any $x\in\Lambda_N$ and
any cube $Q=Q_{r, \sigma}(y)$ such that  $(1-3\sigma)  Q$ and $W^u(x,\rho_0)$ intersect,
the measure $\mu^u_x$ induced on $W^u(x,\rho_0)$
by the disintegration of the measure of maximal entropy of $\Lambda_N$  satisfies:
$$\mu^u_x(Q\setminus (1-\eta)  Q)\leq
\delta \sum_{Q'\in \cN^u(Q)} \mu^u_x((1-\eta)  Q').$$
\end{lemma}
\begin{proof} The proof is similar to that of Theorem~\ref{t.doubling-uniform}.
We first introduce for each point $\zeta\in \Lambda_N$ and each $k\geq 1$
the domains $\Delta_{\zeta,k}$ and $\Delta_{\zeta,k,\pm}$
as intersections of the $k$-th backward image of rectangles
$f^{-k}(R_i)$ or $f^{-k}(R_{i,\pm})$ with the unstable plaque of $\zeta$.

\begin{sublemma}
For every $\eta_0>0$, there exists $\eta_1\in(0,\eta_0)$
such that for any $\sigma\in (0,1)$
the following property holds if $\rho_0$ is sufficiently small.

For any $N\geq N_0$, any cube $Q=Q_{r, \sigma}(y)$ and any $z\in (1-3\sigma)  Q\cap \Lambda_N$,
it holds that any point $\zeta$ in
$\Lambda_N\cap (1+\eta_1)  Q\setminus (1-\eta_1)  Q$
belongs to some preimage $\Delta_{\zeta,k}$ such that:
\begin{enumerate}
\item\label{l.density1} $\Delta_{\zeta,k}\subset (1+\eta_0)  Q\setminus (1-\eta_0)  Q$, and
\item\label{l.density2} $\Delta_{\zeta,k,-}$ or $\Delta_{\zeta,k,+}$ is contained in
a cube $(1-\eta_1)  Q'$ with $Q'\in \cN^u(Q)$.
\end{enumerate}
\end{sublemma}
\begin{proof}
Assuming that $\rho_0$ is small enough, the bundles
$E^s$, $E^{c}$, $E^{uu}$
(viewed in the charts) are $C^0$-close to constant bundles and
the unstable plaques are $C^1$-close to affine spaces.
The sets $\Delta_{\zeta,k}$ are 
 inside the unstable plaque of $z$,
which is stretched along a central curve $\gamma_k$ and very thin in the strong unstable direction (see~\eqref{e.thickedcurve}).

Since $z\in (1-3\sigma)  Q$, the plaque $W^u(z,\rho)$ intersects
any $Q'\in \cN^u(Q)$ along its unstable boundary. The intersection with each
unstable face of $Q'$ is transverse.
It follows that  the set $(1+\eta_0)  Q\setminus \bigcup_{Q'\in \cN^u(Q)} (1-\eta_1)  Q'$
is a union of $2 {d^u}$ thickened hypersurfaces $S_1,\dots,S_{2 {d^u}}$
of the unstable plaque of $z$ whose width along the central direction $E^{c}$
belongs to $[\eta_1r/2,2\eta_1r]$.

By~\eqref{e.length},
for any $\zeta\in \Lambda_N\cap (1+\eta_1)  Q\setminus (1-\eta_1)  Q$,
one can consider a domain $\Delta_{\zeta,k_0}$ associated to a central curve
$f^{-k_0}(\gamma_{k_0})$ of length contained in
$[2\eta_1rC^2\frac {\widehat L} L, \kappa^{-1}\eta_1rC^2\frac {\widehat L} L]$ (using the constants
$C,L,\widehat L,\kappa>0$ introduced in the proof of Theorem~\ref{t.doubling-uniform}).
The two domains $\Delta_{\zeta,k_0,+}$ and $\Delta_{\zeta,k_0,-}$ are associated to
subintervals $I^-,I^+\subset \gamma_{k_0}$ whose $k_0$-th preimages are separated by
$C^{-2}\frac L{\widehat L}\length(f^{-k_0}(\gamma_{k_0}))$, which is larger than $2\eta_1r$.
It follows that only one domain $\Delta_{\zeta,k_0,+}$ or $\Delta_{\zeta,k_0,-}$
can intersect each thickened hypersurface $S_i$.
By the definition of split Markov partition, either
$\Delta_{\zeta,k_0,+}$ or $\Delta_{\zeta,k_0,-}$ must contain both
$\Delta_{\zeta,k_0+1,+}$ and $\Delta_{\zeta,k_0+1,-}$.
We thus deduce that
in the family $\Delta_{\zeta,k_0}$, \dots $\Delta_{\zeta,k_0+2{d_u}}$,
there exists a domain $\Delta_{\zeta,k}$ such that
either $\Delta_{\zeta,k,-}$ or $\Delta_{\zeta,k,+}$ is disjoint from the thickened hypersurfaces
$S_1,\dots,S_{2 {d^u}}$, and hence is contained in a cube
$(1-\eta_1)  Q'$ with $Q'\in \cN^u(Q)$.

By~\eqref{e.length}, the larger domain $\Delta_{\zeta,k_0+2^{d_u}}$
is associated to a curve $\gamma_{k+2^{d_u}}$ whose length is smaller
than $\left(\frac 1 {2 \kappa}\right)^{2^{d_u}} \kappa^{-1}\eta_1rC^2\frac {\widehat L} L$.
It is thus contained in $(1+\eta_0)  Q\setminus (1-\eta_0)  Q$, provided
$\eta_1$ is chosen so that
$$\eta_0>\left(\frac {1} {2 \kappa}\right)^{2 {d_u}}  C^2\frac {\widehat L} L\eta_1.$$
\end{proof}

For $\eta_1<\eta_0$ as in the previous sublemma,
any point $\zeta$ belongs to a maximal set $\Delta_\zeta$ satisfying  conditions~\ref{l.density1}
and~\ref{l.density2}.
In particular, the domains $\Delta_\zeta$ are disjoint or equal, they cover
$\Lambda_N\cap (1+\eta_1)  Q\setminus (1-\eta_1)  Q$,  and they satisfy
$$  \mu_x^u(\Delta_{\zeta})\leq \frac 1 2 \sum_{Q'\in \cN^u(Q)} \mu^u_x(\Delta_{\zeta}\cap (1-\eta_1)  Q').$$
We obtain that
$$\mu^u_x((1+\eta_1)  Q\setminus (1-\eta_1)  Q)\leq
\frac 1 2 \sum_{Q'\in \cN^u(Q)} \mu^u_x((1-\eta_1)  Q'\cap (1+\eta_0)  Q
\setminus (1-\eta_0)  Q).$$
Applying the sublemma inductively, we construct a sequence
$0<\eta_k<\eta_{k-1}<\dots<\eta_0<1$ such that
$2^{-k}<\delta$ and
$$\mu^u_x((1+\eta_\ell)  Q\setminus (1-\eta_\ell)  Q)\leq
 2^{-\ell} \sum_{Q'\in \cN^u(Q)} \mu^u_x((1-\eta_\ell)  Q'\cap (1+\eta_{\ell-1})  Q
\setminus (1-\eta_{\ell-1})  Q),$$
for $\ell =  0 ,\ldots, k$. Thus for  $\eta=\eta_k$, we have:
$$\mu^u_x(Q\setminus (1-\eta)  Q)\leq
\delta \sum_{Q'\in \cN^u(Q)} \mu^u_x((1-\eta_1)  Q'),$$
which implies the conclusion of  Lemma~\ref{l.density}.
\end{proof}
\bigskip

\subsubsection{Localization of s-bad cubes.}
In order to control the image $f^N(Q)$ of a cube, we require that it be smaller than $\rho_1$
(and contained in the domain of the chart $\varphi$):
\begin{equation}\label{e.boundN}
\exp(N_0\widehat \lambda)\;2^{-k_0}<\rho_1,
\end{equation}
and that its diameter (in the chart) along the coordinate space $\{0\}^{d_u}\times \RR^{d^s}$ be less than $\frac \beta 2 2^{-k}$:
\begin{equation}\label{e.boundbeta2}
\exp(-N_0\nu)<\frac \beta 2.
\end{equation}
\medskip

We define a \emph{strong unstable strip} of an unstable plaque
$W^u(z,\rho_0)$, $z\in \Lambda$, as the region bounded by two strong unstable leaves
$W^{uu}(z_1,2\rho_0)$, $W^{uu}(z_2,2\rho_0)$ with $z_1,z_2\in W^u(z,\rho_0)$:
this is the set of points $\zeta\in W^u(z,\rho_0)$ that belong to a central curve
in $B(z,2\rho_0)$ joining $W^{uu}(z_1,2\rho_0)$ to $W^{uu}(z_2,2\rho_0)$.

Fix a continuous central cone field, that is, at each point $x$ close to $\Lambda$
a cone in $T_xM$ that is a small neighborhood of $E^{c}(x)$.
The distance between two different $W^{uu}$-leaves in $W^u(y,\rho_0)$
is the infimum of the length of curves joining the two leaves and tangent to the
central cone field. This allows us to define the width of a strong unstable strip and the distance between two strong unstable strips.
Note that the strong unstable manifolds $W^{uu}$  form a $C^{1+\alpha}$ subfoliation of any unstable manifold $W^u(y)$, so if $\rho_0$ is small and the central
cones are thin, the length of the central curves differ by a multiplicative constant, which we could take to be as close to $1$ as we want.
\medskip

The next statement asserts that s-bad cubes are contained in a union of strips $L_1,\dots,L_q$
(see Figure~\ref{f.split}) that are well separated by a distance $r>0$ and all have the same width $\eta r$ where $\eta\in (0,1/2)$ is chosen small, but large enough relative to $\beta$ so that that the following condition is satisfied:
\begin{equation}\label{e.boundbeta3}
\beta <d_s^{-2d_s}\left(\frac \eta {100}\right)^{(d_s+1)}.
\end{equation}

\begin{lemma}\label{l.strip}
For any $\beta,\sigma,\eta\in (0,1)$, $N_0\geq 1$ satisfying ~\eqref{e.boundbeta2} and~\eqref{e.boundbeta3}, and
such that $\sigma < \beta/2$, 
the following property holds if $\rho_0$ is sufficiently small.

For any $N\geq N_0$, choose $k_0$ satisfying~\eqref{e.boundN}.
For any $k\geq k_0$, $Q\in \cQ_{\varphi,k,\sigma}$ and any $x\in Q\cap \Lambda_N$,
there exist strong unstable strips $L_1,\dots,L_q$
in $W^u(f^N(z),\rho_0)$ and $r\in [d_s^{-2d_s}\;(\frac \eta {100})^{d_s}\; \frac {2^{-k}}\sigma ,\frac {2^{-k}}\sigma]$
such that:
\begin{itemize}
\item each strip has width less than $\frac {\eta r} 4$;
\item the distance between two strips is greater than $4r$;
\item the intersection of each s-bad cube $Q'\in \cQ_{\varphi,k,\sigma}$
with $f^N(W^u(x,\rho_0))$ is contained in the union of the strips $\bigcup_j L_j$.
\end{itemize}
\end{lemma}

\begin{figure}
\includegraphics[scale=0.36]{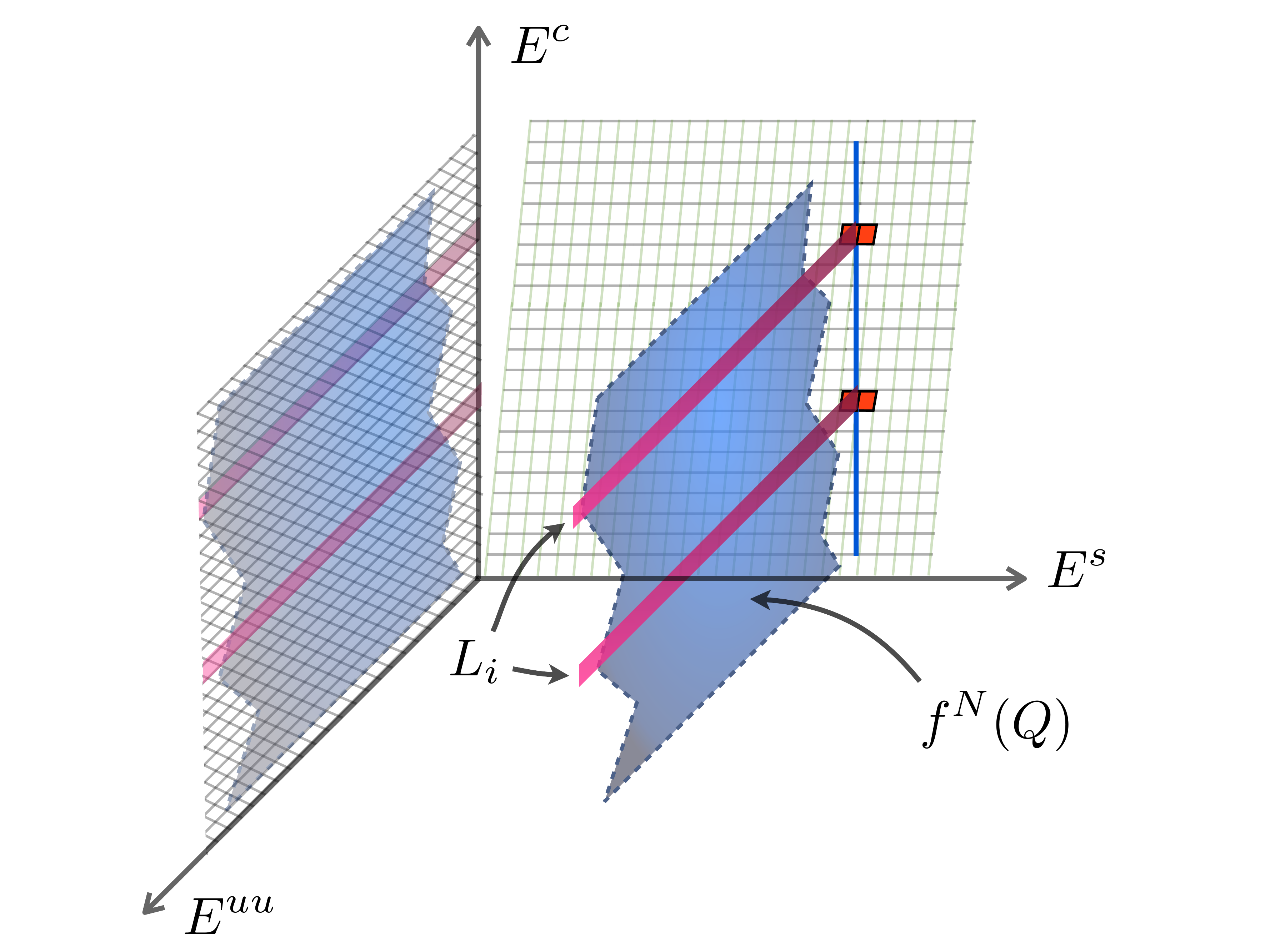}
\caption{\label{f.covering-sbad} Covering s-bad cubes by strong unstable strips.}
\end{figure}

\begin{proof}
By~\eqref{e.boundN},
the image $f^N(Q)$ has diameter smaller than $\rho_1$,
and the projection of $f^N(Q)$ on the space $\{0\}^{d_u}\times \RR^{d_s}$
(inside the chart $\varphi$)
has diameter less than $\frac \beta 2 2^{-k}$.
In particular, for any point $y\in f^N(Q)$, there exist $z\in W^u(f^N(x),\rho_0)$
such that $\varphi(z)-\varphi(y)$ belongs to $\{0\}^{d_u}\times \RR^{d_s}$
and has norm less than $\frac \beta 2 2^{-k}$.
It follows that
the s-bad cubes $Q'$ that intersect the plaque
$f^N(W^u(x,\rho_0))$ are a subset of those $Q'\in \cQ$ satisfying: 
$$
\partial_{\beta}^s Q' \cap W^u(f^N(x),\rho_0) \neq \emptyset.
$$
In the charts, the plaque $W^u(f^N(x),\rho_0)$ is close to a plane $H$ parallel to $\RR^{d_u}\times \{0\}^{d_s}$.
The intersection $H\cap \bigcup_{Q'\in\cQ}\partial^s Q'$ is contained in
a finite union $Z$ of hyperplanes parallel to $\RR^{d_u-1}\times \{0\}^{d_s+1}$ in $H$.
By projecting these hyperplanes on the $e_{d_u}$-axis, one obtains
a set $X$ of points  that is contained in the  translates of (at most) $d_s$ points under 
$(2^k\sigma)^{-1}\cdot \ZZ$. One can thus
project $X$ to a subset $\bar X$ of $\RR/(2^k\sigma)^{-1}\cdot \ZZ$ of cardinality $d_s$ and then
apply the following elementary lemma with $a=\frac \eta {50}$.

\begin{sublemma} \label{l=elem} Let $X \subset \RR/\ZZ$ be a finite subset of the circle
and $d=\#X$.
Then for every $a\in (0,1/2)$,  there exists
$\kappa\in [d^{-2d}(a/2)^{d},a/2]$ and a collection $I_1,\ldots, I_n \subset \RR/\ZZ$ of open intervals with the following properties:
\begin{itemize}
\item $X\subset I_1\cup\cdots\cup I_n$;
\item $\length(I_j) = \kappa$, for $j=1,\ldots, n$;
\item the length of the connected components of $(\RR/\ZZ)\setminus (I_1\cup\dots\cup I_n)$
is larger than $\kappa/a$.
\end{itemize}
\end{sublemma}

\begin{proof} The proof is by induction on $d$.
The cases $d=1$ and $d=2$ are easy.
For $d>2$, let us set $\ell=d^{-2d}(a/2)^{d}$.
If the minimum distance between distinct points in $X$ is at least $(1+a^{-1})\ell$, then we put an interval of diameter $\kappa:=\ell$ centered at each point of $X$ and the conclusion holds.

Otherwise we collapse all the connected components of $\RR/\ZZ$ of length less or equal to
$(1+a^{-1})\ell$. (There are at most $d-1$ many such components.)
We get a circle of length $L<1$ with at most $d-1$ points.
Rescaling the quotient circle to unit length and applying the inductive hypothesis with
$a'=\frac {d-2}{d-1}a$, we obtain a collection of intervals $I'_1,\dots,I'_n$
of length $\kappa'$ and separated by intervals of length larger than $\kappa'/a'$.

Pulling back the intervals $I'_j$ to the initial circle, we obtain intervals $\widetilde I_j$
of length $\kappa' L+\ell_j$ where $\ell_j$ is the sum of the lengths of the intervals
contained in $\widetilde I_j$ that have been collapsed.
One can enlarge the intervals $\widetilde I_j$ and get intervals
$I_1,\dots,I_n$ of length $\kappa:=\kappa' L+(d-1)(1+a^{-1})\ell$ and separated by distances larger than $\frac {\kappa'}  {a'} L -(d-1)(1+a^{-1})\ell$.

By definition, $\kappa\geq \ell$.
Using $L\leq 1$, $\kappa'\leq a'/2$, $(1+a^{-1})a/2<1$ and the definitions of $\kappa$, $\ell$ and $a'$, one gets easily $\kappa \leq a/2$.

In order to check that the intervals $I_j$ are separated by $\kappa/a$, we estimate
\begin{equation}\label{e.separate}
\begin{split}
&\left(\frac {\kappa'}  {a'} L  -(d-1)(1+a^{-1})\ell\right)-\kappa/a=\\
&\mbox{}\quad\quad\quad\quad
\frac {2(d-1)L\ell}{a^2}
\left(\frac{\frac a 2 \kappa'}{(d-1)(d-2)\ell}-\frac{(1+a)^2}{2L}\right).
\end{split}
\end{equation}
The induction assumption gives $\kappa'\geq (d-1)^{-2(d-1)}\left(\frac {a'} 2\right)^{d-1}$
and together with the definitions of $a'$ and $\ell$ we get:
$$\frac{\frac a 2 \kappa'}{(d-1)(d-2)\ell}\geq \frac{d^{2d}(d-2)^{d-2}}{(d-1)^{3d-2}}>\exp(1).$$
On the other hand,
$\frac{(1+a)^2}{2L}$ is smaller than $3/2$, since $a<1/2$. It follows  that~\eqref{e.separate} is positive, which concludes the proof.
\end{proof}

From the Sublemma~\ref{l=elem} 
we obtain $\kappa\in [d_s^{-2d_s}(a/2)^{d_s},a/2]$ and some subintervals $I_1,\dots,I_n$
of the circle.
Pulling back the intervals $I_i$ to the $e_{d_u}$-axis,
one can extract a finite collection of intervals $J_1,\dots, J_q$ of length $\frac \kappa \sigma2^{-k}$,
separated by $\frac \kappa {a\sigma} 2^{-k}$ and whose union contains $X$.

\begin{figure}[h]
\includegraphics[scale=0.32]{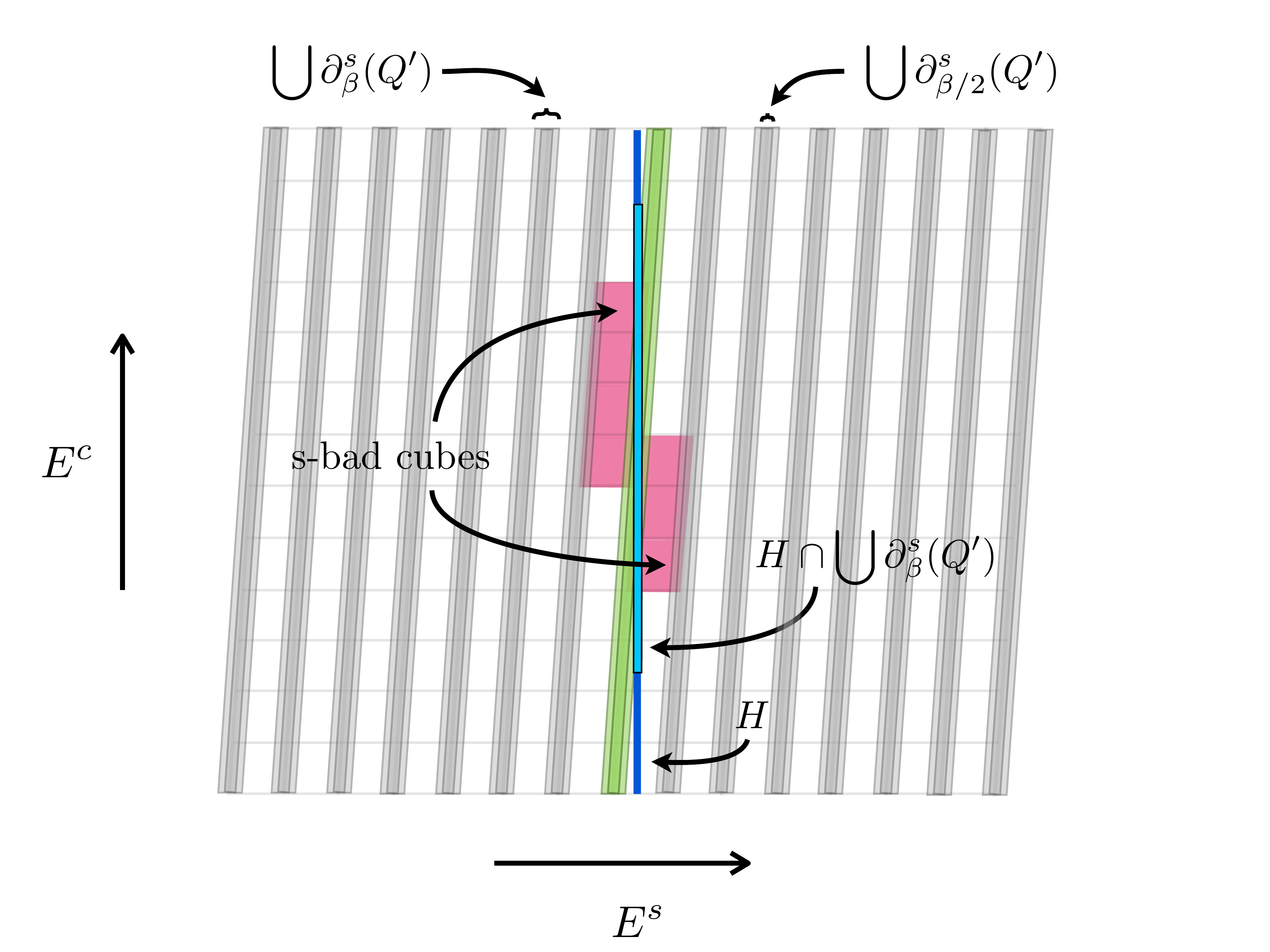}
\caption{\label{f.sbadcover} $H\cap \bigcup_{Q'\in\cQ}\partial^s_{\beta} Q'$ contains the intersection of $H$ with the s-bad cubes for $Q$.}
\end{figure}

Consider the intersection $H\cap \bigcup_{Q'\in\cQ}\partial^s_{\beta} Q'$ (see Figure~\ref{f.sbadcover}).  On the one hand,
since $\sigma < \beta/2$, it contains the intersection of $H$ with the s-bad cubes for $Q$.  On
the other hand,  $H\cap \bigcup_{Q'\in\cQ}\partial^s_{\beta} Q'$ is contained
in the $\frac {2\beta} \sigma 2^{-k}$-neighborhood of $Z$, hence in $q$
strips of $H$ of width smaller than $\frac 3 2 \frac { \kappa} \sigma2^{-k}$
(since by our choice of $\beta$ we have $\beta<\kappa/5$) and separated by $\frac 9 {10} \frac {\kappa} {a\sigma}2^{-k}$.
We obtain the strips $L_1,\dots,L_q$
by projection of the strips of $H$ to the plaque
$W^u(f^N(x),\rho_0)$ and set $r=\frac {\kappa} {5a\sigma}2^{-k}$:
the strips are separated by a distance larger than $4r$ and have width smaller
than $10a r$, which is smaller than $\frac {\eta r} 4$ by our choice of $a$.
We deduce the estimates on $r$ from the bounds on $\kappa$.
\end{proof}
\medskip

We then cover the strips obtained from Lemma~\ref{l.strip}
by a collection of unstable $\eta$-boundaries of cubes $C$ of size $r$
(see Figure~\ref{f.stripcover}).
We require the cubes $C$ to be smaller than the unstable $\beta$-boundary of the
image cube $f^N(Q)$:
\begin{equation}\label{e(u)ppera}
\frac \eta\sigma <{\textstyle \frac 1 {100}} \beta \exp(N_0\widehat \nu).
\end{equation}

\begin{corollary}\label{c.strip}
In the setting of Lemma~\ref{l.strip}, assume that~\eqref{e(u)ppera} holds.
Then
there exists a collection of cubes $C_1=Q_{r,\sigma}(y_1),\dots,C_n=Q_{r,\sigma}(y_n)$ such that:\begin{enumerate}
\item\label{E.1} each cube $C_i$ is disjoint from $f^N(\partial^u_{\beta/2}Q)$;
\item\label{E.2} each set $C_i\cap W^u(f^N(x),\rho_0)$ is non-empty and contained in
$f^N(Q)$;
\item\label{E.3} $(L_1\cup\dots\cup L_m)\cap f^N((1-\frac 3 4\beta)  Q)\subset \bigcup_i C_i$;
\item\label{E.4} if $C_i$ intersects $(L_1\cup\dots\cup L_m)\cap f^N((1-\frac 3 4\beta)  Q)$,
then $\cN^u(C_i)\subset\{C_1,\dots,C_n\}$;
\item\label{E.5} the interiors of the cubes $C_i$ are pairwise disjoint;
\item\label{E.6} the inner cubes $(1-\eta/2)  C_i$ and the strips $L_j$ do not intersect.
\end{enumerate}
\end{corollary}

\begin{figure}[h]
\includegraphics[scale=0.38]{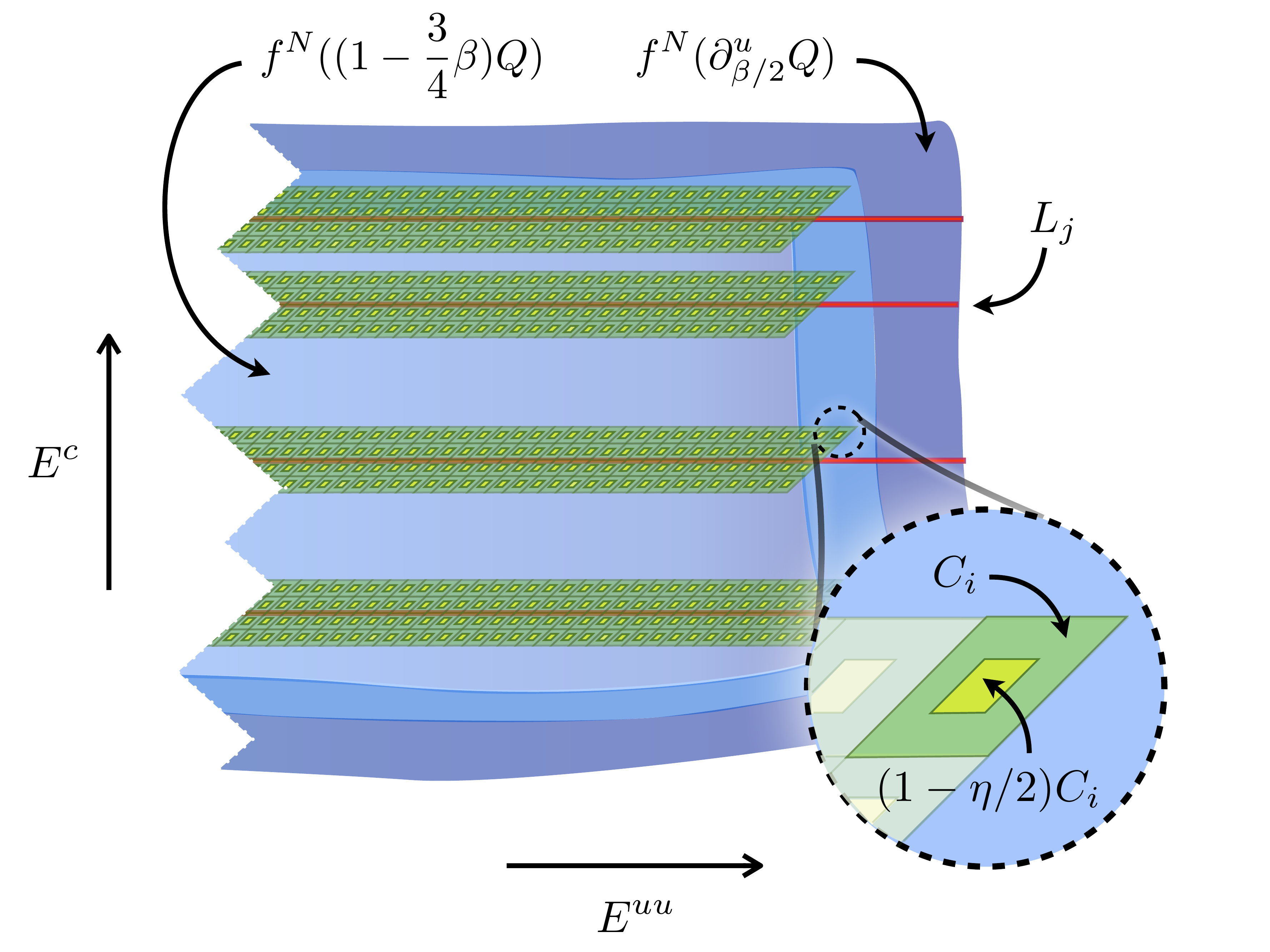}
\caption{\label{f.stripcover} Covering strips in $f^N(Q)$ with (larger) cubes $C_i$ in order to bound their measure.}
\end{figure}

\begin{proof}
For each strip $L_j$, we choose a point $y_0$ in $L_j$
and (in the chart $\varphi$) we consider the cubes 
of the form
$$C=Q_{r,\sigma}(y) ,\, \text{ with }
y\in \varphi^{-1}(\varphi(y_0)+rL_\sigma(\ZZ^d+{\textstyle \frac 1 2} e_{d_u}))$$
that either intersect the strip $L_j\cap f^N((1-\textstyle \frac 3 4 {\beta})  Q)$
or that have an unstable neighbor cube that intersects this strip.
By construction, the union of these cubes $C$ contains
$L_j\cap f^N((1-\frac 3 4\beta)  Q)$.

By the inequality~\eqref{e(u)ppera}, the size $r\leq \frac \eta{2\sigma}2^{-k}$
of the cubes $C$ is much smaller than 
the distance between $f^N((1-\frac 3 4\beta)  Q)$ and $f^N(\partial^u_{\beta/2} Q)$ in
the plaque $W^u(f^N(z),\rho_0)$. This implies item~(\ref{E.1}). Moreover the intersection of each cube $C$ with $W^u(f^N(z),\rho_0)$ is contained in
$f^N(Q)$.

Recall that $L_j$ is $C^1$-close to a plane spanned
by $e_1,\dots,e_{(d_u-1)}$ and has width smaller than $\eta r/4$.
Note also that $y_0$ belongs to the center of an unstable face of some adjacent cubes $C_0,C_0'$.
Any cube $C$ that intersects $L_j$ is the image of $C_0$ or $C_0'$ by a translation
by a bounded vector in $\ZZ e_1+\dots+\ZZ e_{(d_u-1)}$:
it also intersects $L_j$ near the center of an unstable face.
The other cubes $C$ are unstable neighbor cubes of those who intersect $L_j$.
In particular, item~(\ref{E.2}) holds.
Since the width of $L_j$ is smaller than $\eta r/4$, we deduce that the inner cubes $(1-\eta/2)  C$ do not intersect the strip $L_j$.
\medskip

The collection $\{C_1,\dots, C_n\}$ is the union of the collection of cubes $C$
associated to each strip $L_j$. Since the distance between the strips is larger than $4r$,
the cubes associated to $L_j$ do not intersect the cubes associated to $L_{j'}$ if $j\neq j'$.
The other items follow.
\end{proof}

\subsubsection{End of the proof of Proposition~\ref{p=selectcubes}}
Given $\varepsilon>0$, we make the following choices:
\begin{itemize}
\item $\delta \in \big(0, 3^{-(d_u+1)}(e^{\varepsilon/2}-1)\big)$ controls the $\mu^u$-measure of poorly crossed cubes.
\item The boundary size $\eta=\eta(\delta)$ is chosen according to Lemma~\ref{l.density}.
\item The boundary size $\beta$ 
is chosen small in order to satisfy~\eqref{e.boundbeta3}, and the shear  is fixed to be $\sigma=\beta/10$.
\item The iterate $N$ is larger than a bound $N_0$ satisfying
the properties of Section~\ref{ss.firstextraction}, and the inequalities~\eqref{e.boundbeta1},~\eqref{e.boundbeta2} and~\eqref{e(u)ppera}.
\item The scales $\rho_0$ and $\rho_1\ll \rho_0$ are chosen (independently from $N_0$ and $N$) so that the variations of the splitting $E^{uu}\oplus E^{c}\oplus E^{s}$
are small (see Section~\ref{ss.cube}, Lemmas~\ref{l.density} and~\ref{l.strip}).
\item A lower bound $k_0$ on the scales $k$, chosen so that $f^N$ is nearly linear at scale
$2^{-k_0}$ and such that~\eqref{e.boundN} holds.
\end{itemize}

\begin{lemma}\label{l.boundsbad} Under these assumptions, the union of the s-bad cubes in $\cQ$ that intersect $f^N((1-\frac 3 4 \beta)  Q)$
has $\mu^u_{f^N(x)}$ measure smaller than
$$
3^{(d_u+1)}\delta\sum_{Q\rightarrow Q'}\mu^u_{f^N(x)}((1-\eta)  Q').$$
\end{lemma}
\begin{proof}
By Lemma~\ref{l.strip}, and~\eqref{e.boundbeta1} (the size of the cube in $\cQ$ is much smaller than
the unstable $\beta$-boundary of $f^N(Q)$),
we aim to bound the
$\mu^u_{f^N(x)}$-measure of
$$\Delta:=(L_1\cup\dots\cup L_m)\cap f^N((1-{\textstyle \frac 3 4} \beta)  Q).$$
By Corollary~\ref{c.strip}, items~(\ref{E.3}) and~(\ref{E.6}), it is bounded by:
$$\mu^u_{f^N(x)}(\Delta)
\leq \sum_{C_i\cap \Delta\neq\emptyset}\mu^u_{f^N(x)}(C_i\setminus (1-\eta)  C_i).$$
By Lemma~\ref{l.density}, the measure of the $\eta$-boundary of each $C_i$ is smaller than
$$\mu^u_{f^N(x)}(C_i\setminus (1-\eta) C_i)\leq \delta \sum_{C\in \cN^u(C_i)} \mu^u_{f^N(x)}((1-\eta)  C).$$
If $C_i\cap \Delta\neq \emptyset$, then the cubes $C\in \cN^u(C_i)$
still belong to the family $\{C_1,\dots,C_n\}$
(by Corollary~\ref{c.strip} item~\ref{E.4}); moreover
each cube $C$ belongs to no more than $3^{d_u}$ sets $\cN^u(C_i)$.
This gives
$$\mu^u_{f^N(x)}(\Delta)
\leq
3^{d_u}\;\delta\sum_{i=1}^n\mu^u_{f^N(x)}((1-\eta)  C_i).$$
\medskip

Consider any cube $Q'\in \cQ$ that intersects $(1-\eta)  C_i\cap W^u(f^N(x),\rho_0)$
and take $Q''\in \cN^u(Q')$. Then we claim that $Q''\subset (1-\frac \eta 2)  C_i$ and
$Q\rightarrow Q''$. Indeed:
\begin{itemize}
\item $Q''\subset (1-\frac \eta 2)  C_i$, since from $\sigma=\frac \beta {10}$ and~\eqref{e.boundbeta3} the size of the cubes in $\cQ$ is smaller than the $\eta$-boundary size of the cubes $C_i$.
\item $Q''$ is thus disjoint from $L_i$ and is not an s-bad cube by Corollary~\ref{c.strip} item~\ref{E.6}.
\item $Q''$ and $W^u(f^N(x),\rho_0)$ intersect: since $Q'$ intersects $W^u(f^N(x),\rho_0)$, is not an s-bad cube, and since $\sigma=\beta/10$, the inner cube $(1-3\sigma)  Q'$ and $W^u(f^N(x),\rho_0)$ intersect;
this implies that the unstable neighbor cube $Q''$ also intersect $W^u(f^N(x),\rho_0)$.
\item $Q''$ intersects $f^N(Q)$, since $Q''\cap W^u(f^N(x),\rho_0)$
is non-empty and contained in $C_i\cap W^u(f^N(x),\rho_0)$, which is contained in
$f^N(Q)$ by Corollary~\ref{c.strip} item~\ref{E.2}.
\item $Q''\subset C_i$ does not meet $f^N(\partial^u_{\beta/2} Q)$ by corollary~\ref{c.strip} item~\ref{E.1}.
\end{itemize}

Applying Lemma~\ref{l.density} to the cubes $Q'\in \cQ$ that intersect $(1-\eta)  C_i\cap W^u(f^N(x),\rho_0)$ it follows that
$$\mu^u_{f^N(x)}((1-\eta)  C_i)\leq (1+3^{d_u}\;\delta)\sum_{Q''\in \cQ(C_i)} \mu^u_{f^N(x)}((1-\eta)  Q''),$$
where $\cQ(C_i)$ is the family of cubes $Q''\in \cQ$ such that $Q\rightarrow Q''$
and whose intersection with
$W^u(f^N(x),\rho_0)$ is non-empty and contained in $(1-\frac \eta 2)  C_i$.

By Corollary~\ref{c.strip} item~\ref{E.5}, the sets $(1-\frac \eta 2)  Q_i$ are disjoint.
This gives
$$\mu^u_{f^N(x)}(\Delta)
\leq
3^{d_u}\; \delta (1+3^{d_u}\; \delta)\sum_{Q\rightarrow Q'}\mu^u_{f^N(x)}((1-\eta)  Q'),$$
which implies the result since $\delta$ is smaller than $3^{-d_u}$.
\end{proof}
\bigskip

We now bound the measure:
$$\mu^u_{f^N(x)}(f^N((1-\beta)  Q))\leq \sum_{Q'\in \cQ(x)} \mu^u_{f^N(x)}(Q')$$
where $\cQ(x)$ is the collection of $Q'\in \cQ$ that intersect
$f^N((1-\beta)  Q)\cap W^u(f^N(x),\rho_0)$.
\medskip

The last sum may be decomposed in three parts:
\begin{enumerate}
\item We first consider the cubes $Q'$ that are s-bad: Lemma~\ref{l.boundsbad} implies
$$\sum_{Q'\in \cQ(x) \text{ s-bad}} \mu^u_{f^N(x)}(Q')\leq 3^{(d_u+1)}\delta\sum_{Q\rightarrow Q'}\mu^u_{f^N(x)}((1-2\beta)  Q').$$
\item We then consider the inner part $(1-2\beta)  Q'$ of the cubes that are not s-bad.
As explained at the beginning of Section~\ref{ss.proof-selectcube} we necessarily have $Q\rightarrow Q'$.
$$\sum_{Q'\in \cQ(x) \text{ not s-bad}} \mu^u_{f^N(x)}((1-2\beta)  Q')
\leq \sum_{Q\rightarrow Q'}\mu^u_{f^N(x)}((1-2\beta)  Q').$$
\item We finally consider the boundary part $Q'\setminus (1-2\beta)  Q'$ of the cubes that are not s-bad.
Since $3\sigma<2\beta$,  Lemma~\ref{l.density} applies and gives
$$\sum_{\substack{Q'\in \cQ(x) \\ \text{ not s-bad}}} \mu^u_{f^N(x)}(Q'\setminus (1-2\beta)  Q')
\leq \delta \sum_{\substack{Q'\in \cQ(x) \\ \text{ not s-bad}}}\sum_{Q''\in\cN^u(Q')}
\mu^u_{f^N(x)}((1-2\beta)  Q'').$$
Note that a cube $Q''$ appears at most $3^{d_u}$ times in the last sum.
Moreover since it is an unstable neighbor of a cube $Q'$ that is not s-bad and intersects
$f^N((1-\beta)  Q)\cap W^u(f^N(x),\rho_0)$, the cube $Q''$  must also intersect $f^N((1-\frac 3 4\beta)  Q)\cap W^u(f^N(x)\rho_0)$. One may thus again distinguish those that are s-bad, where we use Lemma~\ref{l.boundsbad}, from those that are not s-bad and satisfy $Q\rightarrow Q''$.
This gives:
$$\sum_{\substack{Q'\in \cQ(x) \\ \text{ not s-bad}}} \mu^u_{f^N(x)}(Q'\setminus (1-2\beta)  Q')
\;\leq\; 3^{d_u} \delta\; (1+3^{(d_u+1)}\delta) \sum_{Q\rightarrow Q'}\mu^u_{f^N(x)}((1-2\beta)  Q').$$
\end{enumerate}
Combining these estimates  gives:
$$\mu^u_{f^N(x)}(f^N((1-\beta)  Q))\leq (1+3^{(d_u+1)}\delta)^2
\sum_{Q\rightarrow Q'}\mu^u_{f^N(x)}((1-2\beta)  Q').$$

By our choice of $\delta$ we have $(1+3^{(d_u+1)}\delta)^2<e^\varepsilon$.
One gets the estimate of Proposition~\ref{p=selectcubes} since $f^N$ has a constant Jacobian
along the unstable leaves for the measure $\mu^u$.


\section{Birth of blenders inside large horseshoes}\label{s.birth}

In this section we prove Theorem C.

\subsection{The recurrent compact criterion for horseshoes}

Let $\Lambda$ be a horseshoe for a $C^1$ diffeomorphism $f$
whose unstable bundle has a dominated splitting $E^u=E^{uu}\oplus E^c$.
Let $d_{uu},d_u,d_s$ be the strong unstable, unstable and stable dimensions
and let $d_{cs}=(d_u-d_{uu})+d_s$ the dimension of the bundle $E^c\oplus E^s$.
\medskip

\noindent
\emph{The local strong unstable lamination $\mathcal{W}^{uu}_{loc}$.}
For any $g$ that is $C^1$-close to $f$ and any $x$ in the continuation $\Lambda_g$, there exist
local unstable manifolds $W^u_{loc}(g,x)$, which depend continuously on $(x,g)$
in the $C^1$ topology and which satisfy
$$g(W^u_{loc}(g,x))\supset \overline{W^u_{loc}(g,g(x))}.$$
Each local unstable manifold supports a strong unstable foliation tangent to $E^{uu}$
and for each $z\in W^u_{loc}(g,x)$ we denote by $W^{uu}_{loc}(g,z)$ the (connected)
leaf containing $z$.
The collection of all local strong unstable leaves defines the local strong unstable lamination $\mathcal{W}^{uu}_{loc}(g)$
associated to $\Lambda_g$.
Since $\Lambda$ is totally disconnected
we may assume furthermore that for any $x,y\in \Lambda_g$
the plaques $W^u_{loc}(g,x),W^u_{loc}(g,y)$ are either disjoint or equal.
\medskip

\noindent
\emph{The transversal $\mathcal{D}$.}
We then consider a submanifold $\mathcal{D}\subset M$ such that
for each $x\in \Lambda$, the intersection between the closures of $\mathcal{D}$ and
$W^u_{loc}(x)$ is contained in $\mathcal{D}\cap W^u_{loc}(x)$, is transverse to the strong unstable foliation
and contains at most one point of each strong unstable leaf.
In particular  these properties still hold for $\mathcal{D}$ and the local strong unstable lamination
$\mathcal{W}^{uu}_{loc}(g)$ if $g$ is $C^1$-close to $f$.
\medskip

The following definition comes from~\cite{MS} and generalizes a property
introduced in~\cite{MY} for producing robust intersections of regular Cantor sets.
\begin{definition}\label{d.recurrent}
A compact set $K\subset \mathcal{D}$ is \emph{recurrent}
(with respect to $\mathcal{W}_{loc}^{uu}$)
if for any points $x\in \Lambda$ and $z\in W^u_{loc}(x)\cap K$,
there exists $n\geq 1$, $x'\in \Lambda$
and some point $z'\in W^u_{loc}(x')$ contained in the interior of $K$ (in the topology of $\mathcal{D}$),
such that $f^{-n}(x')\in W^u_{loc}(x)$ and $f^{-n}(\overline{W^{uu}_{loc}(z')})\subset W^{uu}_{loc}(z)$.
\end{definition}

This property is robust and implies that $\Lambda$ is a $d_{cs}$-stable blender,
as stated in the two next propositions (which also come from~\cite{MS}):

\begin{proposition}
The compact set $K$ is still recurrent with respect to the strong unstable leaves
$\mathcal{W}_{loc}^{uu}(g)$ for the diffeomorphisms $g$ that are $C^1$-close to $f$.
\end{proposition}
\begin{proof}
Consider the (closed) set $I_g=\{(x,z)\in \Lambda_g\times K, z\in W^u_{loc}(g,x)\}$:
it is contained in a small neighborhood of $I_f$ if $g$ is $C^1$-close to $f$.
For any $(x_0,z_0)\in I_f$,  consider $(x'_0,z'_0)\in I_f$ with $z'_0\in \operatorname{interior}(K)$ and $n\geq 1$
such that $f^{-n}(x'_0)\in W^u_{loc}(x_0)$ and $f^{-n}(W^{uu}_{loc}(z'_0))\subset W^{uu}_{loc}(z_0)$.
Then for any $g$ that is $C^1$-close to $f$ and any $(x,z)\in I_g$
that is close to $(x_0,z_0)$, one can consider $x'\in \Lambda_g$ that is close to $x'_0$
such that $g^{-n}(x')\in W^u_{loc}(g,x)$. Since $z'_0$ belongs to $\operatorname{interior}(K)$,
by continuity there exists $z'\in \operatorname{interior}(K)\cap W^u_{loc}(g,x')$
such that $g^{-n}(W^{uu}_{loc}(g,z'))\subset W^{uu}_{loc}(g,z)$.
By compactness, there exists a small $C^1$ neighborhood $\cU$ of $f$
such that this property holds for any $(x_0,z_0)\in I_f$ for any $g\in \cU$.
\end{proof}

\begin{proposition}[Recurrent compact criterion]\label{p.criterion}
Consider a horseshoe $\Lambda$ with a strong unstable bundle of codimension $d_{cs}$.
If $\Lambda$ admits a recurrent compact set $K$
(with respect to $\mathcal{W}^{uu}_{loc}$) which intersects at least one plaque $W^u_{loc}(x)$
then it is a $d_{cs}$-stable blender.
\end{proposition}
\begin{proof}
We start with preliminary considerations:
\begin{enumerate}
\item By compactness, the integer $n\geq 1$ in Definition~\ref{d.recurrent}
may be chosen to be bounded by some $n_0$.

\item\label{i.close2} Changing the metric, we may assume that $\|Df_{|E^{s}}(x)\|<1$ at each point $x$ of $\Lambda$.
Iterating backwards, we may also assume that the local unstable manifolds are arbitrarily small.
Since $\Lambda$ is totally disconnected, there exists a smooth foliation $\cF$ in a neighborhood
whose leaves have tangent spaces close to the bundle $E^{s}$. This defines local projections onto
local unstable manifolds.

\item\label{i.close3} The dominated splitting $E^{uu}\oplus (E^c\oplus E^s)$ gives the existence of an invariant cone field
$\cC^{uu}$: for each $x$ in a neighborhood of $\Lambda$, the cone $\cC^{uu}_x\setminus \{x\}$
is open in $T_xM$, and $Df_x(\overline{\cC^{uu}(x)})\subset \cC^{uu}(f(x))$.
Moreover, $\cC^{uu}_x$ contains $E^{uu}_x$ and is transverse to $E^c\oplus E^s$.
Replacing the cone field by a forward iterate, we obtain an arbitrarily thin cone field around the bundle $E^{uu}$
and defined on a neighborhood of the lamination $\cW^{uu}_{|K}$.
For any $g$ that is $C^1$-close to $f$, the cone field is still invariant.
\end{enumerate}

Let $\Gamma$ be a \emph{submanifold $C^1$-close to a local strong unstable manifold
$W^{uu}_{loc}(z)$} where $z\in K\cap W^u_{loc}(x)$ and $x\in \Lambda$.
More precisely, this means that there exists a small constant $\varepsilon>0$ such that
$\Gamma$ is tangent to $\cC^{uu}$,
its distance to $W^u_{loc}(x)$ (measured along the leaves of $\cF$)
is smaller than $\varepsilon$
and its projection to $W^u_{loc}(x)$ contains a point $\varepsilon$-close to $z$.

\begin{claim}
Let $g$ be a diffeomorphism $C^1$-close to $f$.
Assume that $\Gamma$ is close to the local strong unstable manifold
$W^{uu}_{loc}(z)$ where $z\in K\cap W^u_{loc}(x)$ and $x\in \Lambda$.
Consider $x'\in \Lambda$, $n\in \{1,\dots,n_0\}$, and $z'\in W^u_{loc}(x')\cap \operatorname{interior}(K)$
such that $f^{-n}(\overline{W^{uu}_{loc}(z')})\subset W^{uu}_{loc}(z)$.
Then $g^{n}(\Gamma)$ contains an open submanifold $\Gamma'$
$C^1$-close to $\overline{W^{uu}_{loc}(z'')}$, for some $z''\in W^u_{loc}(x')\cap \operatorname{interior}(K)$.
\end{claim}
\begin{proof}
Since $f^{-n}(\overline{W^{uu}_{loc}(z')})\subset W^{uu}_{loc}(z)$ and $n$ is bounded,
there exists $\Gamma'\subset g^{n}(\Gamma)$ whose closure is close to $\overline{W^{uu}_{loc}(z')}$in the Hausdorff topology. From~(\ref{i.close3}), it is tangent to the cone field $\cC^{uu}$
and from~(\ref{i.close2}), its distance to $W^u_{loc}(x')$ is smaller than $\varepsilon$.
The projection of $\Gamma'$ to $W^u_{loc}(x')$ intersects $\cD$ at a point $z''$ close to $z'$.
Since $z'$ belongs to the interior of $K$, the point $z''$ belongs to $\operatorname{interior}(K)$ also.
Consequently $\Gamma'$ is $C^1$-close to $\overline{W^{uu}_{loc}(z'')}$ and
the claim is proved.
\end{proof}

Now consider $g$ and $\Gamma$ as in the claim.
Applying the claim inductively,  we get a decreasing sequence of submanifolds $\Gamma_k$
and an increasing sequence of integers $n_k\to +\infty$ such that each
$g^{n_k}(\Gamma_k)$ is close to a local strong unstable manifold. This proves that
the forward orbit of the point $y=\cap_k \Gamma_k$ remains in a small neighborhood of $\Lambda_g$.
Consequently, $y$ belongs to the intersection of $\Gamma$ and a local stable manifold of
a point of $\Lambda_g$. This proves that $\Lambda$ is a $d_{cs}$-blender.
\end{proof}
\bigskip

\subsection{The recurrent compact criterion for affine horseshoes}\label{ss.affine-recurrent}
Let $B^u=[0,1]^{d_u}$, $B^s=[0,1]^{d_s}$. 
We now assume that $\Lambda$ satisfies the following additional property.

\begin{definition}\label{d.standard}
We say that $\Lambda$ is a \emph{standard affine horseshoe} if there exist:
\begin{itemize}
\item a chart $\varphi\colon U\to \RR^{d_u+d_s}$
with $[0,1]^{d_u+d_s}\subset \varphi(U)$,
\item a linear map $A\in \GL(d_u+d_s,\RR)$ which is a product $A^u\times A^s$ where
$(A^u)^{-1}$ and $A^s$ are contractions of $\RR^{d_u}$ and $\RR^{d_s}$,
\item pairwise disjoint translated copies $B^s_1$,\dots, $B^s_\ell$ of $A^s(B^s)$
in $\operatorname{interior}(B^s)$,
\item pairwise disjoint translated copies $B^u_1$,\dots, $B^u_\ell$ of $(A^u)^{-1}(B^u)$
in $\operatorname{interior}(B^u)$,
\end{itemize}
with the following properties:
\begin{itemize}
\item $f$ sends $\varphi^{-1}(B^u_j\times B^s)$ to $\varphi^{-1}(B^u\times B^s_j)$ for each $1\leq j\leq \ell$,
\item $\varphi\circ f\circ \varphi^{-1}$ agrees on ${B^u_j\times B^s}$  with the map $z\mapsto A(z)+v_j$, for some $v_j\in \RR^d$, and
\item $\Lambda$ is the maximal invariant set in $\bigcup \varphi^{-1}(B^u_j\times B^s)$.
\end{itemize}
\end{definition}
Note that, denoting $R:=\varphi^{-1}([0,1]^{d_u+d_s})$,
the cubes $R_j:=\varphi^{-1}(B^u_j\times B^s)$ are the connected components of $R\cap f^{-1}(R)$
that intersect $\Lambda$. See Figure~\ref{f.horseshoe}.
\begin{figure}[ht]
\begin{center}
\mbox{}\hspace{-1.4cm}\includegraphics[width=4cm,height=8cm,angle=90]{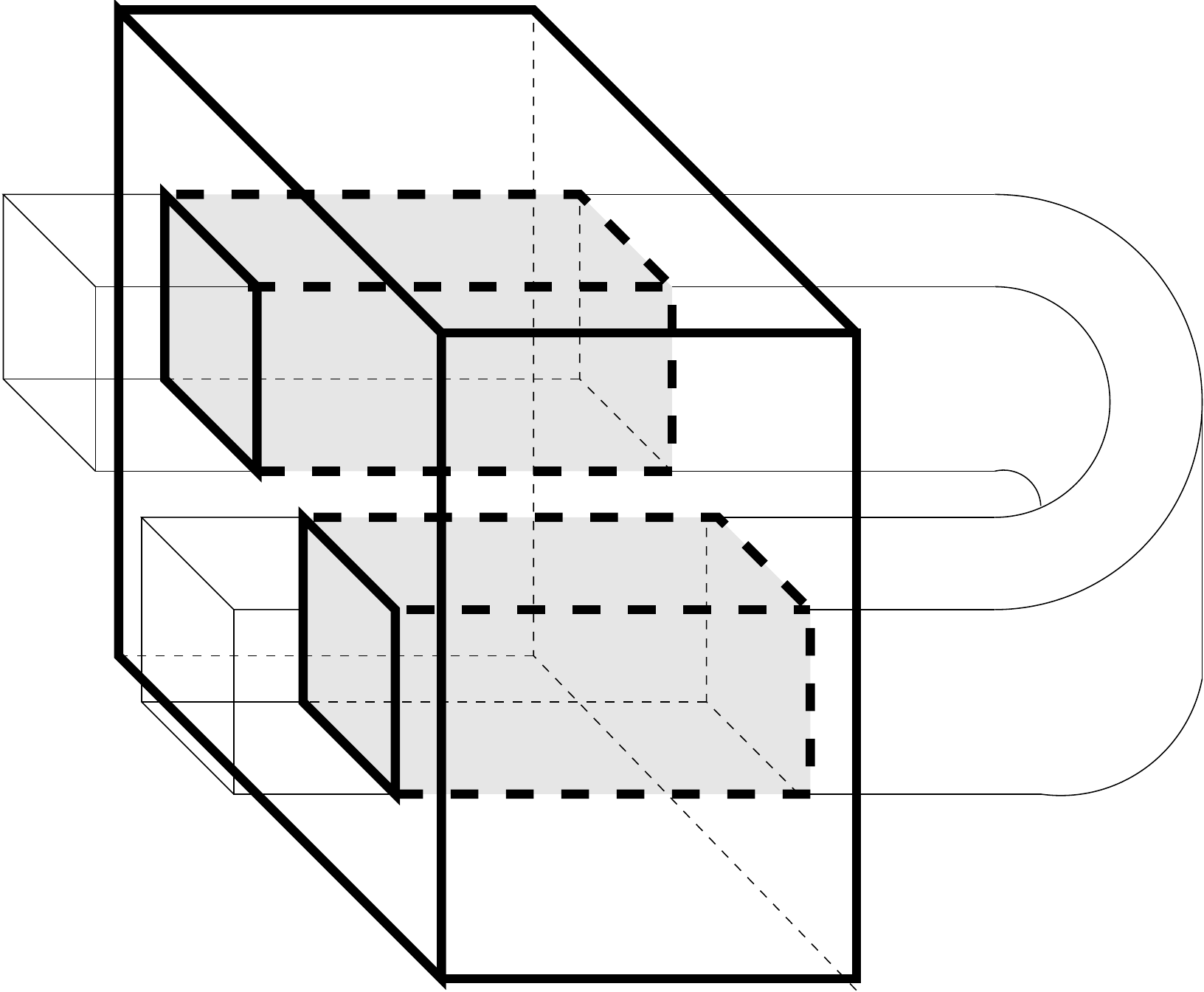}\end{center}
\begin{picture}(0,0)
\put(-155,50){$B^u_1\times B^s$}
\put(-180,90){$B^u_2\times B^s$}
\put(37,10){$R$}
\put(100,60){$f^{-1}(R)$}
\end{picture}
\caption{A standard affine horseshoe.\label{f.horseshoe}}
\end{figure}
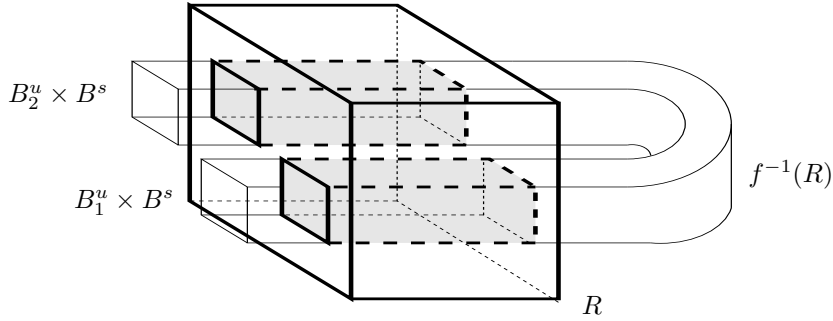
\medskip

We will assume furthermore that $\Lambda$ is partially hyperbolic:
$A$ preserves
the dominated splitting $\RR^{d_s+d_u} = \RR^{d_{uu}}\oplus \RR^{d_c}\oplus \RR^{d_s}$,
where $d_u=d_{uu}+d_c$.
In particular, $A$ and $B^u$ are products $A=A^{uu}\times A^c\times A^s$, $B^u=B^{uu}\times B^c$
and each $B^u_j$ is a product $B^{uu}_j\times B^c_j$ where $B^c=A^c(B^c_j)+\pi^c(v_j)$
and $\pi^c\colon \RR^{d_u+d_s}\to \RR^{d^c}$ is the projection onto the center coordinate.
We thus obtain a family of affine contractions $L_j$, $1\leq j\leq \ell$, which send $B^c$ to
$B_j^c$ respectively and coincide with $z\mapsto (A^c)^{-1}(z-\pi^c(v_j))$. This defines a (finitely generated)
iterated function system (IFS) in $B^c$.

Denoting $\beta=|\det(A^{uu})|$, we have that each strong unstable plaque $B^{uu}\times \{z\}$ with $z\in B^c\times B^s$
can intersect at most $\beta$ distinct rectangles $B^u_j\times B^s$.
This remark applied to any iterate $f^n$ shows that
for any point $x\in B^c$,
\begin{equation}\label{e.bound-images}
\operatorname{Card}\left\{(j_1,\dots,j_n)\in \{1,\dots,\ell\}^n, \; x\in L_{j_n}\circ\dots \circ L_{j_1}(B^c)\right\}\leq \beta^n.
\end{equation}

We state a definition for IFS which is the analogue of Definition~\ref{d.recurrent}.
\begin{definition} The IFS $\mathcal{L}=\{ L_j \colon B^c \to B^c : j=1,\ldots, \ell\}$ satisfies the
{\em recurrent compact condition} if there exists a nonempty compact set $K^c \subset \bigcup_j B^c_j$
such that for every $z \in K^c$ there exist $z'\in \operatorname{interior}(K^c)$ and $1\leq j\leq \ell$ with $L_j(z')=z$.
\end{definition}

\begin{proposition}\label{p.standard}
If the central IFS associated to a partially hyperbolic standard affine horseshoe $\Lambda$ as above satisfies the recurrent compact condition,
then $\Lambda$ admits a recurrent compact set.
\end{proposition}
\begin{proof}
We introduce the local strong unstable lamination by plaques $W^{uu}_{loc}$ of the form
$[0,1]^{d_{uu}}\times\{a\}$ where $a\in [0,1]^{d_c+d_s}$.
We then define the section $\mathcal{D}=\{1/2\}^{d_{uu}}\times (0,1)^{d_c+d_s}$
and introduce the compact set $K=K^c\times \bigcup_j B^s_j$.

Let us consider $x=(x_{uu},x_c,x_s)$ in $\Lambda\subset \RR^{d_{uu}}\times \RR^{d_c}\times \RR^{d_s}$
and a point $z=\{1/2\}^{d_{uu}}\times (z_c,x_s)$ in $W^{u}_{loc}(x)\cap K$.
We have $z_c\in K^c$, hence there exists $z'_c\in \operatorname{interior}(K^c)$ and $1\leq j\leq \ell$
with $L_j(z'_c)=z_c$.
Since $W^{uu}_{loc}(z)$ intersects $B^u_j\times B^s$,
there exists a point $x'\in \Lambda$
of the form $x'=(x'_{uu},x'_c,x'_s)$ such that $A^{-1}(x'-v_j)$ belongs to $W^u_{loc}(x)$.
The point $z'=\{1/2\}^{d_{uu}}\times(z'_c,x'_s)$ has the property that
$f^{-1}(\overline{W^{uu}_{loc}(z')})\subset W^{uu}_{loc}(z)$. Moreover
since $z'_c$ belongs to $\operatorname{interior}(K^c)$, the point $z'$ belongs to $\operatorname{interior}(K)$
as required.\!\!\!
\end{proof}
\smallskip

\subsection{Perturbations of iterated function systems}
Let $B=[-1/2,1/2]^d$, let $L\in \GL(d,\RR)$ be a linear contraction, and write $J=|\det(L)|$.
For $H\geq 1$, let $v_0,\dots,v_H\in B$ such that the affine contractions
$L_j\colon z\mapsto L.z+v_j$ send $B$ into its interior.
For any finite word $\d=(j_1,...,j_n)$ in $\{1,\dots,H\}^n$
let $L_\d=L_{j_n} \circ \cdots \circ L_{j_1}$.
For $n$ large enough,
we consider the IFS $\cL_n=\{L_0\circ L_\d,\;\d\in \{1,\dots,H\}^n\}$.

We will assume furthermore that there exists $\beta>1$ and $c\in (0,1)$ such that:
\begin{itemize}
\item $\beta^{2-c}<(JH)^{2}$,
\item $\operatorname{Card}\left\{\d\in \{1,\dots,H\}^n, \; x\in L_\d(B)\right\}\leq \beta^n$ for any $n\geq 1$
and $x\in B$.
\end{itemize}

For $n\geq 1$
we define $m=[c\cdot n]+1$ and consider the space $\Omega_n$
of functions $w\colon \{1,\dots,H\}^{n}\to B$
that satisfy $w_{\d}=w_{\d'}$ each time the last $m$ letters
$j_{n-m+1},\dots, j_n$ of $\d$ and $\d'$ coincide.
Since $\Omega_n$ can be identified with $B^{H^{n-m}}$ it is a probability space
(endowed with the product Lebesgue measure).
\medskip

The following probabilistic argument allows us to perturb the initial iterated function system $\cL_n$
so that the recurrent compact condition is satisfied.

\begin{proposition}\label{p.IFS}
Under the previous assumptions, if $n$ is large enough, then there exists $w\in \Omega_n$
such that the IFS induced by the affine contractions
$$z\mapsto L_0\circ L_\d(z)+10.L^{n+1}(w_\d), \quad\d\in \{1,\dots,H\}^n$$
satisfies the recurrent compact condition.
\end{proposition}
\begin{proof}
Let $\alpha_n=\frac 1 2 J^{n+1} (\beta^{-1} H)^n$.

\begin{claim}
There exists a compact set $A\subset B$ with Lebesgue measure
$|A|\geq \alpha_n$ of points $x$ that belong to a least $\frac 1 2 J^{n+1}H^n\beta^{m-n}$ images
$L_0\circ L_\d(B)$ associated to sequences $\d\in\{1,\dots,H\}^n$ whose
$m$ last letters $j_{n-m+1},\dots,j_n$ are different.
\end{claim}
\begin{proof}
Let $A\subset B$ be the set of points $x$ that belong to at least $J^{n+1}H^n/2$ images
$L_0\circ L_\d(B)$ where $\d\in \{1,\dots,H\}^n$.
By our assumptions, a point $x$ can belong to at most
$\beta^n$ such images (here we apply the estimate~\eqref{e.bound-images} to the point $L_0^{-1}(x)$).
Considering the images of $B$ by all possible maps $L_0\circ L_\d$, one gets:

\begin{equation*}
\begin{split}J^{n+1} H^n&=\sum_{\d}|L_0\circ L_\d(B)|=\int_B \operatorname{Card}\{\d\in \{1,\dots,H\}^n, x\in L_0\circ L_\d(B)\} dx\\
&\leq \beta^{n}|A|+\frac{J^{n+1}H^n}{2}.
\end{split}
\end{equation*}
This gives $|A|\geq \alpha_n$.

Fix $j_{n-m+1},\dots,j_n$ in $\{1,\dots,H\}^m$ and some $x\in B$.
The number of images $L_0\circ L_\d(B)$ containing $x$
and whose last $m$ letters of $\d$ coincide with $j_{n-m+1},\dots,j_n$
is equal to the number of images $L_{\d'}(B)$ with $\d'\in \{1,\dots,H\}^{n-m}$
which contain $(L_0\circ L_{j_n}\circ\dots\circ L_{j_{n-m+1}})^{-1}(x)$.
It is thus smaller than $\beta^{n-m}$.
It follows that the points $x\in A$ belong to at least $\frac 1 2J^{n+1}H^n\beta^{m-n}$
images $L_0\circ L_\d(B)$ associated to sequences $\d\in\{1,\dots,H\}^n$ whose
$m$ last letters $j_{n-m+1},\dots,j_n$ are different.
\end{proof}

Consider the tiles $L^{n+1}(u+B)$, where $u\in \ZZ^d$,
and their cores $L^{n+1}(u+\frac 1 2 B)$.
We denote by $K$ the union of the tiles which intersect $A$
and by $\tilde K$ the union of the cores of these tiles.
We have $|\tilde K|=2^{-d}|K|\geq 2^{-d}\alpha_n$.
Finally, we define $$A' =K\cap L^{n+1}\left(\frac 1 {100}.\ZZ^d\right).$$
Then the conclusion of the proposition holds for any parameter $w\in \Omega_n$
such that:
\begin{equation}\label{e.recurrent}
A'\subset \bigcup_{\d\in \{1,\dots,H\}^n} L_0 \circ L_\d(\tilde K)+10.L^{n+1}(w_\d),
\end{equation}

For a fixed $z \in A'$, we estimate the measure of the set of parameters 
\begin{equation}\label{e.badevent}
\Omega_n(z):=\{w\in \Omega_n,\; z\not\in \bigcup_{\d\in \{1,\dots,H\}^n} L_0 \circ L_\d(\tilde K)+10.L^{n+1}(w_\d)\}.
\end{equation}

Pick some $z_0 \in A$ which belongs to a same tile $L^{n+1}(u+B)$ as $z$.
For each $\d\in\{1,\dots,H\}^n$ such that $z_0\in L_0\circ L_\d(B)$, the
probability that $z \in L_0 \circ L_\d(\tilde K)+10.L^{n+1}(w_\d)$ is
$$\frac{|L_0 \circ L_\d(\tilde K)|}{|10.L^{n+1}(B)|}=10^{-d}|\tilde K|\geq 20^{-d}\alpha_n.$$

The point $z_0$ belongs to at least $\frac 1 2 J^{n+1}H^{n}\beta^{m-n}$ images $L_0\circ L_\d(B)$
associated to words $\d\in \{1,\dots,H\}^n$ whose last $m$ letters
$j_{n-m+1},\dots, j_n$ are pairwise different.
For these different words $\d$, the events $z \in L_0 \circ L_\d(\tilde K)+10.L^{n+1}(w_\d)$
are independant in the parameter space $\Omega_n$.
Consequently, the probability of the event~\eqref{e.badevent} is at most
$(1-(20^{-d} \alpha_n))^{\frac 1 2 J^{n+1}H^{n}\beta^{m-n}}$. From the definition of $\alpha_n$ and $m$,
this gives:
$$|\Omega_n(z)|\leq \exp\left(-{100^{-d}}J^{2(n+1)}H^{2n}\beta^{(c-2)n}\right).$$

By our choice of $c$, we have $J^2H^{2}\beta^{c-2}>1$.
Since the cardinality of $A'$ grows at most exponentially with $n$,
it follows that the measure of the set
$$\Omega_n\setminus \bigcup_{z\in A'}\Omega_n(z)$$
becomes arbitrarily small.
In particular, for $n$ large there exists $w\in \Omega_n$
such that  condition~\eqref{e.recurrent} holds.
Hence the result follows.
\end{proof}

\subsection{Reduction to standard affine horseshoes}
Any affine horseshoe contains large standard horseshoes
(Definition~\ref{d.standard}). Note that the condition~\eqref{e.almost-pesin}
can be preserved by this construction.

\begin{proposition}\label{p.reduction-standard}
Consider $f$ and an affine horseshoe $\Lambda$ with constant
linear part. Then for any $\varepsilon>0$ there exist $\Lambda'\subset \Lambda$,
a chart $\varphi\colon U\to \RR^{d_u+d_s}$ and $N$ such that
\begin{itemize}
\item $[0,1]^{d_u+d_s}\subset \varphi(U)$ and $\Lambda'$ is a standard horseshoe of $f^N$ for the chart $\varphi$,
\item $\frac 1 N h_{top}(\Lambda',f^N)>h_{top}(\Lambda,f)-\varepsilon$,
\item the diameter of $R:=\varphi^{-1}([0,1]^{d_u+d_s})$ is smaller than $\varepsilon$,
\item if  $R_1,\dots,R_\ell$ are the connected components of
$R\cap f^{-1}(R)$ that intersect $\Lambda'$, then $f^i(R_j)\cap R=\emptyset$ for each $1\leq i<N$ and $1\leq j\leq \ell$.\end{itemize}
\end{proposition}
\begin{proof}
By extraction one first reduces to the case of a horseshoe whose dynamics are conjugate to a full shift.
Indeed, by~\cite{Anosov-horseshoe}, any horseshoe is topologically conjugate to a transitive Markov subshift
$(X_0,\sigma_0)$ over a finite alphabet $\mathcal{A}_0$, hence it is enough to work with a transitive Markov subshift
and to apply the following lemma.

\begin{lemma}\label{l.reduce-full-shift}
Let $(X_0,\sigma_0)$
be a transitive Markov subshift  on a finite
alphabet $\cA_0$. Then for any $\varepsilon>0$,
there exists a compact subset $X\subset X_0$
and $k,\ell\geq 1$ such that
\begin{itemize}
\item $X=\sigma_0^k(X)$ and $\sigma_0^i(X)\cap X=\emptyset$ for $1\leq i<k$,
\item the restriction of $\sigma^k_0$ to $X$ is conjugate to the shift on $\{0,\dots,\ell\}^\ZZ$, and
\item $\frac 1 k \log (\ell)>h_{top}(X_0,\sigma_0)-\varepsilon$.
\end{itemize}
\end{lemma}
\begin{proof}
Let us fix a symbol $*\in \cA_0$ and consider for $n$ large
the collection $\cL(n,*)$ of words $(a_0,a_1,\dots,a_{n-1})$ of length $n$
such that $a_0=*$ and $(a_0,a_1,\dots,a_{n-1}, *)$
appears in $X_0$. Their number is larger than
$\exp(n.(h_{top}(X_0,\sigma_0)-\varepsilon/4))$.
Note that the set of words $w\in \cL(n,*)$
that can be written as
$w=w'w'...w'$ for some sub word $w'$
with length smaller than $n$ has cardinality smaller than
$\exp(\frac 1 2 n.(h_{top}(X_0,\sigma_0)+\varepsilon))$.
Consequently, there exists a word $w_0\in \cL(n,*)$
which does not have such a decomposition.
In particular $w_0$ appears only twice
as a sub-word of $w_0w_0$.

The number of different subwords of length $n$ in $w_0w_0$
is smaller or equal to $n$.
Consequently, there exists a collection
$\cL\subset \cL(n,*)$ of
$\exp(n.(h_{top}(X_0,\sigma_0)-\varepsilon/2))$
words which are not subwords of $w_0w_0$.

If $k\geq 1$ is a large multiple of $n$,
let $\cA$ be the collection of words of length $k$
of the form $w_1w_2w_3\dots w_{\frac k n-1}$
with $w_1=w_2=w_0$ and $w_i\in \cL$ for $i\not\in\{1,2\}$.
One labels the elements of $\cA$ by $\{0,\dots,\ell\}$
where $\ell$ is larger than
$\exp(k.(h_{top}(X_0,\sigma_0)-\varepsilon))$.

Let $X$ be the collection of sequences $(a_j)\in \cA_0^\ZZ$ such that
$(a_{ik},\dots,a_{(i+1)k-1})$ belongs to $\cA$ for each $i\in \ZZ$.
We have $X\subset X_0$ and the restriction of $\sigma_0^k$
to $X$ is conjugate to the shift on $\cA^\ZZ$.
By construction of $w_0$ and  $\cL$,
the image $\sigma_0^i(X)$ intersects $X$ if and only if $i\in k\ZZ$.
\end{proof}

From Lemma~\ref{l.reduce-full-shift}
one can now assume that the dynamics of $f$ on $\Lambda$ is conjugate to the
shift on $\{0,\dots,\ell\}^\ZZ$ by a homeomorphism $h$ and that the entropy on the subhorseshoe
$h^{-1}(\{1,\dots,\ell\}^\ZZ)$ obtained by deleting one symbol is larger than $h_{top}(\Lambda,f)-\varepsilon/2$.

Since $\Lambda$ is affine, up to a conjugacy,  we may assume that it is contained in
$\RR^{d_u+d_s}$ and that $f$ is piecewise affine in a neighborhood of $\Lambda$
with a constant linear part $A=A^u\times A^s$. By an affine change of the coordinates,
we may also assume that $A^s$ sends $B^s=[0,1]^{d_s}$ into its interior and that
$(A^u)^{-1}$ sends sends $B^u=[0,1]^{d_s}$ into its interior.
Let $p\in \Lambda$ denote the fixed point $h^{-1}(\underline 0)$ and let
$R$ be a small cube centered at $p$ that is the pre image of the standard cube $B^u\times B^s$ by
an affine map of the form $\varphi\colon z\mapsto \lambda.z+v$.
Hence $R$ is a product $R=R^u\times R^s$.

One can find $x^s\in W^s(p)$ and $x^u\in W^u(p)$ contained in $\operatorname{interior}(R)\cap \Lambda$
such that $f^{-i}(x^s),f^{i}(x^u)\not\in R$ for $i\geq 1$.
Denote by $W^s_{loc}(x^u)$ the connected component of $W^s(x^u)\cap R$ containing $x^u$:
it is a cube of the plane $p+\{0\}^{d_u}\times \RR^{d_s}$.
If $R$ has been chosen small enough, we may assume that
$f^i(W^s_{loc}(x^u))$ is disjoint from $R$ for any $i\geq 1$.
In the same way, denoting by $W^u_{loc}(x^s)$ the connected component of $W^u(x^s)\cap R$ containing $x^s$,
we may assume that
$f^{-i}(W^u_{loc}(x^s))$ is disjoint from $R$ for any $i\geq 1$.

We now fix $n\geq 1$ and consider $m$ much larger. Set $N=2n+m$.
Let $(a^u_j)=h(x^s)$ and $(a^s_j)=h(x^u)$.
To any sequence $b=(b_1,\dots,b_m)\in\{1,\dots,\ell\}^m$ we associate
the point $x(b)\in \Lambda$ such that $(c_j)=h(x(b))$ satisfies:
\begin{itemize}
\item $c_j=a^u_j$ for $j< n$,
\item $c_j=b_{j-n}$ for $n\leq i< n+m$,
\item $c_j=a^s_{j-2n-m+1}$ for $n+m \leq j$.
\end{itemize}

By construction $x(b)$ and $f^{N}(x(b)$ belong to the interior of $R$,
and the forward images of the connected component of $W^s(x(b))\cap R$ containing $x(b)$
(resp. the backward images of the connected component of $W^u(f^{N}x(b))\cap R$ containing $f^{N}(x(b))$)
are disjoint from the boundary of $R$.
Since $f$ is affine in a neighborhood of $\Lambda$, this implies that
the connected component $R(b)$ of $R\cap f^{-N}(R)$ containing $x(b)$
is a set of the form $R^u(b)\times R^s$, where $R^u(b)$ is contained in the interior of $R^u$.
Analogously, $f^{N}(R(b))$
is a set of the form $R^u\times R^s(b)$, where $R^s(b)$ is contained in the interior of $R^s$.
Moreover the iterates $f^i(R(b))$ for $1\leq i <N$ are disjoint from $R$.

The $\ell^m$ rectangles $R(b)$ (resp. $f^N(R(b))$) for $b\in \{1,\dots,\ell\}^m$ are pairwise disjoint.
Consequently, the maximal invariant set $\Lambda'$ for $f^N$ in $\bigcup_b R(b)$ is a standard horseshoe.

If $n$ is sufficiently large, then the entropy of $f^N$ on $\Lambda'$  satisfies the required lower bound:
$$\frac 1 N h_{top}(\Lambda',f^N)=\frac m N \log (\ell)
\geq h_{top}(\Lambda,f)-\varepsilon.$$
\end{proof}

\subsection{Proof of Theorem C }
We can now complete the construction of blenders for affine horseshoes with large entropy.
Consider $f$ and an affine horseshoe $\Lambda$ with constant linear part $A$
that is partially hyperbolic and satisfies condition~\eqref{e.almost-pesin}.
Denote by $\chi^u_{inf}(A)$ and $\chi^u_{max}(A)$ the smallest and the largest positive Lyapunov exponents of $A$;
by~\eqref{e.almost-pesin}, we can choose $c\in (0,1)$ such that
\begin{equation}\label{e.ck}
c.k.\chi^u_{max}(A)<\chi^u_{inf}(A),\,\,\hbox{ and}
\end{equation}
\begin{equation}\label{e.defc}
h_{top}(\Lambda,f)>\log |\det(A_{|E^u})|-\frac c 2 \chi_{max}^u(A).
\end{equation}

By Proposition~\ref{p.reduction-standard}, we can assume that there exist $N\geq 1$, a decomposition
$$\Lambda=\Lambda'\cup f(\Lambda')\cup\dots\cup f^{N-1}(\Lambda'),$$
and a chart $\varphi\colon U\to \RR^{d_u+d_s}$ such that:
\begin{itemize}
\item $\Lambda'$ is a standard horseshoe for $f^N$ in the neighborhood $R:=\varphi^{-1}([0,1]^{d_u+d_s})$,
\item $\Lambda'$ is contained in finitely many components
$R_0,\dots,R_H$ of $R\cap f^{-N}(R)$,
\item $f^i(R_j)\cap R=\emptyset$ for any $0\leq j\leq H$ and $1\leq i<N$, and
\item $\frac 1 N \log(1+H)=\frac 1 N h_{top}(\Lambda',f^N)$ is arbitrarily close to $h_{top}(\Lambda,f)$.
\end{itemize}
Consequently from~\eqref{e.defc} we have:
\begin{equation}\label{e.defc2}
\log(H)-\log |\det(A_{|E^u}^N)|>-\frac c {2}\chi^u_{max}(A^N)\geq -\frac c 2 \log |\det(A^N_{|E^{uu}})|.
\end{equation}

As explained in Section~\ref{ss.affine-recurrent},
the standard horseshoe $\Lambda'$ of $f^N$ defines an IFS inside the center space $E^c$
with affine contractions of the form $L_j\colon z\mapsto L.z+v_j$, $0\leq i\leq H$. Moreover
$J:=|\det(L)|$ coincides with $|\det(A^N_{E^c})|^{-1}$ and if we set
$\beta=|\det(A^N_{|E^{uu}})|$, then the property~\eqref{e.bound-images} holds and~\eqref{e.defc2} gives $\beta^{2-c}<(JH)^2$.

We can thus apply  Proposition~\ref{p.IFS} and find $n\geq 1$ large and a function
$w\colon \{1,\dots,H\}\to B^c$ such that, setting $m=[c.n]+1$, we have:
\begin{itemize}
\item the IFS induced by the affine contractions
$$z\mapsto L_0\circ L_\d(z)+10.L^{n+1}(w_\d), \quad\d\in \{1,\dots,H\}^n$$
satisfies the recurrent compact condition,
\item we have $w_\d=w_{\d'}$ each time the last $m$ letters of $\d$ and $\d'$ coincide.
\end{itemize}

After conjugating in the chart $\varphi$, the horseshoe $\Lambda'$ for $f^N$
is the maximal invariant set in the cube $[0,1]^{d_u+d_s}=B^u\times B^s$
for the dynamics induced by some affine maps:
$$F_j\colon B^u_j\times B^s\to B^u\times B^s_j\colon\quad\quad z\mapsto A^N(z)+v_j.$$
Consider the standard horseshoe $\Lambda''$ for $f^{(n+1).N}$ contained in $\Lambda'$ and defined by the affine maps $F_\d\circ F_0$
where $\d\in\{1,\dots,N\}^N$ and $F_\d=F_{j_1}\circ F_{j_2}\circ\dots\circ F_{j_n}$.
Each of these maps is defined on a domain $B^u_{(\d,0)}\times B^s$.
If $T_j$ denotes the contraction of $B^u$ induced by the map $F_j^{-1}$,
then $B^u_{(\d,0)}=T_0\circ T_{j_n}\circ\dots\circ T_{j_1}(B^u)$.
\bigskip

We now explain how to modify $f$ on $B^u_0\times U^s$, where $U^s$
is a small neighborhood of $B^s\subset \RR^{d_s}$.
Start by fixing a word $\underline i\in\{1,\dots,H\}^m$. Now consider the image
$B^u_{\underline i}=T_0\circ T_{i_m}\circ\dots\circ T_{i_1}(B^u)$;
it contains all the sets $B^u_{(\d,0)}$ such that the last $m$ letters of
$\d$ coincide with $\underline i$. Moreover the distance from each such set $B^u_{(\d,0)}$
to the boundary of $B^u_{\underline i}$ is greater than $\delta.\exp(-\chi^u_{max}.m.N)$,
where $\delta$ is a lower bound for the distance of the sets $B^u_j$
to the complement of $B^u$.

By construction, there exists $w_{\underline i}\in B^c$ that coincides with
all the vectors $w_\d$ associated to the words $\d$ whose last letters coincide with $\underline i$.
We modify $f$ inside
$B^u_{\underline i}\times U^s$ by a diffeomorphism that coincides
on each domain $B^u_{(\d,0)}\times U^s+(0,10.L^{(n+1)}(w_{\underline i}),0)$ with
$$(z^{uu},z^c,z^s)\mapsto f\big(z^{uu},z^c-10.L^{(n+1)}(w_{\underline i}),z^s\big).$$
Note that
$$\|10.L^{(n+1)}(w_{\underline i})\|^{1/k}\leq 10\exp(-\chi^u_{inf}.(n+1).N/k)$$
is much smaller than $\delta.\exp(-\chi^u_{max}.m.N)$, and hence than the distance between the sets $B^u_\d$ and the complement
of $B^u_{\underline i}$, from inequality~\eqref{e.ck}.
Consequently  diffeomorphism we obtain is $C^k$-close to $f$, provided $n$ has been chosen large enough.
Repeating this construction independently in each domain $B^u_{\underline i}\times U^s$,
we obtain  the diffeomorphism $g$. Note that $g$ can be chosen to preserve the volume if $f$ does.

By construction, the $(n+1).N$ consecutive iterates of $B^u_0\times U^s$ by $f$ are disjoint.
It follows that the diffeomorphism $G=g^{(n+1).N}$ has a standard horseshoe $\Lambda''_G$
whose IFS coincides with the affine contractions 
$z\mapsto L_0\circ L_\d(z)+10.L^{n+1}(w_\d)$, $\d\in \{1,\dots,H\}^n$.
By Propositions~\ref{p.criterion} and~\ref{p.standard} this horseshoe is a $d_{cs}$-stable blender.
Consequently the union of the iterates $g^i(\Lambda''_G)$, $1\leq i\leq (n+1).N$, is
a $d_{cs}$-stable blender for $g$, as required. Note also that the support of the perturbation is contained in
a small neighborhood of $R$, which has arbitrarily small diameter. \eproof


\section{Appendix: Approximation of hyperbolic measures by horseshoes}
\label{s.katok}

\def\cY{\mathcal Y}

We prove in this section the version of Katok's approximation of hyperbolic measures stated in  Theorem~\ref{t.katok}.
Recall that we have fixed a diffeomorphism $f$ in $\diff^{1+\alpha}(M)$,
an ergodic hyperbolic measure $\mu$, a constant $\delta>0$ and a neighborhood
$\cV$ of $\mu$ in the weak* topology.

\subsection{Uniformity blocks}
Our goal is to extract a subset of $M$ that generates a horseshoe with large entropy. 
In order to select enough orbit types, we will need to shadow certain pseudo-orbits for $f$ by true $f$-orbits.  Such shadowing lemmas exist for nonuniformly hyperbolic dynamics, but we need one especially adapted to our setting, allowing us to control Lyapunov exponents.  The orbits will be selected from points visiting a special set called a \emph{uniformity block}.

We first introduce the Oseledets-Pesin charts associated to an ergodic measure.
See Theorem S.3.1 in~\cite{katok-hasselblatt}.

\begin{theorem}[Pesin]\label{t=pesin} Let $f\colon M\to M$ be a $C^{1+\alpha}$ diffeomorphism preserving the ergodic probability measure $\mu$ and let $\eta>0$.
Then there exist  $Z\subset M$, measurable with $\mu(Z)=1$,
two measurable functions $r,K\colon M\to (0,1]$ and a measurable family
of invertible linear maps $C(x)\colon T_xM\to \RR^d$ satisfying the following properties.
\begin{enumerate}
\item $|\log(r(f(x)) / r(x))|<\eta$ and $|\log(K(f(x))/K(x))|<\eta$ for any $x\in Z$.

\item If $\chi_1>\dots>\chi_\ell$ are the Lyapunov exponents of $\mu$,
with multiplicities $n_1,\dots,n_\ell$, then for any $x\in Z$
there are $A_{i}(x)\in GL(n_i,\RR)$ such that
$$L=C(f(x))\cdot Df(x) \cdot C^{-1}(x) = \diag\left(A_{1}(x),\ldots, A_{\ell}(x)\right),$$
$$e^{\chi_i - \eta} < \|A_{i}^{-1}\|^{-1} \leq \|A_{i}\| < e^{\chi_i + \eta}.$$

\item In the charts $\phi_x\colon x\mapsto C(x)\circ \exp_x^{-1}$,
the maps $f_x=\phi_{f(x)}\circ f \circ \phi_x^{-1}$ are defined on $B(0,r(x))$
and satisfy $d_{C^1}(f_x, Df_x(0))<\eta$ at every $x\in Z$.

\item For any $x\in Z$ and any $y,y'\in \phi_x^{-1}(B(0,r(x)))$,
$$ d(y,y') \leq \|\phi_{x}(y) - \phi_{x}(y')\| \leq  \frac 1 {K(x)} \, d(y,y').$$
\end{enumerate}
\end{theorem}
The sets $\phi_x^{-1}(B(0,r(x)))$ are called {\em regular neighborhoods} of the points $x\in Z$.  They decay slowly exponentially in size (at rate at most $e^{-\eta}$) along the orbits.

\begin{definition}  A compact set $X\subset M$ with $\mu(X)>0$ is a {\em uniformity block for $\mu$ (with tolerance $\eta>0$)}  if there exist $Z,r,K,C$ as in Theorem~\ref{t=pesin}
with $X\subset Z$ such that $r,K,C$ are continuous on $X$.
\end{definition}
For any tolerance $\eta>0$, Theorem~\ref{t=pesin} and Lusin's theorem imply the existence of uniformity blocks
with measure arbitrarily close to $1$.

\subsection{Shadowing theorem}
The shadowing theorem we will use applies to special pseudo-orbits -- those with jumps in a uniformity block $X$
for a hyperbolic measure $\mu$.  
A sequence $(x_n)_{n\in \ZZ} \subset M$ is an {\em $\epsilon$-pseudo-orbit with jumps in a set $X\subset M$} if
$\{x_n, n\geq n_0\}$ and $\{x_n, n\leq -n_0\}$ meet $X$ for arbitrarily large $n_0$ and moreover
$$\forall n\in \ZZ,\quad d(f(x_n), x_{n+1}) > 0 \implies f(x_n), x_{n+1} \in X \text{ and }d(f(x_n),x_{n+1}) < \epsilon.$$
The following theorem uses the hyperbolicity assumption on the measure $\mu$.

\begin{theorem}\label{p=shadowing}  For any $f\in \diff^{1+\alpha}(M)$ with ergodic hyperbolic measure $\mu$, 
there exists $\kappa>0$ such that, for every
$\eta>0$ sufficiently small and every uniformity block $X$ for $\mu$ of tolerance $\eta$,
the following property holds for some constants $C_0,\epsilon_0>0$.

If $(x_n)$ is an $\epsilon$-pseudo-orbit with jumps in $X$ and $\epsilon\in (0,\epsilon_0)$,
then there exists a unique $y\in M$ whose orbit $C_0\epsilon$-shadows $(x_n)$, i.e. $d(f^n(y), x_n) < C_0\epsilon$ for all $n\in \ZZ$.
Moreover $y$ belongs to the regular neighborhood $\phi_{x}^{-1}(B(0,r(x)/2))$.

If $y,y'$ shadow two pseudo-orbits $(x_n)$, $(x'_n)$ such that
$x_n=x'_n$ for $|n|\leq N$, then
$$d(y,y')\leq C_0e^{- \kappa N}.$$\end{theorem}

We will use the classical shadowing theorem for sequences of diffeomorphisms (see, e.g. \cite{Pi}).

\begin{theorem} \label{t=sequenceshadow} For every $\kappa> 0$ there exist $\theta,\eta_0>1$ with the following property.

Let $(L_n^s)\in \GL(d_s,\RR)^\ZZ$ and $(L_n^u)\in \GL(d_u,\RR)^\ZZ$ satisfying 
$\|L_n^s\|, \|(L_n^u)^{-1}\| < e^{-\kappa}$ and for each $n$
let $L_n=\diag(L_n^u,L_n^s)$.
Let $(g_n) \in \Diff^1(\RR^d)^\ZZ$, where $d=d_u+d_s$,
be a sequence of diffeomorphisms with $d_{C^1}(g_n,L_n) < \eta_0$
and fix any $\epsilon>0$.

Then every $\epsilon$-pseudo-orbit for $(g_n)$  is $\theta\epsilon$-shadowed by a unique $(g_n)$-orbit.
\smallskip

\noindent
More precisely: for every sequence $(x_n)$ in $\RR^d$ satisfying
$\|g_n(x_n) - x_{n+1}\| < \epsilon$ for all $n\in \ZZ$, there exists a sequence $(y_n)$ in $\RR^d$ satisfying
$g_n(y_n) = y_{n+1}$ and   $\|y_n - x_n\|<\theta\epsilon$ for all $n\in \ZZ$.  This sequence is unique:
for any sequence $(z_n)$ in $\RR^d$ satisfying $g_n(z_n) = z_{n+1}$ for all $n\in \ZZ$, and for any $n_1\leq k\leq n_2$,
we have:
$$\| z_{k} - y_k\| \leq  e^{ - \kappa (k-n_1)}\|z_{n_1}- y_{n_1} \|+  e^{- \kappa (n_2-k)} \|z_{n_2}- y_{n_2}\|.$$
\end{theorem}

\begin{proof}[Proof of Theorem~\ref{p=shadowing}]
Choose a positive lower bound $\chi$ for the $|\chi_i|$.
We set $\kappa = \chi/2$ and get constants $\theta, \eta_0$ from Theorem~\ref{t=sequenceshadow}.
We choose $\eta < \min(\chi/2,\eta_0/4)$.
We consider a uniformity block $X$ for $\mu$ with tolerance $\eta$ and associated to functions $r,K,C$.
Assuming that $\eta$ is small enough, for any $C^1$ map $f_0\colon B(0,r)\to \RR^d$
and for any linear map $L\in \GL(d,\RR)$ such that $d_{C^1}(f_0, L)<\eta$,
there exists a $C^1$-map $g_0\colon \RR^d\to \RR^d$ which coincides with $f_0$ on $B(0,r/2)$
and satisfies $d_{C^1}(g_0, L)<2\eta$.
We extend in this way each diffeomorphism $f_x$ to a diffeomorphism $g_x$ of $\RR^d$
agreeing with $f_x$ on $B(0,r(x)/2)$ and such that $d_{C^1}(g_x,Df_x(0))<3\eta$.

Let $\Delta>1$ be an upper bound for $\frac 1 {K(x)}$, $x\in X$.
We choose $\epsilon_0>0$ such that $r(x)>2\theta\Delta\epsilon_0$ for each $x\in X$.
Moreover if $\epsilon_0$ is small enough, when $f(x),x'\in X$ are $\epsilon_0$-close,
the diffeomorphisms $f_{x,x'}=\phi_{x'}\circ f \circ \phi_x^{-1}$ and $f_x$
are $\eta$-close in the $C^1$-topology. In particular $f_{x,x'}$ can be extended
to a diffeomorphism  $g_{x,x'}$ of $\RR^d$
that coincides with $f_{x,x'}$ on $B(0,r(x)/2)$ so that $d_{C^1}(g_{x,x'},Df_x(0))<3\eta$.

Fix $\epsilon<\epsilon_0$ and let $(x_n)$ be a $\epsilon$-pseudo-orbit with jumps in $X$.
We consider the sequence $(g_n)$ of diffeomorphisms of $\RR^d$ defined by
$g_n=g_{x_n}$ if $x_{n+1}=f(x_n)$ and $g_{n}=g_{x_n,x_{n+1}}$ otherwise.
In the first case $g_n(0)=0$ and in the second case
$|g_n(0)|\leq \Delta \epsilon$. The pseudo-orbit $(0)$ is thus $\theta\Delta\epsilon$-shadowed by an orbit $(\bar y_n)$
of $(g_n)$.

At times $n$ such that $x_n\in X$ we have $\bar y_n\in B(0,r(x_n)/2)$ by our choice of $\epsilon_0$.
At other times, we consider the smallest interval $\{n_1,\dots, n_2\}$ of $\ZZ$  containing $n$
such that $x_{n_1},x_{n_2}\in X$.
The sequence $(\bar y_n)$ may be compared to the orbit of $(g_n)$ which coincides with $0$
for the indices $n_1,\dots, n_2-1$. Consequently,
$$\|\bar y_n\|\leq (e^{- \kappa (n-n_1)}\;+\;e^{- \kappa (n_2-n)})\theta\Delta\epsilon.$$
From the property (1) in Theorem~\ref{t=pesin} we have:
$$r(x_n)\geq \max(e^{-\eta(n-n_1)}r(x_{n_1})\;,\; e^{-\eta(n_2-n)}r(x_{n_2})),$$
which implies that $\bar y_n\in B(0,r(x_n)/2)$ still holds.
The projection $y_n=\phi_{x_n}^{-1}(\bar y_n)$
satisfies $d(y_n,x_n)\leq d(\bar y_n, 0) \leq \theta\Delta\epsilon$,
so that the sequence $(y_n)$ $C_0\epsilon$-shadows $(x_n)$ where $C_0:=\theta\Delta$.
By construction $(y_n)$ is an orbit of $f$.

If $(y'_{n})$ is another orbit that shadows a pseudo-orbit $(x'_n)$ such that
$x_n=x'_n$ for $|n|\leq N$,
the lifts $\phi_{x_n}(y_n)$ and $\phi_{x_n}(y'_n)$ may be compared.
From Theorem~\ref{t=sequenceshadow},
\begin{equation*}
\begin{split}
d(y_0,y'_0)&\leq \|\phi_{x_0}(y_0),\phi_{x_0}(y'_0))\|\\
&\leq
2e^{-\kappa N} \max(\|\phi_{x_N}(y_N),\phi_{x_N}(y'_N)\|),
\|\phi_{x_{-N}}(y_{-N}),\phi_{x_{-N}}(y'_{-N}))\|\\
&\leq 2\theta \Delta e^{-\kappa N}\leq C_0 e^{-\kappa N}.
\end{split}
\end{equation*}

This completes the proof of Theorem~\ref{p=shadowing}.
\end{proof}

\subsection{Metric entropy}
For $n\geq 1$ we define the dynamical distance
$$d_{f,n}(x,y)=\min\{d(f^k(x),f^k(y)),\; k=0,\dots,n-1\}.$$
For $\rho>0$ and  $\beta>0$, let $C_\mu(n,\rho,\beta)$ denote the
minimal cardinality of the families of $d_{f,n}$-balls of radius $\rho$
whose union $\cup_iB_{f,n}(x_i,\rho)$ has measure larger than $\beta$.

\begin{theorem}[Katok, Theorem 1.1 in~\cite{katok}] For any $\beta>0$,
$$  h(\mu, f) = 
\lim_{\rho\to 0} \varliminf_{n\to\infty} \frac{1}{n}\log C(n,\rho,\beta)
= \lim_{\rho\to 0} \varlimsup_{n\to\infty} \frac{1}{n}\log C(n,\rho,\beta).
$$
\end{theorem}

\subsection{Proof of Theorem~\ref{t.katok}}
We choose a $\gamma>0$ and a finite 
collection of continuous functions $\psi_1,\ldots, \psi_k$ on $M$
such that $\cV$ contains the set of probability measures $\nu$ satisfying:
$$
\max_{1\leq i\leq k}\left| \int_M \psi_i \, d\mu- \int_M \psi_i\, d\nu\right| < \gamma.
$$
\bigskip

\paragraph{\it The tolerance $\eta$}
Define cones $\cC^1,\dots,\cC^{\ell-1}$ of $\RR^d$ by:
$$\cC^i=\{v=(v_1,v_2)\in \RR^{d_i}\times \RR^{d-d_i} \; 
\|v_1\|>\|v_2\|\},$$
where $d(i)=n_1+\dots+n_i$.
Let  $\eta\in (0, \min\|\chi_i\|)$ small such that the shadowing theorem holds.
We will consider linear maps $D\in \GL(d,\RR)$ that are $\eta$-close to diagonal maps
$\diag(A_1,\dots,A_\ell)$ with
$$\exp(\chi_i - \eta) < \|A_{i}^{-1}\|^{-1} \leq \|A_{i}\| < \exp(\chi_i + \eta).$$
If $\eta$ is chosen sufficiently small, then for any $i\in \{1,\dots,\ell-1\}$ and for any such $D$, the
closure of the cone $\cC^i$ is mapped by $D$
inside $\cC^i\cup \{0\}$. Moreover if $(v'_1,v'_2)$ is the image of $(v_1,v_2)\in \cC^i$
by $D$, then
$$\exp(\chi_1+\delta)\|v_1\|\geq \|v'_1\|\geq \exp(\chi_i-\delta)\|v_1\|.$$
Symmetrically, if $(v'_1,v'_2)\in \RR^d\setminus \cC^1$ is the image of $(v_1,v_2)$
by $D$ then
$$\exp(\chi_{i+1}+\delta)\|v_2\|\geq \|v'_2\|\geq \exp(\chi_{d}-\delta)\|v_2\|.$$

We then choose a uniformity block for $\mu$ with tolerance $\eta$.
\medskip

\paragraph{\it The separation scale $\rho$}
Choose $\xi>0$  such that
$$
\xi<  \frac{\delta}{h(\mu, f) + 4} \quad \left(<\frac\delta4\right).
$$
Next choose a separation scale $\rho\in (0,\delta)$ such that
$$d(x,y)<\rho \implies \max_{1\leq i\leq k} |\psi_i(x) - \psi_i(y)| < \gamma/2,$$
and a separation time $N_0\in \NN$ such that $n\geq N_0$ implies
$$
 \frac{1}{n}\log C(n,\rho,\mu(X)/2) > h(\mu, f) -\xi.
$$
For any time $n\geq N_0$,
in any subset of $X$ of $\mu$-measure at least $\mu(X)/2$,
the maximal $(n,\rho)$-separated sets
contain at least $\exp(n(h(\mu, f) - \xi))$.
\bigskip

\paragraph{\it The shadowing scale $\epsilon$ and the common return time $N$}  We choose $\epsilon\in (0,\epsilon_0)$ so that $C_0\epsilon  < \rho/2$, where $C_0,\epsilon_0$ are given by Proposition~\ref{p=shadowing} for the uniformity block $X$.  

\begin{lemma}
There are $N\geq N_0$, an $\epsilon/2$-ball $B$ centered at a point in $X$
and a set $Y\subset B\cap X$ such that:
\begin{itemize}
\item the points of $Y$ are $\rho$-separated in the distance $d_{f,N}$,
\item $f^N(Y)\subset B\cap X$,
\item $\left|\frac{1}{N}\sum_{j=1}^{N} \psi_i(f^j(x)) - \int_M \psi_i\, d\mu \right| < \gamma/2$
for each $y\in Y$ and $i\in\{1,\dots,k\}$,
\item the cardinality of $Y$ is larger than $\exp({N(h(\mu,f) - \delta)})$.
\end{itemize}
\end{lemma}
\begin{proof}
Cover $X$ by finitely many $\epsilon/2$-balls centered at points $x_1,\ldots,x_t\in X$. For $m\geq N_0$, let
$$
X_m^{0} =\left\{x\in X : \exists n, i, \text{ s.t. }
m\leq n <(1+\xi) m, \; 1\leq i\leq t, \;  \hbox{and } x, f^n(x) \in  B(x_i,\epsilon/2)\right\},
$$
$$
\text{ and }X_m  =\left\{ x\in X_m^0: \sup_{n\geq m}\, \max_{1\leq i\leq k} \left|\frac{1}{n}\sum_{j=1}^{n} \psi_i(f^j(x)) - \int_M \psi_i\, d\mu \right| < \gamma/2.\right\}.
$$
The Birkhoff Ergodic Theorem implies that $\mu(X\setminus X_m)\to 0$ as $m\to \infty$.
Fix $m>  \max(N_0,\xi^{-1}\log t)$ such that $\mu(X_m)>\mu(X)/2$, and
let $E_m$ be a maximal $(m,\rho)$-separated set in $X_m$.  By our choice of $N_0$, we have
$$
\# E_m \geq  \exp({m(h(\mu,f) - \xi)}).
$$

For $n\in \{m,(1+\xi)m-1\}$, let
$$
V_n = \{x\in E_m :  x, f^n(x)\in  B(x_i,\epsilon/2),\hbox{ for some } i\in\{1,\dots t\}\},
$$
and let $N$ be a value of $n$ maximizing $\# V_n$.
Then
$$
\# V_N \geq \frac{\# E_m}{\xi m} \geq \frac{e^{m(h(\mu,f) - \xi)}}{ \xi m} > \exp({m(h(\mu,f)- 2\xi)}).
$$

Next choose $i\in\{1,\dots,t\}$ such that $B(x_i,\epsilon/2)\cap V_N$ has maximal cardinality,  and let
$B = B(x_i, \epsilon/2)$ and
$Y= B(x_i, \epsilon/2)\cap V_N$.  
 Then from the choice of $\xi$ and since $m>\xi^{-1}\log(t)$, we get the required estimate
$$
\#Y \geq \frac{\# V_N}{t} \geq  \frac{1}{t}\exp({m(h(\mu,f) - 2\xi)})>\exp(N(h(\mu,f)-\delta)).
$$
\end{proof}
\medskip

\paragraph{\it The set $\Lambda$}
Consider the shift $\cY=Y^\ZZ$ over the alphabet $Y$. Since the diameter of $B$ is smaller than $\epsilon$,
each $(y_{n})\in \cY$ is associated to the
$\epsilon$-pseudo-orbit:
$$\ldots, y_{0}, \ldots, f^{N-1}(y_{0}), y_{1}, \ldots,f^{N-1}(y_{1}),\ldots,$$
which is obtained by concatenating the orbit segments of length $N$
starting at the points $y_n$. Theorem~\ref{p=shadowing} implies that for each such pseudo-orbit there is a unique point
$\pi((y_n))$ whose orbit shadows. Consequently, the union $\Lambda_0$ of all the orbits that shadow the elements of $\cY$ is an
$f^N$-invariant compact set. By the last property of the shadowing theorem,
the map $\pi\colon \cY\to \Lambda_0$ is continuous
(for the natural product topology on $\cY=Y^\ZZ$).
Hence $\pi$ conjugates the shift
to $f^N$.

Since for distinct $y, y'\in Y$, there exists $j\in[0,N-1]$ such that $d(f^{j}(y), f^j(y')) > \rho$,
and since $C_0\epsilon < \rho/2$, it follows that if $x$ and $x'$ $C_0\epsilon$-shadow distinct elements of $\cY$, then $x$ and $x'$ are distinct as well.  This implies that $\pi$ is a homeomorphism.
We define $\Lambda=\Lambda_0\cup f(\Lambda_0)\cup\dots\cup f^{N-1}(\Lambda_0)$.
It is a transitive $f$-invariant compact set.
From this we will obtain the entropy estimate (conclusion (2) of Theorem~\ref{t.katok}):
$$
h(\Lambda, f) = \frac{1}{N} h(\Lambda_0, f^N) >  h(\mu,f) -  \delta.
$$
\medskip

From the construction of $Y$, for any $y_i\in Y$, we have 
$$\max_{1\leq i\leq k} \left|\frac{1}{N}\sum_{j=1}^{N} \psi_i(f^j(y_i)) - \int_M \psi_i\, d\mu \right| < \gamma/2.$$
By our choice of $\rho$, if  for some $x\in X$, we have $\max_{0<j \leq N} d(f^j(x), f^j(y_i)) < \rho$,
then 
$$\max_{1\leq i\leq k} \left| \frac{1}{N}\sum_{j=1}^{N} \psi_i(f^j(y_i)) -   \frac{1}{N}\sum_{j=1}^{N} \psi_i(f^j(x))\right| < \gamma/2.$$ 
Since $C_0\epsilon < \rho$,  it follows that if a point $x$ $C_0\epsilon$-shadows a pseudo-orbit in 
$\cY$, then 
$$\varlimsup_{n\to\infty} \max_{1\leq i\leq k} \left|\frac{1}{n}\sum_{j=0}^{n-1} \psi_i(f^j(x)) - \int_M \psi_i\, d\mu \right| < \gamma.$$
From this conclusion (3) follows.

By construction $\Lambda$ is contained in the $C_0\epsilon$-neighborhood of the support of $\mu$.
Assuming that $\cV$ has been chosen small enough, $\Lambda$ intersects any $\delta$-ball centered at points
in the support of $\mu$, hence $\Lambda$ is $\delta$-close to the support of $\mu$ in the Hausdorff distance,
which gives conclusion (1).
\medskip

By the shadowing theorem, any orbit in $\Lambda_0$ stays in the regular neighborhoods of the pseudo-orbit
it shadows. Moreover $\Lambda_0$ is contained in the regular neighborhood of 
the center $x$ of $B$ and is preserved by $f^N$.
By properties (2) and (3) in Theorem~\ref{t=pesin}
and by the choice of the tolerance constant, it follows that after conjugacy by the chart $\phi_x$,
the derivative of $f^N$ preserves the cones $\cC^1,\dots,\cC^{\ell-1}$. The cone field criterion
implies that $\Lambda$ has a dominated splitting as in condition (4):
$$T_\Lambda M=E_1\oplus\dots\oplus E_\ell,\quad \text{ with } \dim(E_i)=n_i.$$
By definition and invariance of the cones, the growth of the iterates of any vector in $E_i$
is given by its second coordinate in $\RR^d=\RR^{d(i-1)}\times \RR^{n_i}\times \RR^{d-d(i)}$.
Consequently for $v$ in $E_i$ and $n\geq 1$ large enough we obtain condition (5):
$$\exp((\chi_i - \delta)n ) \leq \|Df^n(v)\| \leq \exp((\chi_i + \delta)n).$$
Since the exponents $\chi_i$ are all different from zero, the set $\Lambda$ is hyperbolic.
\medskip

We can reduce to the case where $\Lambda$ is a horseshoe by applying the following proposition.

\begin{proposition}
Let $K$ be a hyperbolic set for a $C^1$-diffeomorphism $f$.
Then for any $\delta>0$, there exists a horseshoe $\Lambda$ that is $\delta$-close to $K$
in the Hausdorff topology and satisfies
$h_{top}(\Lambda,f)\geq h_{top}(K,f)-\delta$.
\end{proposition}
\begin{proof}
The set $\Lambda$ will be obtained through the shadowing lemma
after fixing a collection of pseudo-orbits. This construction ensures that $\Lambda$
is semi-conjugate to a subshift. In order to prove that $\Lambda$ is totally disconnected
we need to choose these pseudo-orbits so that this semi-conjugacy is a conjugacy.
This uses an argument similar to the proof of Lemma~\ref{l.reduce-full-shift}.

Let $h=h_{top}(K,f)$.
The shadowing lemma associates arbitrarily small constants $\epsilon, \rho>0$
such that any $\epsilon$-pseudo-orbit contained in $K$ is $\rho$-shadowed by an orbit of $f$.
Arguing as previously, there exists $N\geq 1$ arbitrarily large and a point $x\in K$
such that $K$ contains a set $\cL$ of sequences $L=(x_1,\dots,x_{N-1})$ satisfying:
\begin{itemize}
\item $\cL$ contains at least $\exp(N(h-\delta/4))$ elements,
\item each sequence $(x,L,x)=(x,x_1,\dots,x_{N-1},x)$ is a $\epsilon$-pseudo-orbit, and
\item any two such sequences $L=(x_i)$, $L'=(x'_i)$ are $4\rho$-separated:
there exists $1\leq i<N$ such that $d(x_i,x'_i)>\rho$.
\end{itemize}

We may assume that there is $L_0=(x_1^0,\dots,x_{N-1}^0)$ in $\cL$ such that the $N$-periodic pseudo-orbit
$(x,L_0)=(x,x^0_1,\dots,x^0_{N-1})$ is shadowed by an orbit whose minimal period is exactly $N$.
Indeed the collection of periodic orbits which have same period and shadow different $(x,L)$, $L\in \cL$ are $\rho$-separated.
Hence the number of those whose minimal period is smaller than $N/2$
is at most $\exp((h-\delta/4)N/2)$. This is less than the total number of elements of $\cL$.

Fix such a sequence $L_0$. By construction
the only sub-intervals of length $N-1$ of $(L_0,x,L_0)=(x^0_1,\dots,x^0_{N-1},x,x^0_1,\dots,x^0_{N-1})$
that are $2\rho$-shadowed by $L_0$ are the initial and the final sub-intervals.

The number of sequences $L\in \cL$ that $2\rho$-shadow
a sub-interval of length $N-1$ of $(L_0,x,L_0)$
is smaller than $N$, by the separation assumption.
Consequently, one can find a subset $\cL'\subset \cL$
with $\#\cL'\geq \exp(N(h-\delta/2))$
such that the $L\in \cL'$ do not $\rho$-shadow any sub-interval of length $N-1$ of
$(L_0,x,L_0)$.

We then choose $\ell$ large and consider
pseudo-orbits of the form
$$\dots,x,L(0),x,L(1),x, L(2),x,\dots$$
such that $L(i)=L_0$ when $i=0$ or $1 \;\mod\; \ell$
and $L(i)\in \cL'$ otherwise.
Each of these sequences is shadowed by the orbit of a unique point.
This collection of orbits is a compact set $\Lambda_0$ invariant by
$f^N\ell$ and conjugate to the shift on $(\cL')^{N(\ell-2)}$.

By the choice of $L_0$ and $\cL'$, two different pseudo-orbits are $2\rho$-separated
and are $\rho$-shadowed by disjoint orbits.
Consequently, $\Lambda_0$ is disjoint from its $N\ell-1$ first iterates.
The invariant compact set $\Lambda=\Lambda_0\cup\dots\cup f^{N\ell-1}(\Lambda_0)$
is a hyperbolic set with entropy larger than $h-\delta$ and conjugate to a transitive subshift of finite type.
By \cite{Anosov-horseshoe}, any hyperbolic set conjugate to a subshift is a horseshoe and so
$\Lambda$ is a horseshoe. By construction it is arbitrarily
close to $K$ in the Hausdorff topology.
\end{proof}

The proof of Theorem~\ref{t.katok} is now complete.

\bigskip
\noindent
{\bf Acknowledgements.}\\
We thank Carlos Gustavo Moreira, Xiao-Chuan Liu and Pierre-Antoine Guih\'eneuf for helpful conversations and comments.

\end{document}